\documentclass[11pt,a4paper]{amsart}

 \usepackage{amsfonts,amsmath,amscd,amssymb,amsbsy,amsthm,amstext,amsopn}
 \usepackage{fullpage,mathrsfs,subfigure}
 \usepackage{pstricks,pst-node,pst-text,pst-tree}
 \usepackage[all,dvips,arc,curve,color,frame]{xy}
 \usepackage{epsfig,epic}

 \numberwithin{equation}{section}
 \newtheorem{theorem}{Theorem}[section]
 \newtheorem{proposition}[theorem]{Proposition}
 \newtheorem{lemma}[theorem]{Lemma}
 \newtheorem{claim}[theorem]{Claim}
 \newtheorem{corollary}[theorem]{Corollary}
 \newtheorem{conjecture}[theorem]{Conjecture}
 \theoremstyle{definition}
 
 \newtheorem{example}[theorem]{Example}
 
 \newtheorem{remark}[theorem]{Remark}

 \newtheorem*{acknowledgements}{Acknowledgements}

 \theoremstyle{remark}
 \newtheorem{exercise}[theorem]{Exercise}
 \newtheorem*{sketch}{Sketch proof}
 \newtheorem*{roadmap}{Sketch proof}
 \newtheorem*{roadmap2}{Sketch proof (continued)}

 \newcommand{\one}{\ensuremath{(\mathrm{i})}}
 \newcommand{\two}{\ensuremath{(\mathrm{ii})}}
 \newcommand{\three}{\ensuremath{(\mathrm{iii})}}
 \newcommand{\kk}{\ensuremath{\Bbbk}}
 \newcommand{\CC}{\ensuremath{\mathbb{C}}}
 \newcommand{\QQ}{\ensuremath{\mathbb{Q}}}
 \newcommand{\NN}{\ensuremath{\mathbb{N}}}
 \newcommand{\ZZ}{\ensuremath{\mathbb{Z}}}

 \DeclareMathOperator{\Circuit}{Cir}
 \DeclareMathOperator{\coh}{coh}
 \DeclareMathOperator{\Cone}{cone}
 \DeclareMathOperator{\End}{End}
 \DeclareMathOperator{\Ext}{Ext}
 \DeclareMathOperator{\Gale}{Gale}
 \DeclareMathOperator{\GL}{GL}
 \DeclareMathOperator{\Hom}{Hom}
 \DeclareMathOperator{\Ker}{Ker}
 \DeclareMathOperator{\Irr}{Irr}  
  
 \DeclareMathOperator{\Lderived}{\mathbf{L}\!}
 
 \DeclareMathOperator{\Pic}{Pic} 
 \DeclareMathOperator{\Proj}{Proj}
 \DeclareMathOperator{\Rderived}{\mathbf{R}\!}
 \DeclareMathOperator{\SL}{SL}
 
 \DeclareMathOperator{\Sp}{Sp}

 \DeclareMathOperator{\Spec}{Spec}
 \DeclareMathOperator{\Sym}{Sym}

 \DeclareMathOperator{\Wt}{Wt}
 \DeclareMathOperator{\codim}{codim}
 \DeclareMathOperator{\conv}{conv}
 \DeclareMathOperator{\diag}{diag}
 \DeclareMathOperator{\face}{face}
 \DeclareMathOperator{\head}{h}
 \DeclareMathOperator{\homdim}{homdim}
 \DeclareMathOperator{\inc}{inc}
 \DeclareMathOperator{\id}{id}

 \DeclareMathOperator{\opp}{op}
 \DeclareMathOperator{\pic}{pic}
 \DeclareMathOperator{\rad}{rad}
 \DeclareMathOperator{\rank}{rank}
 \DeclareMathOperator{\rep}{rep}
 \DeclareMathOperator{\supp}{supp}
 \DeclareMathOperator{\tail}{t}

 \newcommand{\ltensor}{\overset{\mathbf{L}}{\otimes}} 
 \newcommand{\Gcoh}{\ensuremath{G}\operatorname{-coh}}
 \newcommand{\GExt}{\ensuremath{G}\operatorname{-Ext}}
 \newcommand{\GHom}{\ensuremath{G}\operatorname{-Hom}}
 \newcommand{\ghilb}{\ensuremath{G}\operatorname{-Hilb}}
 \newcommand{\hilbg}{\operatorname{Hilb}\ensuremath{^G}}
 \newcommand{\Irel}{\ensuremath{I_{\varrho}}}
 \newcommand{\rel}{\ensuremath{\varrho}}
 \newcommand{\modA}{\operatorname{mod}(\ensuremath{A})}
 \newcommand{\modAop}{\operatorname{mod}(\ensuremath{A^{\opp}})}
 
 \newcommand{\ModAop}{\operatorname{Mod}(\ensuremath{A^{\opp}})}

 \newcommand{\git}{\ensuremath{/\!\!/\!}}

 \renewcommand{\div}{\operatorname{div}}

\begin{document}
 \bibliographystyle{plain} 

 \title[Quiver representations in toric geometry]{Quiver
 representations in toric geometry} \author{Alastair Craw}
 \dedicatory{To Miles Reid on his 60th birthday} \address{Department
 of Mathematics,Glasgow University, Glasgow G12 8QW, Scotland}
 \email{craw@maths.gla.ac.uk}

 \begin{abstract}
   This article is based on my lecture notes from summer schools at
   the Universities of Utah (June 2007) and Warwick (September 2007).
   We provide an introduction to explicit methods in the study of
   moduli spaces of quiver representations and derived categories
   arising in toric geometry. The first main goal is to present the
   noncommutative geometric approach to semiprojective toric varieties
   via quivers. To achieve this, we use geometric invariant theory to
   construct both semiprojective toric varieties and moduli spaces of
   quiver representations.  The second main goal builds on the first
   by presenting an introduction to explicit methods in derived
   categories of coherent sheaves in toric geometry. We recall the
   notion of tilting bundles with examples, and describe the McKay
   correspondence as a derived equivalence in some detail following
   Bridgeland, King and Reid. We also describe extensions of their
   result beyond the $G$-Hilbert scheme to other fine moduli spaces of
   bound quiver representations.
     \end{abstract}

 \maketitle
 \tableofcontents

 \medskip
 
 \section{Introduction}
 This article provides an introduction to explicit methods in the
 study of moduli spaces of quiver representations and derived
 categories arising in toric geometry. Our guiding principle is to
 understand when a fine moduli space description can be given for a
 class of toric varieties so that the resulting universal family gives
 insight into geometric properties of the toric variety itself. For
 instance, if the universal family defines a collection of line
 bundles on the variety, do the classes of these bundles freely
 generate the Picard group of line bundles, the Grothendieck group of
 vector bundles, or the bounded derived category of coherent sheaves?
 We provide the necessary background from toric geometry,
 representations of quivers and derived categories in order to study
 this question. We also go beyond the toric category in several
 places, notably in presenting a detailed description of the work of
 Bridgeland, King and Reid that establishes the McKay correspondence
 as an equivalence of derived categories in dimension three.
 
 To begin, we provide a brief and novel introduction to toric
 varieties by describing the class of semiprojective toric varieties.
 A brief tour of geometric invariant theory is given in
 Section~\ref{sec:GIT}, where results for diagonalisable group actions
 on affine toric varieties are listed, including a full description of
 finite group quotients and the geometric interpretation of $\Proj$ of
 a $\ZZ$-graded ring. This paves the way for the description of
 semiprojective toric varieties in Section~\ref{sec:toric}. We do not
 assume that toric varieties are normal, but we do restrict attention
 only to toric varieties that are projective over an affine toric
 variety. In the normal case, we describe the relation to polyhedra
 and fans, and describe the canonical GIT description of Cox, giving
 worked examples for the first Hirzebruch surface $\mathbb{F}_1$ and
 the minimal resolution of the $A_2$-singularity.
 
 The second part of the article introduces the construction of normal
 semiprojective toric varieties as fine moduli spaces of bound quiver
 representations by Craw--Smith~\cite{CrawSmith}.
 Section~\ref{sec:quivers} presents the abelian categories of quiver
 representations and bound quiver representations with examples,
 before recalling the construction of fine moduli spaces of quiver
 representations by King~\cite{King1}. We restrict the dimension
 vector to ensure that the resulting moduli space is toric, or at
 worst (in the case of bound quivers) where each irreducible component
 is a semiprojective toric variety.  Section~\ref{sec:noncommutative}
 uses these moduli spaces to extend the classical notion of the linear
 series $\vert L\vert $ of a single basepoint-free line bundle $L$,
 obtaining the multilinear series $\vert \mathcal{L}\vert $ of a
 collection of basepoint-free line bundles
 $\mathcal{L}=(\mathscr{O}_X,L_1,\dots, L_r)$ on a normal
 semiprojective toric variety $X$. Under nice conditions, the image
 under the natural morphism $\varphi_{\vert \mathcal{L}\vert}\colon
 X\to \vert \mathcal{L}\vert$ is shown to represent a functor, and
 hence carries a tautological bundle that encodes a description of $X$ as
 a fine moduli space of bound quiver representations.
 
 In the final part, comprising
 Sections~\ref{sec:tilting}-\ref{sec:VGIT}, we use quiver
 representations to study derived categories in toric geometry. Fine
 moduli space constructions and derived category questions and have
 been linked closely since the pioneering work on abelian varieties by
 Mukai~\cite{Mukai}. This relationship blossomed with the development
 of Fourier--Mukai transforms by Bondal, Orlov, Bridgeland and others,
 culminating in the striking moduli space construction of threefold
 flops by Bridgeland~\cite{Bridgeland2}. Huybrechts~\cite{Huybrechts}
 provides an excellent introduction to Fourier--Mukai transforms but,
 as Huybrechts remarks in his preface: ``\emph{questions related to
 representation theory of, e.g., quivers, or to modules over
 (noncommutative) rings, have not been touched upon.}''.  The final
 part of this article introduces some of these ideas, focusing largely
 (but not entirely) on the toric case.  Section~\ref{sec:tilting}
 introduces the link between derived categories and moduli space
 descriptions for semiprojective toric varieties that goes back to
 King's work on tilting bundles in toric geometry~\cite{King2}. We
 present several of King's results, and describe briefly the extension
 to smooth toric DM stacks by Kawamata~\cite{Kawamata5} and, more
 recently, Borisov--Hua~\cite{BorisovHua}. Section~\ref{sec:McKay}
 provides a detailed look at the proof of the celebrated theorem of
 Bridgeland--King--Reid~\cite{BKR} that establishes the derived McKay
 correspondence for the $G$-Hilbert scheme in dimension three. Here we
 gain nothing by restricting to the toric case, so we do not do so. In
 Section~\ref{sec:VGIT}, we conclude by moving beyond the $G$-Hilbert
 scheme to other moduli spaces of quiver representations for which the
 universal family determines explicitly the bounded derived category of
 coherent sheaves, describing briefly the work of
 Craw--Ishii~\cite{CrawIshii}.

 \begin{acknowledgements}
   Thanks to Tarig Abdel Gadir, Chris Brav, Akira Ishii and Greg Smith
   for comments and useful discussions. Thanks also to the participants of
   the graduate summer schools in Utah and Warwick, several of
   whom have since made valuable contributions in this area (see, for
   example, \cite{BSW, BoerQuinteroVelez, BorisovHua}). Special thanks
   to the organisers of both schools, and in particular to Miles Reid
   for introducing me to the $G$-Hilbert scheme a decade ago. It is
   with great pleasure that I dedicate this article to him on the
   occasion of his 60th birthday (conference).
 \end{acknowledgements}

 \section{GIT for diagonalisable group actions}
 \label{sec:GIT}
 This section reviews some of the basic elements in the construction
 of orbit spaces in algebraic geometry. Our primary examples are
 provided by quotients of affine space by a finite group, and
 the projective spectrum of a graded ring. Throughout these notes
 we write $\kk$ for an algebraically closed field of characteristic
 zero.

 \subsection{Affine quotients by finite groups}
 Let $R$ be a finitely generated integral $\kk$-algebra.  For any set
 of $\kk$-algebra generators $f_1,\dots, f_n\in R$, write $\psi\colon
 \kk[x_1,\dots, x_n]\rightarrow R$ for the surjective $\kk$-algebra
 homomorphism obtained by sending $x_i$ to $f_i$ for $1\leq i\leq n$.
 The prime ideal $\Ker(\psi)$ that records the relations between the
 generators $f_1,\dots, f_n$ cuts out the affine variety 
 \[
 \Spec(R)=\big\{p\in \mathbb{A}^n_\kk : f(p)=0 \text{ for all }f\in
 \Ker(\psi)\big\}
 \]
 with coordinate ring $R$. An affine variety need not be presented as
 a subvariety of affine space, though one may always choose an
 embedding. A different choice of $\kk$-algebra generators for $R$
 provides an alternative embedding of $X$ into affine space.

 \begin{example}
 \label{ex:GLV}
 Let $V$ be a $\kk$-vector space of dimension $n$ with dual space
 $V^*$. The symmetric algebra $R=\bigoplus_{k\geq 0} \Sym^k(V^*)$ is
 isomorphic to the polynomial ring in $n$ variables, and we identify
 $V$ with affine space $\mathbb{A}^n_\kk = \Spec(R)$. Similarly,
 we regard $\GL(V)=\GL(n,\kk)$ as the affine variety whose coordinate
 ring is the localisation of $\kk[x_{i,j} : 1\leq i,j\leq n]$ at the
 determinant function.
 \end{example}

 \begin{example}
 \label{ex:third11}
 Let $G$ denote the cyclic group of order three acting on
 $V=\mathbb{A}^2_\kk$ with generator the diagonal matrix
 $\diag(\omega, \omega)$ for $\omega$ a primitive third root of unity.
 The dual action of $G$ on $V^*$ extends to an action on the
 polynomial ring $\kk[x,y] = \bigoplus_{k\geq 0} \Sym^k(V^*)$, and a
 monomial is invariant under the action of $G$ precisely when it has
 total degree divisible by three. It follows that the algebra of
 $G$-invariant functions is
 \[
 \kk[x,y]^G\cong \kk[x^3,x^2y,xy^2,y^3].
 \]
 The $\kk$-algebra homomorphism $\psi \colon \kk[y_1,y_2,y_3,y_4]\to
 \kk[x,y]^G$ determined by setting $\psi(y_1)=x^3$, $\psi(y_2)=x^2y$,
 $\psi(y_3)=xy^2$ and $\psi(y_4)=y^3$ has kernel
 \begin{equation}
 \label{eqn:twistedcubic}
 \Ker(\psi)=\big{(}y_1y_3-y_2^2, y_1y_4-y_2y_3, y_2y_4-y_3^2\big{)}.
 \end{equation}
 Thus, $\Spec\big(\kk[x,y]^G\big)$ is the singular surface in
 $\mathbb{A}^4_\kk$ cut out by the ideal $\Ker(\psi)$.
 \end{example}

 \begin{example}
 \label{ex:714}
 Let $G$ denote the cyclic group of order 7 that acts on
 $\mathbb{A}^2_\kk$ with generator the diagonal matrix
 $\diag(\epsilon, \epsilon^4)$, for $\epsilon$ a primitive 7th root of
 unity. The algebra of $G$-invariants $R = \kk[x,y]^G$ is generated minimally
 by the monomials $x^{7},x^{3}y,x^{2}y^{3},xy^{5},y^{7}$,
 and the induced surjective map of $\kk$-algebras $\varphi\colon
 \kk[y_1,\dots,y_5]\rightarrow R$ determines the affine subvariety of
 $\mathbb{A}^5_\kk$ cut out by the equations
 \begin{equation}
 \label{eqn:eqns}
 \rank\left(\begin{array}{cccc} y_{0} & y_{1} & y_{2} & y_{3} \\
                               y_{1}^{2} & y_{2} & y_{3} & y_{4} \end{array}\right) \leq 1.
 \end{equation}
 This rank conditions means simply that the equations are the
 binomials obtained as the $2\times 2$ minors of this matrix.
 \end{example}

 This example can be generalised significantly. For now we observe
 only that for any finite subgroup $G$ of $\GL(n,\kk)$, the set of
 $G$-orbits in $\mathbb{A}^n_\kk$ admits the structure of an affine variety.

 \begin{proposition}
 \label{prop:finiteGIT}
 For $R=\kk[x_1,\dots,x_n]$ and $G\subset \GL(n,\kk)$ a finite subgroup, we have:
 \begin{enumerate}
 \item[\one] the invariant ring $R^G$
 is a finitely generated integral $\kk$-algebra; and 
 \item[\two] the morphism $ \pi \colon \mathbb{A}^n_\kk\rightarrow
 \Spec(R^G)$ induced by the inclusion $\iota \colon R^G\hookrightarrow
 R$ is surjective and closed. Moreover, for all points $p,p^\prime \in
 \mathbb{A}^n_\kk$,
 \begin{equation}
 \label{eqn:Gorbits}
 \pi(p)=\pi(p^\prime) \iff G\cdot p = G\cdot p^\prime.
 \end{equation}
 \end{enumerate}
 In particular, the \emph{affine quotient} $\mathbb{A}^n_\kk/G :=\Spec
 \big(R^G\big)$ parametrise all $G$-orbits in $\mathbb{A}^n_\kk$.
 \end{proposition}

 \begin{proof}
   The map $\iota \colon R^G\hookrightarrow R$ is integral, so $R$ is
   finitely-generated as an $R^G$-module.  Since $R$ is a
   finitely-generated $\kk$-algebra, part (1) follows from the
   result of Artin-Tate (see \cite[Prop~7.8]{AtiyahMacdonald}). The
   going-up theorem \cite[Ex~5.10]{AtiyahMacdonald} implies that $\pi$
   is a surjective and closed map, and, moreover, that any pair of
   closed, disjoint and $G$-invariant subsets $Z_1, Z_2\subset
   \mathbb{A}^n_\kk$ satisfy $\pi(Z_1)\cap \pi(Z_2)=\emptyset$. This
   gives the statement from \eqref{eqn:Gorbits}.
 \end{proof}

 \begin{exercise}
 \label{ex:r111}
 For $r\in \NN$ and $a_1,\dots, a_n\in\NN$ satisfying $a_i<r$ for
 $1\leq i\leq n$, consider the action of the cyclic group $G=\ZZ/r$ on
 $\mathbb{A}^n_\kk$ with generator the diagonal matrix
 $\diag(\epsilon^{a_1}, \dots, \epsilon^{a_n})$ for $\epsilon$ a
 primitive $r$th root of unity. This is the \emph{action of type
   $\frac{1}{r}(a_1,\dots, a_n)$}, and the variety
 $\mathbb{A}^n_\kk/G$ is the \emph{cyclic quotient singularity of type
   $\frac{1}{r}(a_1,\dots, a_n)$}. Generalise Example~\ref{ex:714} by
 writing down the defining equations for the quotient singularity of
 type $\frac{1}{r}(1,1,\dots,1)$.
\end{exercise}

 \begin{remark}
 \label{rem:verbatim}
   For $X=\Spec(R)$, choose an embedding in
   $\mathbb{A}^n_\kk$ and suppose that a finite subgroup $G\subset
   \GL(n,\kk)$ acts on $X$ by restriction. The proof of
   Proposition~\ref{prop:finiteGIT} applies verbatim to show that the
   \emph{affine quotient} $X/G:=\Spec\big(R^G\big)$ parametrises all
   $G$-orbits in $X$.
 \end{remark}

 \subsection{Diagonalisable algebraic groups} 
 \label{sec:algebraicactions}
 In order to study further group actions on algebraic varieties we
 first establish the actions that are of interest to us. 
 
 An affine algebraic group is an affine variety $G$ over $\kk$
 equipped with three morphisms, namely \emph{multiplication} $\nu\colon
 G\times G\to G$; an \emph{identity} $e\colon \Spec(\kk)\to G$; and an
 \emph{inverse} $\iota\colon G\to G$, where:
 \begin{enumerate}
 \item[\one] $\nu\circ(\id_{G}\times \nu)=\nu\circ(\nu\times
 \id_G) \colon G\times G\times G\to G$;
 \item[\two] $\id_G=\nu\circ(e\times \id_G) = \nu\circ(\id_G\times e) \colon G\to G$;
 \item[\three] $\nu \circ(\id_G\times \iota)\circ \Delta = \nu \circ(\iota\times \id_G)\circ \Delta = e\circ \pi\colon G\to G$.
 \end{enumerate}
 Here, $\id_G\colon G\to G$ is the identity morphism, $\Delta\colon
 G\to G\times G$ is the diagonal embedding, and $\pi\colon G\to
 \Spec(\kk)$ is the natural map.

 \begin{exercise}
   Draw commutative diagrams expressing each condition and convince
   yourself that these conditions are simply the axioms for a group
   stated in the category of affine varieties.
 \end{exercise}
 
 Let $\Lambda$ be a finitely generated abelian group with group
 algebra $\kk[\Lambda]$. The group generators of $\Lambda$ determine
 finitely many $\kk$-algebra generators of $\kk[\Lambda]$, and we
 write $G_\Lambda:=\Spec\big(\kk[\Lambda]\big)$. The $\kk$-algebra
 homomorphisms $\nu^*\colon \kk[\Lambda]\to \kk[\Lambda]\otimes_\kk
 \kk[\Lambda]$, $e^*\colon \kk[\Lambda]\to \kk$ and $\iota^*\colon
 \kk[\Lambda]\to\kk[\Lambda]$ defined by setting
 $\nu^*(\lambda)=\lambda\otimes \lambda$, $e^*(\lambda)=1$ and
 $\iota^*(\lambda)=\lambda^{-1}$ for $\lambda\in \Lambda$ induce
 morphisms $\nu\colon G_\Lambda\times G_\Lambda\to G_\Lambda$,
 $e\colon \Spec(\kk)\to G_\Lambda$ and $\iota\colon G_\Lambda\to
 G_\Lambda$ that make $G_\Lambda$ an affine algebraic group. An affine
 algebraic group is \emph{diagonalisable} if it is isomorphic to
 $G_\Lambda$ for some finitely generated abelian group $\Lambda$.

 \begin{proposition}
   The assignment $\Lambda\mapsto G_\Lambda$ establishes a
   contravariant equivalence from the category of finitely generated
   abelian groups to the category of diagonalisable algebraic groups,
   with quasi-inverse given by sending each group $G$ to its character
   group $G^*$.
 \end{proposition}
 \begin{proof}
 See, for example, Milne~\cite{Milne}.
 \end{proof}

 To explain the terminology, we mention the following result~\cite{Milne}.

 \begin{lemma}
 Let $V$ be a finite dimensional $\kk$-vector space and $G\subset \GL(V)$ an
 algebraic subgroup. There is a basis of $V$ in which
 $G$ acts diagonally on $V$ if and only if $G=G_\Lambda$ for some finitely
 generated abelian group $\Lambda$.
 \end{lemma}

 \begin{example}
 Let $\Lambda$ be a cyclic group.
 \begin{itemize}
 \item If $\Lambda$ has order $r\in \NN$ then $\kk[\Lambda]\cong
 \kk[t]/\langle t^r-1\rangle$, so $G_\Lambda\cong \mu_r:= \{p\in
 \mathbb{A}^1_\kk : t^r=1\}$.
 \item If $\Lambda\cong \ZZ$ then $G_\Lambda\cong \kk^\times
 :=\Spec \big(\kk[t,t^{-1}]\big)$ is the algebraic torus of dimension one.
 \end{itemize}
 \end{example}
 
 Direct sums of cyclic groups induce tensor products of group algebras
 that are compatible with the algebraic groups structure. Thus, the
 structure theorem finitely generated abelian groups implies that
 every diagonalisable algebraic group is isomorphic to a product
 \[
 G_\Lambda \cong \kk^\times\times \dots \times \kk^\times \times \mu_{r_1}^{m_1}\times \dots
 \times \mu_{r_t}^{m_t}
 \]
 of finitely many copies of the algebraic torus $\kk^\times$ with a
 finite abelian group.

 \subsection{On actions and grading}
 An \emph{algebraic action} of an affine algebraic group $G$ on an affine
 variety $X=\Spec(R)$ is a morphism $\mu\colon G\times
 X\longrightarrow X$ of affine varieties such that the following
 compatibility conditions hold:
 \begin{enumerate}
 \item $\id_X=\mu\circ(e\times \id_X) \colon X\to X$;
 \item $\mu\circ(\id_{G}\times \mu)=\mu\circ(\nu\times
 \id_X) \colon G\times G\times X\to X$.
 \end{enumerate}
 These conditions simply the natural conditions for the action of a
 group on a set restated in the category of affine varieties.

 \begin{example}
   Let $V$ be a finite-dimensional $\kk$-vector space and $G$ an
   affine algebraic group. Recall from Example~\ref{ex:GLV} that
   $\GL(V)$ is an affine variety. Algebraic actions of $G$ on $V$
   correspond precisely to those representations $\rho\colon G\to
   \GL(V)$ that are morphisms in the category of affine varieties.
 \end{example}

 We consider only algebraic actions of a diagonalisable algebraic
 group $G$ on an affine variety $X=\Spec(R)$.  In the
 presence of such an action, $G$ acts dually on $R$ via
 \begin{equation}
 \label{eqn:dualaction}
 (g\cdot f)(p) = f(g^{-1}\cdot p) \quad \text{ for all }g\in G, f\in R, p\in X.
 \end{equation}
 Let $\iota \colon R^G\hookrightarrow R$ denote the inclusion of the
subalgebra of $G$-invariant functions, i.e., functions $f\in R$ such
that $g\cdot f = f$. The following algebraic reformulation is extremely useful.

 \begin{theorem}
 \label{thm:grading}
 Let $\Lambda$ be a finitely generated abelian group. The following
 are equivalent:
 \begin{enumerate}
 \item[\one] an algebraic action of the diagonalisable group $G_\Lambda$ on
 an affine variety $X=\Spec(R)$;
 \item[\two] a $\Lambda$-grading of $R=\bigoplus_{\chi\in \Lambda} R_\chi$.
 \end{enumerate}
 \end{theorem}
 \begin{proof}
 Applying the $\Spec$ functor to the morphism $\mu\colon G_\Lambda\times
 X\rightarrow X$ and to the morphisms associated to $G_\Lambda$ gives
 a $\kk$-algebra homomorphism $\mu^*\colon R\to
 \kk[\Lambda]\otimes_\kk R$ satisfying conditions
 \begin{enumerate}
 \item[($1^*$)] $\id_R= (e^*\otimes \id_R)\circ \mu^*\colon R\to R$;
 \item[($2^*$)] $(\id_{\kk[\Lambda]}\otimes \mu^*)\circ \mu^*=(\nu^*\otimes
 \id_R)\circ \mu^* \colon R\to\kk[\Lambda]\otimes \kk[\Lambda]\otimes R$.
 \end{enumerate}
 For each character $\chi\in \Lambda$ of the group $G_\Lambda$, define 
 \[
 R_\chi:= \{ f\in
 R : \mu^*(f)=\chi\otimes f\}.
 \]
 Note that $R_\chi\cap R_{\chi'}= 0$ when $\chi\neq \chi'$, and that
 $R_\chi\otimes_\kk  R_\chi'\subseteq R_{\chi+\chi'}$ since $\mu^*$ is a
 $\kk$-algebra homomorphism.  For $f\in R$ decompose $\mu^*(f) =
 \sum_{\chi\in \Lambda} \chi\otimes f_\chi$. Applying condition
 ($2^*$) to $f\in R$ gives
 $\mu^*(f_\chi)=\chi\otimes f_\chi$, so $f_\chi\in R_\chi$. On the
 other hand, condition ($1^*$) gives $f=\sum_{\chi\in\Lambda} f_\chi$,
 from which we obtain $R=\bigoplus_{\chi\in \Lambda}
 R_\chi$. Conversely, given a $\Lambda$-grading of $R$ we construct a
 $\kk$-algebra homomorphism $\mu^*\colon R\to \kk[\Lambda]\otimes_\kk
 R$ by setting $\mu^*(f)=\chi\otimes f$ for each $f\in R_\chi$.
 Reversing the argument above shows that the induced morphism $\mu
 \colon G_\Lambda\times X\to X$ is an algebraic action. \end{proof}

For a character $\chi\in \Lambda$ of $G=G_\Lambda$, the elements $f\in
R_\chi$ are \emph{homogeneous of degree $\chi$ with respect to $G$},
or simply \emph{$G$-semi-invariants of weight $\chi$}.  Explicitly, an
element $f\in R$ is semi-invariant of weight $\chi$ if
 \[
 g\cdot f = \chi(g^{-1}) f \quad \text{ for all }g\in G.
 \]
 In light of \eqref{eqn:dualaction}, this means that $f(g\cdot p) =
 \chi(g)f(p)$ for all $g\in G$ and $p\in X$.  The zero-graded piece
 $R_0$ coincides with the $G$-invariant subalgebra of $R$.

 \subsection{Proj of a \protect$\ZZ\protect$-graded ring}
 The most common examples of diagonalisable group actions on affine
 varieties arise from finitely generated $\kk$-algebras that are
 $\ZZ$-graded:
 \[
 R=\bigoplus_{\chi\in \Lambda} R_\chi \quad\text{for }\Lambda=\ZZ.
 \]
 Assume $R_0=\kk$. Choose homogeneous generators $f_0,f_1,\dots, f_n\in R$
 and suppose for now that each has degree one.
 If we grade the polynomial ring $\kk[x_0,x_1,\dots, x_n]$ by setting
 $\deg(x_i) = 1$ for $0\leq i\leq n$ then the surjective map of
 $\kk$-algebras $\psi\colon \kk[x_0,x_1,\dots, x_n]\to R$ where
 $\psi(x_i) = f_i$ for $0\leq i\leq n$ preserves the
 $\Lambda$-grading. The ideal $\Ker(\psi)$ that records the
 relations between the generators $f_0,\dots, f_n$ is
 $\Lambda$-homogeneous, and cuts out a projective subvariety
 \[
 \Proj(R):= \big{\{}p\in \mathbb{P}^n_\kk : f(p) =0 \text{ for
 all homogeneous } f\in \Ker(\psi)\big{\}};
 \]
 this is the \emph{projective spectrum} of the graded ring $R$. A
 different choice of generators for $R$ leads to an alternative
 embedding of $\Proj(R)$.
 
 A geometric interpretation of $\Proj(R)$ is obtained from
 Theorem~\ref{thm:grading} above. Indeed, for $X=\Spec(R)$ the map
 $\psi\colon \kk[x_0,x_1,\dots, x_n]\to R$ defines a closed immersion
 $\varphi\colon X \rightarrow \mathbb{A}^{n+1}_\kk$ sending $p\in X$
 to the point $\big(f_0(p),f_1(p),\dots, f_n(p)\big)\in
 \mathbb{A}^{n+1}_\kk$. The $\Lambda$-grading determines actions of
 the algebraic torus $\kk^\times= \Spec\big(\kk[\Lambda]\big)$ on both $X$ and
 $\mathbb{A}^{n+1}_\kk$, making $\varphi$ into a
 $\kk^\times$-equivariant map.  Since each variable $x_i$ has degree one,
 let $t\in \kk^\times$ act on $\mathbb{A}^{n+1}_\kk$ as 
 \[
 t\cdot (p_0,p_1,,\dots,p_n)=(tp_0,\dots,t p_n). 
 \]
  If we remove from $X$ the
 common zero-locus of the generators $f_1,\dots, f_n\in R$, then the
 restriction of $\varphi$ descends to a morphism on
 $\kk^\times$-orbits that fits in to the right-hand square of the
 commutative diagram
 \begin{equation}
 \label{diag:GIT}
 \begin{CD}
 \Spec\big(R[f_i^{-1}]\big) @>>> X\smallsetminus \mathbb{V}(f_0,f_1,\dots, f_n) @>\varphi >> \mathbb{A}^{n+1}_\kk\smallsetminus\{0\} \\
 @V\pi VV  @V\pi VV  @VV\pi V \\
 \Spec \big(R[f_i^{-1}]^{\kk^\times}\big) @>>> \Proj (R) @>\overline{\varphi}>> \mathbb{P}^{n}_\kk
  \end{CD}
\end{equation}
where the horizontal maps in the right-hand square are closed
immersions, the horizontal maps in the left-hand square are open
embeddings, and the vertical maps identify points in the same
$\kk^\times$-orbit. The locus $X\smallsetminus
\mathbb{V}(f_0,f_1,\dots, f_n)$ on which the quotient map
$\pi\vert_{X}$ is well defined is covered by the principal affine open
subsets $X\smallsetminus \mathbb{V}(f_i) = \Spec\big(R[f_i^{-1}]\big)$
for $0\leq i\leq n$, and by examining the standard local coordinate
charts on $\mathbb{P}^n_\kk$, we see that $\Proj(R)$ is covered by
local charts of the form $\Spec \big(R[f_i^{-1}]^{\kk^\times}\big)$
for $0\leq i\leq n$, where
 \[
 R[f_i^{-1}]^{\kk^\times}:= \Bigg{\{}\frac{f}{f_i^j} : f\in R_{j} \text{ for
   }j\geq 0\Bigg{\}}
 \]
 is the $\kk^\times$-invariant subalgebra of the localisation of $R$
 at $f_i$. Note that $\Spec \big(R[f_i^{-1}]^{\kk^\times}\big)$ is
 well defined since $R[f_i^{-1}]^{\kk^\times}$ is a finitely generated
 $\kk^\times$-algebra (see Theorem~\ref{thm:affineGIT} to follow).

  \begin{example}
 \label{ex:O3}
 Associate to the line bundle $\mathscr{O}_{\mathbb{P}^1}(3)$ the
 $\ZZ$-graded algebra
 \[
 R:=\bigoplus_{\chi\in \ZZ}
 H^0\big{(}\mathbb{P}^1,\mathscr{O}_{\mathbb{P}^1}(3\chi)\big{)}\cong\kk[tx_1^3,
 tx_1^2x_2,tx_1x_2^2, tx_2^3],
 \]
 where the exponent of $t$ in a monomial gives the degree. The obvious
 homomorphism of graded $\kk$-algebras $\psi\colon
 \kk[y_1,y_2,y_3,y_4]\rightarrow \kk[t,x_1,x_2]$ has kernel the
 $\kk^\times$-homogeneous ideal $I$ from equation
 \eqref{eqn:twistedcubic}, and $\Proj(R)$ is the twisted cubic curve
 cut out by $I$ in the complete linear series $\vert
 \mathscr{O}_{\mathbb{P}^1}(3)\vert =\mathbb{P}^3_\kk$. The affine
 cone $\Spec(R)$ over $\Proj(R)$ featured in
 Example~\ref{ex:third11}.
\end{example}

 \begin{exercise}
   Generalise the previous example to the $d$-uple Veronese embedding
   on $\mathbb{P}^n$, and present $\Spec(R)$ as the quotient of affine
   space by a finite group (compare Exercise~\ref{ex:r111}).
 \end{exercise}

 To extend the construction to algebras $R$ whose generators do not
 all have degree one we recall weight projective space. Fix weights
 $a_0,\dots, a_n\in \ZZ_{>0}$, and let $t\in\kk^\times$ act on
 $\mathbb{A}^{n+1}_\kk$ as
 \[
 t\cdot(p_0,p_1,,\dots, p_n)=(t^{a_0}p_0,t^{a_1}p_1,\dots, t^{a_n}p_{n}).
 \]
 Equivalently, grade the polynomial ring $\kk[x_0,x_1,\dots,x_n]$ by
 setting $\deg(x_i)=a_i$ for $0\leq i\leq n$. For $0\leq i\leq n$, the
 coordinate ring of $U_i:=\mathbb{A}^{n+1}_\kk\smallsetminus (x_i=0)$
 is isomorphic to the quotient singularity of type
 $\frac{1}{a_i}(a_0,\dots, \widehat{a_i},\dots, a_n)$ from
 Exercise~\ref{ex:r111}. These varieties glue nicely along
 intersections and provide the standard open cover of \emph{weighted
   projective space} $\mathbb{P}_\kk(a_0,\dots,a_n)$. 
 
 Suppose now that a $\kk$-algebra $R$ has generators $f_0, f_1, \dots,
 f_n$ where $f_i$ has degree $a_i>0$ for $0\leq i\leq n$. For
 $X=\Spec(R)$, the induced closed immersion $\varphi\colon X
 \rightarrow \mathbb{A}^{n+1}_\kk$ is $\kk^\times$-homogeneous and
 descends to a map $\overline{\varphi}\colon \Proj(R)\rightarrow
 \mathbb{P}_\kk(a_0,\dots,a_n)$ on $\kk^\times$-orbits, where
 $\Proj(R)$ is the variety obtained by gluing the affine open sets
 \[
 R[f_i^{-1}]^{\kk^\times}:= \Bigg{\{}\frac{f}{f_i^j} : f\in R_{ja_i}
   \text{ for }j\geq 0\Bigg{\}}
 \]
 for $0\leq i\leq n$.

 \begin{remark}
 \label{rem:chitrivial}
 We need not always assume that $R_0=\kk$, in which case $\Proj(R)$ is a
   subvariety of $\mathbb{P}^n_{R_0}$.  Thus, while $\Proj(R)$ need
   not be projective, it is always projective over $\Spec(R_0)$.
\end{remark}

 \subsection{Affine GIT}
 The left-hand vertical map in diagram \eqref{diag:GIT} and the map
 analysed in Proposition~\ref{prop:finiteGIT} both arise from the
 inclusion of a $G$-invariant subalgebra $R^G\hookrightarrow R$ for
 some diagonalisable group action on a finitely generated
 $\kk$-algebra. The relation between morphisms of this kind and
 $G$-orbits in $\Spec(R)$ can be straightforward as in
 Proposition~\ref{prop:finiteGIT}, but it need not be as the following
 example shows.

 \begin{example}
 \label{ex:Pn}
 Consider the standard diagonal action of $\kk^\times$ on
 $\mathbb{A}^{n+1}_\kk=\Spec\big(\kk[x_0,\dots,x_n]\big)$ that defines
 $\mathbb{P}^{n}_\kk$.  Note that $\kk^\times$ acts on the open subset
 $U_i:= \mathbb{A}^{n+1}_\kk\smallsetminus (x_i=0)$ for $0\leq i\leq
 n$.
 \begin{enumerate} 
 \item[\one] For the action of $\kk^\times$ on the whole of
   $\mathbb{A}^{n+1}_\kk$ we have
   $\Spec\big(\kk[x_0,\dots,x_n]^{\kk^\times}\big)\cong\Spec(\kk)$, so
   the map $\mathbb{A}^{n+1}_\kk\to\Spec(\kk)$ induced by the
   inclusion $\kk\hookrightarrow\kk[x_0,\dots,x_n]$ identifies every
   point. In this case, the closure of each orbit in
   $\mathbb{A}^{n+1}_\kk$ contains the origin, so the variety
   $\Spec\big(\kk[x_0,\dots,x_n]^{\kk^\times}\big)$ may be thought of
   as parametrising $\kk^\times$-orbit closures.
 \item[\two] For the $\kk^\times$-action on the principal open set
   $U_i$ we have
   $\Spec\big{(}(\kk[x_0,\dots,x_n][x_i^{-1}])^{\kk^\times}\big{)}
   \cong
   \Spec\big{(}\kk[\textstyle{\frac{x_1}{x_i},\dots,1,\dots,\frac{x_n}{x_i}}]\big{)}\cong
   \mathbb{A}^{n}_\kk$.  The $\kk^\times$-orbits in $U_i$ are all
   fixed ratios $[p_0:\dots :p_n]$ with $p_i=1$. Such orbits are not
   closed in $\mathbb{A}^{n+1}_\kk$, but they are closed in $U_i$.
   Thus, the variety
   $\Spec\big(\kk[x_0,\dots,x_n][x_i^{-1}]^{\kk^\times}]\big)$ parametrises
   genuine $\kk^\times$-orbits in $U_i$.
 \end{enumerate}
 \end{example}
 
 \begin{remark}
   Note that Proposition~\ref{prop:finiteGIT} (compare
   Remark~\ref{rem:verbatim}) provides a special case.
 \end{remark}

 To state the general result, consider the action of a diagonalisable
 group $G$ on an affine variety $X=\Spec(R)$. A point $p\in X$ is
 \emph{stable} with respect to $G$ if the orbit $G\cdot p$ is closed
 in $X$, and if the stabiliser subgroup $\{g\in G : gp=p\}$ is finite.
 Let $X^{\textrm{s}}\subseteq X$ denote the locus of stable points in
 $X$. The main result of Geometric Invariant Theory for the action of
 $G$ on an affine variety now applies because $G$ is isomorphic to the
 product of an algebraic torus $(\kk^\times)^r$ with a finite group
 and hence is linearly reductive (see Dolgachev~\cite[\S3.3,
 \S6]{Dolgachev}, Mukai~\cite[\S4.3, \S5]{Mukai} or Mumford
 et.~al.~\cite[\S1.2]{Mumford}).

 \begin{theorem}
 \label{thm:affineGIT}
 Let a diagonalisable group $G$ act on an affine variety
 $X=\Spec(R)$:
 \begin{enumerate}
 \item[\one] the subalgebra $R^G$ is a finitely
 generated integral $\kk$-algebra;
 \item[\two] the stable locus
 $X^{\textrm{s}}\subseteq X$ is open, and there is a
 commutative diagram
 \[
 \begin{CD}
 X^{\textrm{s}} @>>> X=\Spec(R) \\ @VV\pi^{\textrm{s}}V @VV\pi V \\
 X^{\textrm{s}}/G @>>> \Spec(R^G)
  \end{CD}
\]
where the horizontal maps are open embeddings, and $\pi$ is surjective
and closed. Moreover, for all $p,p^\prime\in X$ we have
 \[
 \pi(p)=\pi(p^\prime) \iff \overline{G\cdot p} \cap \overline{G\cdot
 p^\prime}\neq\emptyset,
 \]
 and each fibre of $\pi$ contains a unique closed
 $G$-orbit in $X$; and
\item[\three] each the fibre of $\pi$ over a point of
  $X^{\textrm{s}}/G$ is equal to a unique closed orbit.
 \end{enumerate}
The variety $X/G:= \Spec(R^G)$ is the \emph{affine quotient}
 of $X$ by $G$.
 \end{theorem}

 We prefer whenever possible to work only with stable points, when the
 affine quotient $X^{\textrm{s}}/G$ really does parametrise $G$-orbits
 in $X^{\textrm{s}}$, at the expense of having to throw away part
 (possibly all) of $X$. For example, every point of $X$ is stable for
 the action of a finite group, in which case
 Theorem~\ref{thm:affineGIT} reduces to
 Proposition~\ref{prop:finiteGIT}. On the other hand, for the
 $\kk^\times$-action given in Example~\ref{ex:Pn}, no point of
 $\mathbb{A}^{n+1}_\kk$ is $\kk^\times$-stable whereas every point of
 $U_i\subset\mathbb{A}^{n+1}_\kk$ is $\kk^\times$-stable. This
 illustrates clearly that stability depends on the choice of $G$ and
 $X$.

 \begin{remark}
 The quotient map $\pi\colon X\to X/G$ satisfies a universal property,
 namely, that any morphism $\tau \colon X\to Z$ that is constant on
 $G$-orbits factors through $\pi\colon X\to X/G$. Given the
 categorical nature of this property, $X/G$ is often called the
 \emph{categorical quotient} of $X$ by $G$. The restriction to the
 stable locus $\pi^{\textrm{s}}\colon X^{\textrm{s}}\to
 X^{\textrm{s}}/G$ is a \emph{geometric quotient}.
 \end{remark}

 \subsection{Projective GIT}
 \label{sec:projectiveGIT}
 Let a diagonalisable group $G$ act algebraically on
 $X=\Spec(R)$, so $R$ is graded by a finitely generated abelian
 group $\Lambda$ by Theorem~\ref{thm:grading}. To produce a
 $\ZZ$-graded ring, choose a character $\chi\in \Lambda$ of $G$ and
 let $R_\chi$ denote the $\kk$-vector subspace of $R$ spanned by the
 $G$-semi-invariants of weight $\chi$. The \emph{algebra of
 $\chi$-semi-invariants}
 \begin{equation}
 \label{eqn:chisemi}
 \bigoplus_{j\in \ZZ} R_{j\chi}
 \end{equation}
 is a $\ZZ$-graded ring. We assume that $R_{j\chi}=0$ for $j<0$.

 \begin{lemma}
 The graded ring $\bigoplus_{k\geq 0} R_{k\chi}$ is a
 finitely-generated $\kk$-algebra.
 \end{lemma}
 \begin{proof}
   Lift the $G$-action on $X=\Spec(R)$ to the action on $X\times
   \mathbb{A}^1_\kk = \Spec\big(R[y]\big)$ for which $y$ is semi-invariant of
   weight $\chi^{-1}$. For $j\geq 0$, the function $fy^j\in R[y]$ is
   $G$-invariant if and only if $f$ is semi-invariant of weight
   $\chi^j$. The graded ring $\bigoplus_{j\geq 0} R_{j\chi}$ is
   therefore the $G$-invariant subring of $R[y]$.  The result follows
   from Theorem~\ref{thm:affineGIT}.
 \end{proof}

 The \emph{GIT quotient} of $X = \Spec(R)$ by the action of the
 diagonalisable group $G$ linearised by the character $\chi\in
 \Lambda$ is defined to be the projective spectrum 
 \[
 X\git_{\chi} G:= \Proj
 \Big{(}\bigoplus_{j\geq 0} R_{j\chi}\Big{)}.
 \]
 If $X\git_{\chi} G$ is nonempty then it is projective over the affine
 variety $\Spec(R_0)$. In fact, most of the statements that one might
 make about GIT quotients work well only under the assumption that
 $X\git_{\chi} G\neq\emptyset$, in which case we say that $\chi$ is
 \emph{effective}. While this is not always the case, we
 restrict attention only to effective characters for the sake of
 simplicity.

 \begin{example}
 \label{ex:wPn}
 Consider the action of $\kk^\times$ on $\mathbb{A}^3_\kk =
 \Spec\big(\kk[x,y,z]\big)$ where the variables $x,y,z$ have degree one, two
 and three respectively. For $\chi=6$, the subalgebra of $\kk[x,y,z]$
 spanned by all monomials $x^ay^bz^c$ for which $6\vert a+2b+3c$ is
 generated by $\{x^6, x^4y, x^2y^2, y^3, x^3z, xyz, z^2\}$. The GIT
 quotient $\mathbb{A}^3_\kk\git_\chi\kk^\times$ is the weighted
 projective plane $\mathbb{P}_\kk(1,2,3)$ presented as a subvariety of
 $\mathbb{P}^6_\kk$.
 \end{example}

  To describe the local structure of the GIT quotient, we say that a
 point $p\in X$ is \emph{$\chi$-semistable} with respect to $G$ if
 there exists $j> 0$ and $f\in R_{j\chi}$ satisfying $f(p)\neq 0$. The
 affine open subset $X_\chi^{\textrm{ss}}:= X\smallsetminus
 \mathbb{V}(f_0, f_1, \dots, f_n)$ of $X$ consisting of all such
 points is the \emph{$\chi$-semistable locus} in $X$. The restriction
 of $\pi$ to $X^{\textrm{ss}}_\chi$ and to $X\smallsetminus \mathbb{V}(f_i)$
 gives a commutative diagram
  \begin{equation}
 \label{diag:GITcharacter}
 \begin{CD}
 \Spec\big(R[f_i^{-1}]\big) @>>> X_\chi^{\textrm{ss}} \\ @VVV @VV\pi V \\
 \Spec\big(R[f_i^{-1}]^G\big) @>>> X\git_{\chi} G
  \end{CD}
 \end{equation}
where the horizontal maps are open embeddings, and the left-hand
vertical map is the affine quotient map arising from the $G$-action on
$X\smallsetminus\mathbb{V}(f_i)$. Theorem~\ref{thm:affineGIT} implies
the following.

 \begin{theorem}
 \label{thm:projectiveGIT}
 The categorical quotient $X\git_{\chi} G$ parametrises equivalence
classes of $G$-orbit-closures for all $\chi$-semistable points in
$X_\chi^{\textrm{ss}}$. The geometric quotient $X\git_{\chi} G$
parametrises $G$-orbits of $\chi$-stable points in
$X_\chi^{\textrm{ss}}$.
 \end{theorem}
 
 For the action of a diagonalisable group $G$ on an affine variety
 $X$, any fractional character $\chi\in G^*\otimes_\ZZ \QQ$ is called
 \emph{generic} if every $\chi$-semistable of $X$ is $\chi$-stable.
 Following Thaddeus~\cite{Thaddeus2} and
 Dolgachev--Hu~\cite{DolgachevHu}, define an equivalence relation on
 the set of generic fractional characters by setting $\chi\sim\chi'$
 if and only if every $\chi$-stable point of $X$ is $\chi'$-stable,
 and vice versa.  This gives a polyhedral decomposition of the set of
 effective fractional characters called the \emph{(polarised) GIT
   chamber decomposition} for the given action of $G$ on $X$. As
 $\chi$ varies in a chamber the GIT quotient $X\git_{\chi} G$ remains
 unchanged as a variety, though the graded ring that defines it varies.

 \begin{remark}
 We present several examples of this phenomenon later in these notes,
 including:
 \begin{enumerate}
 \item[\one] an investgation of the GIT chamber decomposition for an
   action of $(\kk^\times)^2$ on $\mathbb{A}^4_\kk$ that defines the
   first Hirzebruch surface $\mathbb{F}_1$ (see
   Exercise~\ref{ex:fundiagramF1}); and
 \item[\two] an explicit construction of the GIT chamber decomposition
   arising from an action of $(\kk^\times)^2$ on an affine fourfold
   associated to the quotient singularity of type $\frac{1}{3}(1,2)$;
   this is better known as the Weyl chamber decomposition of type
   $A_2$ (see Exercise~\ref{exa:VGIT1/3(1,2)}).
 \end{enumerate}
 \end{remark}

 \section{Semiprojective toric varieties}
 \label{sec:toric}
 This section introduces semiprojective toric varieties using the
 projective GIT construction for diagonalisable group actions. This
 simple class of toric varieties includes the overwhelming majority of
 examples that one meets in the course of everyday life. We do not
 assume that such toric varieties are normal, but in that case we
 obtain naturally a convex geometric interpretation via the fan. We
 also present a simple proof of Cox's quotient description of a normal
 toric variety in the semiprojective case.


 \subsection{Affine toric varieties}
 \label{sec:affinetoric}
 For $d\in \mathbb{N}$, let $S\subseteq \ZZ^d$ be a finitely generated
 subsemigroup containing the zero element. Without loss of generality,
 we may assume that $S$ generates $\ZZ^d$ over $\ZZ$, in which case
 $S$ is the image of the subsemigroup $\NN^n\subset \ZZ^n$ under a
 lattice map $\pi\colon \ZZ^n\to \ZZ^d$.  Identify the semigroup
 algebras of $\NN^n$ and $\ZZ^d$ with the polynomial ring
 $\kk[x_1,\dots,x_n]$ and the ring of Laurent polynomials
 $\kk[t_1^{\pm 1},\dots, t_d^{\pm 1}]$ respectively.  Write $\kk[S] =
 \{t^u \in \kk[t_1^{\pm 1},\dots, t_d^{\pm 1}] : u\in S\}$ for the
 semigroup algebra of $S$, where the element $u=(u_1,\dots, u_d)\in S$
 defines the monomial $t^u:= \prod_{1\leq i\leq d} t_i^{u_i}$. The map
 $\pi\colon \NN^n\to \ZZ^d$ factors through $S$ and induces diagrams
 of semigroup algebras and affine varieties as shown:
 \begin{equation}
 \label{eqn:cdaffinetoric}
 \xymatrix@C=.6em{\kk[x_1,\dots,x_n]\ar[d]^{\psi}\ar[drrr]^{\pi} 
    &&& &&& \mathbb{A}^n_\kk &&& \\
    \kk[S]\ar[rrr]^{\tau} &&&\kk[t_1^{\pm 1},\dots, t_d^{\pm
   1}]&&&\Spec\big(\kk[S]\big) \ar[u]^{\psi^*} &&& (\kk^\times)^d\ar[ulll]^{\pi^*}\ar[lll]^{\tau^*}}
 \end{equation}
 \noindent Since $S$ generates
 $\ZZ^d$ over $\ZZ$, the algebra $\kk[t_1^{\pm 1},\dots, t_d^{\pm 1}]$
 can be obtained from $\kk[S]$ by localising suitably, so the
 righthand horizontal map is an open embedding that presents the
 algebraic torus $(\kk^\times)^d$ as a Zariski dense open subset of
 $\Spec\big(\kk[S]\big)$.  Moreover, the lefthand vertical map
 surjects, so the righthand vertical map is a closed immersion with
 image cut out by the \emph{toric ideal} of $\pi$, namely the prime
 ideal
 \[
 \Ker(\psi) =\big{(}x^{u}-x^{u'} \in \kk[x_1,\dots,x_n] :  u-u'\in \Ker(\pi)\big{)}.
 \]
 An affine variety is a \emph{toric variety} if it is of the form
 $\Spec\big(\kk[S]\big)$ for some subsemigroup $S\subset \ZZ^d$ as above. With
 this definition, toric varieties need not be normal.

 \begin{example}
 \label{ex:cusp}
 Define $\pi\colon \NN^4\to \ZZ^2$ by the matrix $\bigl[
    \text{\scriptsize $\begin{array}{rrrr} 4 & 3 & 1 & 0 \\ 0 & 1 & 3
    & 4 \end{array}$} \bigr]$, so $S\subset \ZZ^2$ is generated by the
    columns of the matrix. The toric variety $\Spec\big(\kk[S]\big)$ is not
    normal.
 \end{example}

 \begin{remark}
   For further evidence that nonnormal toric varieties arise naturally
   in a geometric context, see Example~\ref{exa:notnormallinearseries}
   and Remark~\ref{rem:notnormal2}.
 \end{remark}

 To state the normality criterion for a toric variety, associate to a
 subsemigroup $S\subseteq \ZZ^d$ with generators $s_1,\dots, s_n\in S$
 the polyhedral cone 
 \[
 \Cone(S):= \Big{\{}\sum_{1\leq i\leq n}
 \lambda_i s_i\in\QQ^d : \lambda_i\in \QQ_{\geq 0}\Big{\}}.
 \]
 Clearly, $S$ is contained in the set $\Cone(S)\cap \ZZ^d$ of
 integral points of the cone, but this inclusion can be strict, e.g.,
 consider the vector $\bigl[ \text{\scriptsize $\begin{array}{r} 2 \\
 2 \end{array}$} \bigr]$ from Example~\ref{ex:cusp}.

 \begin{lemma}
 \label{lem:normal}
 The affine toric variety $\Spec\big(\kk[S]\big)$ is normal if and only if the
 semigroup $S$ consists of the integral points of a polyhedral cone.
 When this holds,  $S=\Cone(S)\cap \ZZ^d$.
 \end{lemma}

 \begin{exercise} 
 Prove Lemma~\ref{lem:normal}.
 \end{exercise}
 
 For a simple class of examples, recall from Exercise~\ref{ex:r111}
 that the quotient singularity of type $\frac{1}{r}(1,a)$ is the
 affine quotient of $\mathbb{A}^2_\kk$ by the action of the cyclic
 group $\ZZ/r$ with generator the matrix
 $\diag(\epsilon,\epsilon^{a})$ for $\epsilon$ a primitive $r$th root
 of unity. Proposition~\ref{prop:finiteGIT} shows that the coordinate
 ring of $\mathbb{A}^2_\kk/(\ZZ/r)$ is the $\ZZ/r$-invariant subring of
 $\kk[x,y]$. The monomial $x^{u_1}y^{u_2}$ is $\ZZ/r$-invariant if and
 only if $u_1 + au_2 \equiv 0 \mod r$, and we write $S = \big{\{}
 (u_1,u_2)\in \NN^2 : u_1+au_2 \equiv 0 \mod r\big{\}}$ for the
 subsemigroup of (exponents of) $\ZZ/r$-invariant monomials.  To
 describe explicitly the generators of $S$, define the
 Jung--Hirzebruch continued fraction expansion of \(r/(r-a)\) to be
 \[
 \frac{r}{r-a} = b_{1} - \frac{1}{b_{2} - \frac{1}{\dots - \frac{1}{b_{t}}}}.
 \]
 The subsemigroup $S\subset \ZZ^2$ is generated by $\{u_0,u_1,\dots,
 u_{t+1}\}$, where $u_0=(r,0), u_1=(r-a,1)$ and
 $u_{i+1}:=b_iu_i-u_{i-1}\text{ for }1\leq i\leq t$.
 
 More generally, the cyclic quotient singularity of type
 $\frac{1}{r}(a_1,\dots, a_n)$ from Exercise~\ref{ex:r111} is toric.
 For the action of $\ZZ/r$ on $\mathbb{A}^n_\kk$ with generator matrix
 $\diag(\epsilon^{a_1}, \dots, \epsilon^{a_n})$, a monomial
 $x_1^{u_1}x_2^{u_2}\cdots x_n^{u_n}$ is $\ZZ/r$-invariant if and only
 if the exponent vector lies in the subsemigroup
 \[
 S = \Big{\{} (u_1,\dots, u_n)\in \NN^n : \sum_{i=1,\dots, n} a_iu_i
 \equiv 0 \mod r\Big{\}}
 \]
 of (exponents of) $\ZZ/r$-invariant monomials. In fact, the lattice
 point $(u_1,\dots, u_n)\in \NN^n$ lies in $S$ if and only if its
 inner product with the rational vector $\frac{1}{r}(a_1,\dots, a_n)$
 is zero; this explains the terminology. In any event, the cyclic quotient
 $\mathbb{A}^n/(\ZZ/r)\cong \Spec\big{(}\kk[S]\big{)}$ is a toric variety.
 
  \begin{example}
  The singularity of type $\frac{1}{7}(1,4)$ from
  Exercise~\ref{ex:714} is $\Spec\big(\kk[S]\big)$ where the semigroup $S$ is
  obtained as the image of the map $\pi\colon \NN^5\to \ZZ^2$ with
  matrix $\bigl[ \text{\scriptsize $\begin{array}{rrrrr} 7 & 3 & 2 & 1
  & 0 \\ 0 & 1 & 3 & 5 & 7 \end{array}$} \bigr]$.
  \end{example}

 \subsection{Semiprojective toric varieties}
 Let $S\subseteq \ZZ^d$ be a finitely generated subsemigroup that
 generates $\ZZ^d$ over $\ZZ$ as above, and suppose in addition that
 $\nu\colon \ZZ^d\to \ZZ$ is a $\ZZ$-linear map satisfying
 $\nu(S)\subseteq \NN$. Set
 $M:=\Ker(\nu)$. The semigroup algebra $\kk[S]$ admits a $\ZZ$-grading
 $\bigoplus_{j\geq 0} \kk[S]_j$, where $\kk[S]_j$ has a $\kk$-vector
 space basis consisting of monomials $t^u\in \kk[S]$ with
 $\nu(u)=j$. The graded pieces satisfy $\kk[S]_j=0$ for $j<0$ by
 assumption, and $\kk[S]_0$ is the semigroup algebra of $S\cap M$, so
 the variety
 \[
 \Proj\big{(}\kk[S]\big{)}:=\Proj\Big{(}\bigoplus_{j\geq 0} \kk[S]_{j}\Big{)}
 \]
 is projective over the affine toric variety $\Spec\big(\kk[S]_0\big)$. A
 variety is a \emph{semiprojective toric variety} if it is of the form
 $\Proj\big{(}\kk[S]\big{)}$ for some $\ZZ$-graded semigroup algebra
 $\kk[S]$ as above, and it is a \emph{projective toric variety} if
 $\kk[S]_0=\kk$.

 \begin{example}
 \label{exa:notnormallinearseries}
   Let $X$ be a normal projective toric variety and let $L\in \Pic(X)$
   be very ample with complete linear series $\vert L\vert =
   \Proj\big(\bigoplus_{j\geq 0}H^0(X,L^j)\big)\cong \mathbb{P}^m_\kk$
   for $m=\dim_\kk H^0(X,L)-1$. The closed immersion $\varphi_{\vert
   L\vert}\colon X\to \vert L\vert$ obtained by evaluating a
   $\kk$-vector space basis of $H^0(X,L)$ at points of $X$ presents
   $X$ as a toric subvariety of $\vert L\vert$. The affine cone
   $\Spec\big(\bigoplus_{j\geq 0}H^0(X,L^j)\big)$ over $\varphi_{\vert
   \mathcal{L}\vert}(X)$ need not be normal in general. See
   \cite[Example~4.2]{HSS} for details.
 \end{example}

 Put geometrically (compare Theorem~\ref{thm:grading}), every
 algebraic variety that is obtained as the GIT quotient of an affine
 toric variety $\Spec\big(\kk[S]\big)$ by an action of $\kk^\times$ is a
 semiprojective toric variety. In fact, one can say the following.

 \begin{proposition}
 \label{prop:toricGIT}
 Let $G$ be a diagonalisable group acting on an affine toric variety
 $\Spec\big{(}\kk[S]\big{)}$. Then $\Spec\big{(}\kk[S]\big{)}\git_\chi
 G$ is a semiprojective toric variety for any character
 $\chi\in G^*$.
 \end{proposition}
 \begin{proof}
 Let $\Lambda=G^*$ denote the character group of
 $G$. Theorem~\ref{thm:grading} shows that the $G$-action on
 $\Spec\big{(}\kk[S]\big{)}$ gives a $\Lambda$-grading of $\kk[S]$ and
 hence a semigroup homomorphism $\pi\colon S\to \Lambda$. For $\chi\in
 \Lambda$, the $\kk$-algebra $\bigoplus_{j\geq 0} \kk[S]_{j\chi}$ of
 $\chi$-semi-invariant functions is the semigroup
 algebra of
 \[
 S_\chi=\bigoplus_{j\geq 0} \Big(S\cap
 \pi^{-1}\big{(}j\chi\big{)}\Big).
 \]
 The GIT quotient $\Spec\big{(}\kk[S]\big{)}\git_\chi G$ is the
 semiprojective toric variety $\Proj\big(\kk[S_\chi]\big)$.
 \end{proof}

 The class of semiprojectve toric varieties is an especially nice
 class: while it includes all affine and projective toric varieties,
 it also includes many varieties that arise as toric resolutions of
 affine toric singularites (see Example~\ref{exa:res1/3(1,2)}).

 \begin{example}
 \label{ex:dualseq}
 For the graded ring $\kk[x,y,z]$ from Example~\ref{ex:wPn} with
 $\deg(x)=1$, $\deg(y)=2$ and $\deg(z)=3$, the corresponding toric
 variety $\Proj\big(\kk[x,y,z]\big)$ is weighted projective space
 $\mathbb{P}_\kk(1,2,3)$.  As $t^{u_i}$ varies over the generators $x,y$
 and $z$, the ring $\kk[x,y,z][t^{-u_i}]^{\kk^\times}$ is
 $\kk[\frac{y}{x^2},\frac{z}{x^3}]$,
 $\kk[\frac{x^2}{y},\frac{xz}{y^2},\frac{z^2}{y^3}]$ and
 $\kk[\frac{x^3}{z},\frac{xy}{z},\frac{y^3}{z^2}]$ respectively. The
 corresponding charts on $\mathbb{P}_\kk(1,2,3)$ are isomorphic to the
 affine toric varieties $\mathbb{A}^2_{\kk}$, the singularity of type
 $\frac{1}{2}(1,1)$ and the singularity of type $\frac{1}{3}(1,2)$.
 \end{example}

 \begin{exercise} 
 \label{ex:weightedPn}
   For the graded ring $R=\kk[x_0,\dots,x_n]$ with $\deg(x_i)=a_i>0$
   for $0\leq i\leq n$, prove that the charts of $\Proj(R)$
   corresponding to the generators $x_0,\dots,x_n$ of $R$ are the
   quotient singularities of type $\frac{1}{a_i}(a_0,\dots,
   \widehat{a_i},\dots, a_n)$ for $0\leq i\leq n$. In particular,
   weighted projective space admits a cover by affine toric varieties.
 \end{exercise}

 The semiprojective toric varieties in Example~\ref{ex:dualseq} and
 Exercise~\ref{ex:weightedPn} each admit a cover by affine toric
 varieties. This statement holds true in general.

 \begin{proposition}
 \label{prop:charts}
 Let $X$ be a semiprojective toric variety and set $T_X:=\Spec\big(\kk[M]\big)$. Then:
 \begin{enumerate}
 \item[\one] $X$ admits a cover by affine toric varieties, each
 containing the algebraic torus $T_X$ as a Zariski dense open subset;
 \item[\two] the action of $T_X$ on itself extends to an action on
 $X$.
 \end{enumerate} 
 \end{proposition}
 \begin{proof}
 Consider $X = \Proj\big{(}\kk[S]\big{)}$ where $\kk[S]$ is a
   $\ZZ$-graded semigroup algebra. For \one, choose $u_1,\dots, u_m\in
   S$ so that the monomials $t^{u_1},\dots, t^{u_m}$ form a minimal
   $\kk[S]_0$-algebra generating set for $\kk[S]$.  The $\Proj$
   construction shows that $X$ is covered by charts
   $\Spec\big(\kk[S][t^{-u_i}]^{\kk^\times}\big)$ for $1\leq i\leq m$. The
   $\kk$-vector space underlying $\kk[S][t^{-u_i}]^{\kk^\times}$ has
   basis consisting of elements $t^u/(t^{u_i})^j$ for $j\geq 0$ and
   for $t^u\in \kk[S]$ of degree $j\deg(u_i)$. It follows that
   $\kk[S][t^{-u_i}]^{\kk^\times}$ is isomorphic to the semigroup
   algebra $\kk[S_i]$ of the subsemigroup $S_i:= \{u-ju_i\in M : u\in
   S, j\in \NN\}$ of $M$. Since $S$ generates $\ZZ^d$ over $\ZZ$, it
   follows that $S_i$ generates $M$ over $\ZZ$. Therefore $X$ admits a
   cover by the affine toric varieties $\Spec\big(\kk[S_i]\big)$ for
   $1\leq i\leq m$ and, since $\kk[S_i]$ can be localised to obtain
   $\kk[M]$, the torus $T_X$ is dense in $\Spec\big(\kk[S_i]\big)$ for
   $1\leq i\leq m$. This completes the proof of \one. For \two, since
   $S_i\subset M$ we obtain by restriction an $M$-grading of
   $\kk[S_i]$ and hence a $T_X$-action on $\Spec\big(\kk[S_i]\big)$
   that extends the $T_X$-action on itself. These actions agree where
   charts overlap, so the action extends to the whole of $X$.
 \end{proof}

 \begin{remark}
 The existence of a Zariski dense algebraic torus in $X$ whose action
 on itself extends to an action on $X$ can be used to characterise semiprojective
 toric varieties.
 \end{remark}

 \subsection{A dual approach via convex geometry}
 \label{sec:dual}
 In the special case when the toric variety
 $X=\Proj\big{(}\kk[S]\big{)}$ is normal, the affine cover by toric
 charts admits a simple description in terms of polyhedral cones. This
 approach relies on some elementary facts about duality in convex
 geometry that we now recall.

 Let $S\subseteq \ZZ^d$ be a finitely generated subsemigroup that
 generates $\ZZ^d$ over $\ZZ$, and let 
 \begin{equation} 
  \label{eq:gradingdiagram}
  \begin{CD}   
    0 @>>> M @>>> \ZZ^{d} @>\nu >> \ZZ @>>> 0
 \end{CD}
\end{equation}
be a short exact sequence of lattice maps satisfying $\nu(S)\subseteq
\NN$. Write $N:=\Hom_{\ZZ}(M,\ZZ)$ for the lattice dual to $M$ and
$\langle \;\;,\;\;\rangle \colon M\times N\to \ZZ$ for the dual
pairing. Let $\sigma\subseteq N\otimes_\ZZ \QQ$ be a \emph{polyhedral
cone}, that is, the $\mathbb{Q}_{\geq 0}$-span of a finite set of
vectors in $N\otimes_\ZZ \QQ$. The \emph{dual cone}
 \[
 \sigma^\vee:= \big{\{}u\in M\otimes_\ZZ \QQ : \langle u,v\rangle \geq
 0 \text{ for all } v\in \sigma\big{\}}
 \]
 is also a polyhedral cone. A \emph{face} of $\sigma$ is the
 intersection of $\sigma$ with a supporting hyperplane, and $\sigma$
 is \emph{strongly convex} if $\sigma^\vee$ spans $M\otimes_\ZZ \QQ$.
 Gordan's Lemma asserts that the integral points $\sigma^\vee\cap M$
 form a finitely generated subsemigroup of $M$, giving rise to a
 normal affine toric variety $\Spec\big(\kk[\sigma^\vee\cap M]\big)$. This
 dual approach to constructing a semigroup algebra may appear rather
 convoluted at first, but the cones $\sigma\subseteq N\otimes_\ZZ \QQ$
 that encode the toric charts of a normal semiprojective toric variety
 satisfy the following nice convex-geometry property.

 \begin{proposition}
 \label{prop:cones}
 Every normal semiprojective toric variety $X=\Proj\big(\kk[S]\big)$ is
 covered by toric charts of the form $\Spec\big(\kk[\sigma_i^\vee\cap M]\big)$
 for $1\leq i\leq m$, where $\sigma_1,\dots, \sigma_m\subseteq
 N\otimes_\ZZ \QQ$ are strongly convex polyhedral cones such that
 $\sigma_j\cap \sigma_k$ is a face of both $\sigma_j$ and $\sigma_k$
 for $1\leq j,k\leq m$.
 \end{proposition}
 \begin{proof}
 Choosing $\ell\geq 0$ sufficiently large ensures that the
 $\ZZ$-graded $\kk[S]_0$-algebra $\bigoplus_{j\geq 0}\kk[S]_{j\ell}$
 is generated in degree $j=1$, so $X$ is covered by the charts
 $\Spec\big(\kk[S][t^{-u}]^{\kk^\times}\big)$ as $u$ ranges over some subset
 of $S\cap \nu^{-1}(\ell)$. To remove redundancy in this cover, let
 $\conv(S\cap \nu^{-1}(\ell))$ denote the polyhedron obtained as the
 convex hull of the points $S\cap \nu^{-1}(\ell)$ in the vector space
 $\nu^{-1}(\ell)\otimes_\ZZ \QQ$, and let $u_1,\dots, u_m\in S$ be the
 vertices of $\conv(S\cap \nu^{-1}(\ell))$. As in the proof of
 Proposition~\ref{prop:charts}, $\kk[S][t^{-u_i}]^{\kk^\times}$ is
 isomorphic to the semigroup algebra of $S_i:= \{u-u_i\in M : u\in S
 \cap \nu^{-1}(\ell)\}$. Set $N:=\Hom_\ZZ(M,\ZZ)$, and for each $1\leq
 i\leq m$ define the cone $\sigma_i:= \Cone(S_i)^\vee$ in
 $N\otimes_\ZZ \QQ$. Since $X$ is normal, the semigroup $\sigma_i^\vee
 \cap M = \Cone(S_i)\cap M$ is equal to $S_i$ by
 Lemma~\ref{lem:normal} and hence $X$ is covered by the charts
 $\Spec\big(\kk[\sigma_i^\vee\cap M]\big)$ for $1\leq i\leq m$. Each cone
 $\sigma_i$ is strongly convex since $\sigma_i^\vee = \Cone(S_i)$
 spans $M\otimes_\ZZ \QQ$. We leave the proof of the final statement as a simple exercise in convex geometry.
 \end{proof}

 \begin{exercise}
 Complete the proof of Proposition~\ref{prop:cones}.
 \end{exercise}

 \begin{remark}
 \label{rem:normal}
 Normality of $\Spec\big{(}\kk[S]\big{)}$ implies normality of
 $\Spec\big{(}\kk[S]\big{)}\git_\chi G$ by Lemma~\ref{lem:normal}, so
 Proposition~\ref{prop:cones} applies in particular to those varieties
 obtained as GIT quotients of normal affine toric varieties by the
 action of a diagonalisable group (see
 Proposition~\ref{prop:toricGIT}).
 \end{remark}

 \begin{example}
 \label{ex:F1fan}
 Consider the action of $T=(\kk^\times)^2$ on
 $\mathbb{A}^4_\kk=\Spec\big(\kk[x_1,x_2,x_3,x_4]\big)$ given by
 \[
 (t_1,t_2)\cdot (p_1,p_2,p_3,p_4) = (t_1p_1,t_1^{-1}t_2p_2,t_1p_3,t_2p_4),
 \]
 and let $\pi\colon \ZZ^4\to \ZZ^2$ be the map defined by the matrix
 \[
 \begin{pmatrix} 1 & 0
      & -1 & 0 \\ 0 & 1 & 1 & -1 \end{pmatrix}.
 \]
 The character $\chi= (1,1)$ of $(\kk^\times)^2$ defines the
      $\ZZ$-graded semigroup $S_\chi=\bigoplus_{j\geq 0}\big(\NN^4\cap
      \pi^{-1}(j,j)\big)$ whose semigroup algebra $\kk[S_\chi] =
      \bigoplus_{j\geq 0} \kk[x_1,x_2,x_3,x_4]_{j\chi}$ determines the
      projective toric variety $\mathbb{A}^4_\kk\git_\chi T =
      \Proj\big(\kk[S_\chi]\big)$. To compute the toric charts that
      cover $\mathbb{A}^4_\kk\git_\chi T$, note that the $\kk$-algebra
      $\kk[S_\chi]$ is generated in degree $j=1$ by the five monomials
      shown in Figure~\ref{fig:F1}(a). The algebra of $T$-invariant
      functions on $U_1:=\mathbb{A}^4_\kk\setminus (x_3x_4=0)$ is
      $\kk[U_1]^T=\kk\Big{[}\frac{x_1x_4}{x_3x_4},\frac{x_1^2x_2}{x_3x_4},\frac{x_1x_2x_3}{x_3x_4},\frac{x_2x_3^2}{x_3x_4}\Big{]}$
      and hence the affine quotient is $U_1/T \cong
      \Spec\big(\kk[\frac{x_1}{x_3},\frac{x_2x_3}{x_4}]\big)\cong
      \mathbb{A}^2_\kk$. Computing the other charts similarly shows
      that the chart corresponding to the monomial $x_1x_2x_3$ is
      redundant.

 Alternatively, by examining the proof of Proposition~\ref{prop:cones}
 we see that the charts can be read off directly from the polytope
 $\conv(S_{\chi}\cap \nu^{-1}(1))$ shown in Figure~\ref{fig:F1}(a):
 each vertex determines a chart, and the coordinates have exponent
 vectors given by the edge-vectors emanating from the corresponding
 vertex.
 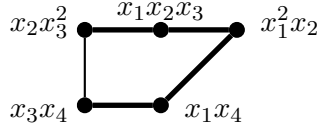
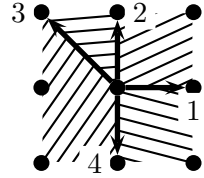
\begin{figure}[!ht]
    \centering \mbox{ \subfigure[Polytope $\conv(S_{\chi}\cap
    \nu^{-1}(1))$]{ \psset{unit=1cm}
        \begin{pspicture}(-1.5,-0.5)(4,2.2)
          \cnode*[fillcolor=black](0,0.5){3pt}{P1}
          \cnode*[fillcolor=black](1,0.5){3pt}{P2}
          \cnode*[fillcolor=black](0,1.5){3pt}{P4}
          \cnode*[fillcolor=black](1,1.5){3pt}{P5}
          \cnode*[fillcolor=black](2,1.5){3pt}{P6}
          \ncline[linewidth=2pt]{-}{P1}{P2}
          \ncline{-}{P1}{P4}
          \ncline[linewidth=2pt]{-}{P4}{P5}
          \ncline[linewidth=2pt]{-}{P5}{P6}
          \ncline[linewidth=2pt]{-}{P6}{P2}
          \rput(-0.6,0.4){\psframebox*{$x_3x_4$}}
         \rput(1.7,0.4){\psframebox*{$x_1x_4$}}
          \rput(-0.6,1.6){\psframebox*{$x_2x_3^2$}}
          \rput(2.7,1.6){\psframebox*{$x_1^2x_2$}}
         \rput(1,1.75){$x_1x_2x_3$}
          \end{pspicture}}
      \qquad \qquad  
       \subfigure[Cones $\sigma_i$ in $N\otimes_\ZZ \QQ$]{
        \psset{unit=1cm}
        \begin{pspicture}(-1.0,-0.5)(3.5,2.2)
        \pspolygon[linecolor=white,fillstyle=hlines*,
          hatchangle=25](1,1)(1,2)(2,2)(2,1)
          \pspolygon[linecolor=white,fillstyle=hlines*,
          hatchangle=5](1,1)(1,2)(0,2)(0,2)
          \pspolygon[linecolor=white,fillstyle=hlines*,
          hatchangle=55](1,1)(0,2)(0,0)(1,0)
          \pspolygon[linecolor=white,fillstyle=hlines*,
          hatchangle=165](1,1)(2,1)(2,0)(1,0)
          \cnode*[fillcolor=black](0,0){3pt}{P1}
          \cnode*[fillcolor=black](1,0){3pt}{P2}
          \cnode*[fillcolor=black](2,0){3pt}{P3}
          \cnode*[fillcolor=black](0,1){3pt}{P4}
          \cnode*[fillcolor=black](1,1){3pt}{P5}
          \cnode*[fillcolor=black](2,1){3pt}{P6}
          \cnode*[fillcolor=black](0,2){3pt}{P7}
          \cnode*[fillcolor=black](1,2){3pt}{P8}
          \cnode*[fillcolor=black](2,2){3pt}{P9}
          \ncline[linewidth=2pt]{->}{P5}{P7}
          \ncline{-}{P5}{P11}
          \ncline[linewidth=2pt]{->}{P5}{P6}
          \ncline[linewidth=2pt]{->}{P5}{P2}
          \ncline[linewidth=2pt]{->}{P5}{P8}
          \rput(2,0.7){\psframebox*{$1$}}
          \rput(1.3,2){\psframebox*{$2$}}
          \rput(-0.3,2){\psframebox*{$3$}}
          \rput(0.7,0){\psframebox*{$4$}}
        \end{pspicture}}
    }
    \caption{The Hirzebruch surface $\mathbb{F}_{1}$ as a toric
    variety} \label{fig:F1}
  \end{figure}
  For each chart $U_i$, the cone generated by the pair of edges
  vectors emanating from the vertex is the cone $\sigma_i^\vee$ whose
  integral points $\sigma_i^\vee\cap M$ define the semigroup algebra
  $\kk[\sigma_i^\vee\cap M]\cong \kk[U_i]^T$. The dual cone $\sigma_i$
  is generated by the inward-pointing normal vectors of
  $\sigma_i^\vee$, and forms one of the two-dimensional cones in
  Figure~\ref{fig:F1}(b). The patching data for the local toric charts
  shows that $\mathbb{A}^4_\kk\git_\chi T$ is the Hirzebruch
  surface $\mathbb{F}_1 = \mathbb{P}(\mathscr{O}_{\mathbb{P}^1}\oplus
  \mathscr{O}_{\mathbb{P}^1}(1))$.
 \end{example}

 \begin{exercise}
 Repeat the analysis of Example~\ref{ex:F1fan} for each of the
 characters $\chi = (1,0)$, $(0,1)$, $(-1,2)$ and $(1,2)$. Explain
 your findings in terms of $\pi$ and the birational geometry of
 $\mathbb{F}_1$.   
 \end{exercise}

 \begin{example}
 \label{exa:toricabelian}
   Let $G \subset \GL(n,\kk)$ be a finite abelian subgroup of order
   $r$. Choose coordinates $x_{1},\dots ,x_{n}$ on $\mathbb{A}^n_\kk$
   to diagonalise the action and write $g =
   \mbox{diag}\big{(}\epsilon^{\alpha_1(g)},\dots
   ,\epsilon^{\alpha_n(g)}\big{)}$ where \(\epsilon\) is a primitive
   \(r^{\text{th}}\) root of unity and $0\leq \alpha_j(g) < r$ for
   $1\leq j\leq n$. Define the lattice
 \[
 N := \ZZ^n + \textstyle{\sum_{g\in G}} \; \ZZ\cdot
\frac{1}{r}\big{(}\alpha_1(g),\dots ,\alpha_n(g)\big{)}
 \]
 and set $M=\Hom(N,\ZZ)$.  A Laurent monomial
 $x_1^{m_1}x_2^{m_2}\cdots x_n^{m_n}\in \kk[x_1^{\pm 1},\dots,x_n^{\pm
   1}]$ is $G$-invariant if and only if the exponent vector lies in
 $M$. Thus, if $\sigma\subset N\otimes_\ZZ \QQ$ denotes the
 positive orthant, then $\kk[\sigma^\vee\cap M]=\kk[x_1,\dots, x_n]^G$
 and hence the abelian quotient singularity $\mathbb{A}_\kk^n/G$ is the
 normal toric variety $\Spec\big{(}\kk[\sigma^\vee\cap M]\big{)}$ for
 $N$ and $\sigma$ as above.
 \end{example}

 \subsection{Fans and lattice polyhedra}
 For a lattice $N$ with dual lattice $M$, a \emph{fan}
 $\Sigma$ in $N\otimes_\ZZ \QQ$ is a finite collection of strongly
 convex polyhedral cones $\sigma\subset N\otimes_\ZZ \QQ$ such that:
 \begin{enumerate}
 \item[\one] for $\sigma, \sigma'\in \Sigma$, the
 cone $\sigma\cap \sigma'$ is a face of both $\sigma$ and $\sigma'$;
 \item[\two] for $\sigma\in\Sigma$, each face of $\sigma$ also lies in
 $\Sigma$.  
 \end{enumerate} 
 We may assume that the cones in any given fan do not all lie in a
 hyperplane in $N\otimes_\ZZ \QQ$.

 Each cone $\sigma\in \Sigma$ determines the normal toric variety
 $U_\sigma = \Spec\big{(}\kk[\sigma^\vee\cap M]\big{)}$ with dense
 algebraic torus $\Spec\big(\kk[M]\big)$. Moreover, each face of
 $\sigma$ is of the form $\tau=\sigma\cap u^\perp$ for some $u\in M$,
 and one can show that $\kk[\tau^\vee\cap M]$ is the localisation of
 $\kk[\sigma^\vee\cap M]$ at $t^u$. It follows that $U_\tau$ is a
 principal open affine subset of $U_\sigma$, and these normal affine
 varieties glue to give
 \[
 X_{\Sigma}:=\bigsqcup_{\sigma\in \Sigma} U_\sigma/\sim_{\textrm{glue}}.
 \] 
 We work only with fans arising from lattice polyhedra as follows. For
 an $n$-dimensional lattice polyhedron $P\subseteq \QQ^n$ and for $w
 \in (\QQ^n)^\ast$ we denote by $\face_{w}(P)$ the face of $P$
 minimising $w$.  Given a face $F$ of $P$, the \emph{inner normal
 cone} $\mathcal{N}_{P}(F)$ is the set of $w \in (\QQ^n)^{\ast}$ such
 that $\face_{w}(P)=F$, and the \emph{inner normal fan}
 $\mathcal{N}(P)$ of $P$ is the fan consisting of the inner normal
 cones $\{\mathcal N_{P}(F)\}$ as $F$ varies over the faces of $P$.

 \begin{proposition}
 Let $X_\Sigma$ be a normal toric variety. Then $X_\Sigma$ is
 semiprojective if and only if $\Sigma$ is the inner normal fan of a lattice
 polyhedron.
 \end{proposition}
 
 To construct the semiprojective toric variety directly from the
 polyhedron $P\subseteq \QQ^n$ let $\widetilde{P}$ be the polyhedral
 cone over $P$ obtained as the closure in $\QQ^{n}\oplus \QQ$ of the
 set $\{(\lambda p,\lambda) : p\in P, \lambda\in \QQ_{\geq 0}\}$.  The
 semigroup $S_P:= \widetilde{P}\cap (\ZZ^n\oplus \ZZ)$ determines the
 normal affine toric variety $\Spec\big(\kk[S_P]\big)$ of dimension
 $n+1$. The second projection $\nu\colon \ZZ^n\oplus\ZZ\to \ZZ$ gives
 $S_P$ a $\ZZ$-grading and hence defines the normal semiprojective
 toric variety $X_P:= \Proj\big(\kk[S_P]\big)$ of dimension $n$. The
 polyhedron $P$ is equal to $\conv(S_P\cap \nu^{-1}(1))$, the convex
 hull of the lattice points $S_P\cap \nu^{-1}(1)$ in the vector space
 $\nu^{-1}(1)\otimes_\ZZ \QQ$. The $n$-dimensional cones
 $\mathcal{N}_{P}(F)$ as $F$ ranges over the vertices of $P$ coincide
 with the cones $\sigma_1,\dots, \sigma_m$ from
 Proposition~\ref{prop:cones} that determine an affine toric cover of
 $X_P$. More generally, as $F$ ranges over the faces of $P$ we produce
 the fan $\Sigma = \mathcal{N}(P)$ that encodes the data to
 reconstruct toric variety $X_P$. It is worth recording two special cases:

 \begin{example}
 If $P\subseteq \QQ^n$ is a polyhedral cone then its normal fan
 consists of the faces of the dual cone $\sigma := P^\vee\subset
 N_\QQ$.  The toric variety $X_P$ is $\Spec\big{(}\kk[\sigma^\vee\cap M]\big{)}$ for
 $M=\Hom_\ZZ(N,\ZZ)$.
 \end{example}

 \begin{example}
 If $P\subseteq \QQ^n$ is a polytope (= bounded polyhedron) then
 $S_P\cap \nu^{-1}(0)$ is zero. Thus, the zero-graded piece of the
 algebra $\kk[S_P]$ is $\kk$, so $X_P$ is projective.
 \end{example}

 \begin{remark}
 \label{rem:smooth}
 Geometric properties of a normal toric variety $X$ are encoded by its
 fan $\Sigma\subset N\otimes_\ZZ \QQ$. An $n$-dimensional toric
 variety $X$ is smooth if and only if the generators of each
 $n$-dimensional cone $\sigma\in \Sigma$ form a $\ZZ$-basis of the
 lattice $N$.  Moreover, $X$ is an orbifold if each $n$-dimensional
 cone has precisely $n$ cone generators, in which case the fan
 $\Sigma$ is \emph{simplicial}. For more on the relation between a
 normal toric variety and its fan, see Fulton~\cite{Fulton} or
 Oda~\cite{Oda}.
 \end{remark}

  \begin{exercise}
 \label{ex:delPezzo}
   A very simple class of toric varieties is the class of smooth toric
   del Pezzo surfaces of which there are only five examples:
   $\mathbb{P}^1\times \mathbb{P}^1, \mathbb{P}^2$ and $\mathbb{P}^2$
   blown-up in one, two, or three points; the fans are shown in
   Figures~\ref{fig:F1} and~\ref{fig:Fanosurfaces}. Show that each fan
   can be realised as the inner normal fan of a \emph{reflexive}
   lattice polygon, that is, a lattice polygon $P$ such that
   the origin is the only lattice point in the interior of $P$.
  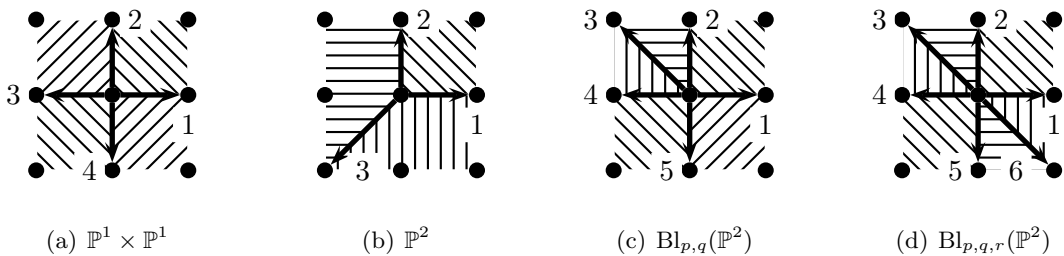
\begin{figure}[!ht]
    \centering
    \mbox{
      \subfigure[$\mathbb{P}^1\times \mathbb{P}^1$]{
        \psset{unit=1cm}
        \begin{pspicture}(0,-0.5)(2,2.2)
        \pspolygon[linecolor=white,fillstyle=hlines*,
          hatchangle=135](1,1)(1,2)(2,2)(2,1)
          \pspolygon[linecolor=white,fillstyle=hlines*,
          hatchangle=45](1,1)(1,2)(0,2)(0,1)
          \pspolygon[linecolor=white,fillstyle=hlines*,
          hatchangle=135](1,1)(0,1)(0,0)(1,0)
          \pspolygon[linecolor=white,fillstyle=hlines*,
          hatchangle=45](1,1)(2,1)(2,0)(1,0)
          \cnode*[fillcolor=black](0,0){3pt}{P1}
          \cnode*[fillcolor=black](1,0){3pt}{P2}
          \cnode*[fillcolor=black](2,0){3pt}{P3}
          \cnode*[fillcolor=black](0,1){3pt}{P4}
          \cnode*[fillcolor=black](1,1){3pt}{P5}
          \cnode*[fillcolor=black](2,1){3pt}{P6}
          \cnode*[fillcolor=black](0,2){3pt}{P7}
          \cnode*[fillcolor=black](1,2){3pt}{P8}
          \cnode*[fillcolor=black](2,2){3pt}{P9}
          \ncline[linewidth=2pt]{->}{P5}{P4}
          \ncline{-}{P5}{P11}
          \ncline[linewidth=2pt]{->}{P5}{P6}
          \ncline[linewidth=2pt]{->}{P5}{P2}
          \ncline[linewidth=2pt]{->}{P5}{P8}
          \rput(2,0.6){\psframebox*{$1$}}
          \rput(1.3,2){\psframebox*{$2$}}
          \rput(-0.3,1){\psframebox*{$3$}}
          \rput(0.7,0){\psframebox*{$4$}}
        \end{pspicture}}
      \qquad \qquad  
      \subfigure[$\mathbb{P}^2$]{
        \psset{unit=1cm}
 \begin{pspicture}(0,-0.5)(2,2.2)
          \pspolygon[linecolor=white,fillstyle=hlines*,
          hatchangle=135](1,1)(1,2)(2,2)(2,1)
          \pspolygon[linecolor=white,fillstyle=hlines*,
          hatchangle=0](1,1)(1,2)(0,2)(0,0)
          \pspolygon[linecolor=white,fillstyle=hlines*,
          hatchangle=90](1,1)(2,1)(2,0)(0,0)
          \cnode*[fillcolor=black](0,0){3pt}{P1}
          \cnode*[fillcolor=black](1,0){3pt}{P2}
          \cnode*[fillcolor=black](2,0){3pt}{P3}
          \cnode*[fillcolor=black](0,1){3pt}{P4}
          \cnode*[fillcolor=black](1,1){3pt}{P5}
          \cnode*[fillcolor=black](2,1){3pt}{P6}
          \cnode*[fillcolor=black](0,2){3pt}{P7}
          \cnode*[fillcolor=black](1,2){3pt}{P8}
          \cnode*[fillcolor=black](2,2){3pt}{P9}
          \ncline[linewidth=2pt]{->}{P5}{P8}
          \ncline{-}{P5}{P11}
          \ncline[linewidth=2pt]{->}{P5}{P6}
          \ncline[linewidth=2pt]{->}{P5}{P1}
          \rput(2,0.6){\psframebox*{$1$}}
          \rput(1.3,2){\psframebox*{$2$}}
          \rput(0.5,0){\psframebox*{$3$}}
        \end{pspicture}}
      \qquad \qquad  
      \subfigure[$\text{Bl}_{p,q}(\mathbb{P}^2$)]{
        \psset{unit=1cm}
        \begin{pspicture}(0,-0.5)(2,2.2)
        \pspolygon[linecolor=white,fillstyle=hlines*,
          hatchangle=135](1,1)(1,2)(2,2)(2,1)
          \pspolygon[linecolor=white,fillstyle=hlines*,
          hatchangle=0](1,1)(1,2)(0,2)(0,2)
          \pspolygon[linecolor=white,fillstyle=hlines*,
          hatchangle=90](1,1)(0,2)(0,1)
          \pspolygon[linecolor=white,fillstyle=hlines*,
          hatchangle=135](1,1)(1,0)(0,0)(0,1)
          \pspolygon[linecolor=white,fillstyle=hlines*,
          hatchangle=45](1,1)(2,1)(2,0)(1,0)
          \cnode*[fillcolor=black](0,0){3pt}{P1}
          \cnode*[fillcolor=black](1,0){3pt}{P2}
          \cnode*[fillcolor=black](2,0){3pt}{P3}
          \cnode*[fillcolor=black](0,1){3pt}{P4}
          \cnode*[fillcolor=black](1,1){3pt}{P5}
          \cnode*[fillcolor=black](2,1){3pt}{P6}
          \cnode*[fillcolor=black](0,2){3pt}{P7}
          \cnode*[fillcolor=black](1,2){3pt}{P8}
          \cnode*[fillcolor=black](2,2){3pt}{P9}
          \ncline[linewidth=2pt]{->}{P5}{P7}
          \ncline[linewidth=2pt]{->}{P5}{P6}
          \ncline[linewidth=2pt]{->}{P5}{P4}
          \ncline[linewidth=2pt]{->}{P5}{P2}
          \ncline[linewidth=2pt]{->}{P5}{P8}
          \rput(2,0.6){\psframebox*{$1$}}
          \rput(1.3,2){\psframebox*{$2$}}
          \rput(-0.3,2){\psframebox*{$3$}}
          \rput(-0.3,1){\psframebox*{$4$}}
          \rput(0.7,0){\psframebox*{$5$}}
        \end{pspicture}}
      \qquad \qquad  
      \subfigure[$\text{Bl}_{p,q,r}(\mathbb{P}^2$)]{
        \psset{unit=1cm}
        \begin{pspicture}(0,-0.5)(2,2.2)
        \pspolygon[linecolor=white,fillstyle=hlines*,
          hatchangle=135](1,1)(1,2)(2,2)(2,1)
          \pspolygon[linecolor=white,fillstyle=hlines*,
          hatchangle=0](1,1)(1,2)(0,2)
         \pspolygon[linecolor=white,fillstyle=hlines*,
          hatchangle=90](1,1)(0,2)(0,1)
          \pspolygon[linecolor=white,fillstyle=hlines*,
          hatchangle=135](1,1)(1,0)(0,0)(0,1)
          \pspolygon[linecolor=white,fillstyle=hlines*,
          hatchangle=0](1,1)(1,0)(2,0)
          \pspolygon[linecolor=white,fillstyle=hlines*,
          hatchangle=90](1,1)(2,1)(2,0)
          \cnode*[fillcolor=black](0,0){3pt}{P1}
          \cnode*[fillcolor=black](1,0){3pt}{P2}
          \cnode*[fillcolor=black](2,0){3pt}{P3}
          \cnode*[fillcolor=black](0,1){3pt}{P4}
          \cnode*[fillcolor=black](1,1){3pt}{P5}
          \cnode*[fillcolor=black](2,1){3pt}{P6}
          \cnode*[fillcolor=black](0,2){3pt}{P7}
          \cnode*[fillcolor=black](1,2){3pt}{P8}
          \cnode*[fillcolor=black](2,2){3pt}{P9}
          \ncline[linewidth=2pt]{->}{P5}{P7}
          \ncline{-}{P5}{P11}
          \ncline[linewidth=2pt]{->}{P5}{P6}
          \ncline[linewidth=2pt]{->}{P5}{P2}
         \ncline[linewidth=2pt]{->}{P5}{P4}
         \ncline[linewidth=2pt]{->}{P5}{P3}
          \ncline[linewidth=2pt]{->}{P5}{P8}
          \rput(2,0.6){\psframebox*{$1$}}
          \rput(1.3,2){\psframebox*{$2$}}
          \rput(-0.3,2){\psframebox*{$3$}}
          \rput(-0.3,1){\psframebox*{$4$}}
          \rput(0.7,0){\psframebox*{$5$}}
          \rput(1.5,0){\psframebox*{$6$}}
        \end{pspicture}}
    }
    \caption{Fans for the remaining four smooth toric del Pezzo surfaces}\label{fig:Fanosurfaces}
  \end{figure}
 \end{exercise}

 \subsection{Cox's construction}
 \label{sec:Cox}
 Every semiprojective toric variety $X$ is by construction the GIT
quotient of an affine toric variety by an action of $\kk^\times$. A
result of Cox~\cite{Cox} asserts that if $X$ is normal then it can be
constructed as a GIT quotient of affine space $\mathbb{A}^n_\kk$,
though one must replace the $\kk^\times$-action by that of a
diagonalisable group $G_\Lambda$. This construction is ubiquitous in
the study of toric varieties, and it is simple to describe in the
semiprojective case.

 For a normal semiprojective toric variety $X$ of dimension $n$ with
fan $\Sigma$, write $\Sigma(1)$ for the set of \emph{rays}
(one-dimensional cones) in $\Sigma$. We associate
to each ray $\rho \in \Sigma(1)$, a $T_X$-invariant Weil divisor
$D_\rho$ in $X$ as follows.  For each cone $\sigma$ containing $\rho$,
the hyperplane $\rho^\perp$ dual to $\rho$ cuts out a facet of
$\sigma^\vee$. This facet determines a divisor in $U_\sigma$ (given by
the vanishing of the function on $U_\sigma$ defined by the
inward-pointing normal vector), and $D_\rho$ is the divisor in
$X_\Sigma$ obtained by gluing these divisors together.  These divisors
generate the free abelian group $\ZZ^{\Sigma(1)}$ of $T_X$-invariant
(Weil) divisors, and there is an exact sequence
\begin{equation} 
  \label{eq:fundiagram}
  \begin{CD}   
    0 @>>> M @>\div>> \ZZ^{\Sigma(1)} @>\pi >> A_{n-1}(X) @>>> 0
 \end{CD}
\end{equation}
where $M:=\Hom(N,\ZZ)$ is the dual lattice.  Here, $\div(u) =
\sum_{\rho\in \Sigma(1)} \langle u,v_\rho\rangle D_\rho$, where
$v_\rho$ is the primitive lattice point on the cone $\rho$.

 The \emph{total coordinate ring} of $X$ is the
polynomial ring $R_X:= \kk[x_\rho : \rho \in \Sigma(1)]$ obtained as
the semigroup algebra of the semigroup $\NN^{\Sigma(1)}$ of
$T_X$-invariant effective divisors.  For each cone $\sigma \in
\Sigma$, write $\widehat{\sigma}$ for the set of one-dimensional cones
in $\Sigma$ that are not contained in $\sigma$, and consider the
\emph{irrelevant ideal}
 \[
 B_X:=
\Big{(}\prod_{\rho \in \widehat{\sigma}} x_{\rho}\in R_X : \sigma \in
\Sigma(n)\Big{)}
 \]
 whose generators are indexed by the top dimensional cones in
 $\Sigma$.  The map $\pi$ from \eqref{eq:fundiagram} induces a grading
 of $R_X$ by the finitely generated $\ZZ$-module $A_{n-1}(X)$, where
 $\deg(x^u):= \pi(u)$. The diagonalisable group
 $G:=\Hom_\ZZ(A_{n-1}(X),\kk^\times)$ acts on the affine space
 $\mathbb{A}^{\Sigma(1)}_\kk = \Spec(R_X)$ by
 Theorem~\ref{thm:grading}, and the ideal $B_X$ is homogeneous
 with respect to this action.

 \begin{theorem}[Cox~\cite{Cox}]
 \label{thm:Cox}
   Let $X$ be a normal semiprojective toric variety with simplicial fan
   $\Sigma$. Then $X$ is isomorphic to:
 \begin{enumerate}
 \item[\one] the GIT quotient $\mathbb{A}^{\Sigma(1)}_\kk\git_{L}G$ for any
   relatively ample line bundle $L$ on $X$;
 \item[\two] the geometric quotient of $\mathbb{A}^{\Sigma(1)}_\kk\setminus
   \mathbb{V}(B_X)$ by the action of $G$.
 \end{enumerate}
 \end{theorem}
 \begin{proof}
  Since the character group of $G$ is $A_{n-1}(X)$, we may regard the
   line bundle $L$ as a character of $G$ and hence
   $\mathbb{A}^{\Sigma(1)}_\kk\git_{L}G$ is well defined. For any $j\in
   \NN$, the set $\NN^{\Sigma(1)}\cap \pi^{-1}(L^{\otimes j})$
   coincides with the set of effective $T_X$-invariant divisors $D$ in
   $X$ satisfying $\mathscr{O}_X(D)=L^{\otimes j}$. It follows that
   the $L^{\otimes j}$-graded piece of ring $R_X$ is $H^0(X,L^{\otimes
   j})$, and hence
 \[
 \bigoplus_{j\geq 0} H^0(X,L^{\otimes j})\cong \bigoplus_{j\geq 0}
 \big{(}R_X\big{)}_{L^{\otimes j}}.
 \]
 Since $L$ is relatively ample, $\Proj$ of the left hand side is equal
 to $X$. Part \one\ follows since $\Proj$ of the right hand side is
 the GIT quotient $\mathbb{A}^{\Sigma(1)}_\kk\git_{L}G$. To establish the
 isomorphism between \one\ and \two, we assume without loss of
 generality that $L$ is (relatively) very ample. Then for each toric
 chart $U_\sigma$ in the cover of $X$, the corresponding face of the
 polyhedron $\conv(S\cap \pi^{-1}(L))$ defines a section
 $s_\sigma\in H^0(X,L)$ obtained as a product $s_\sigma =
 \prod_{\rho\in \widehat{\sigma}} x_\rho^{m_\rho}\in R_X$ for $m_\rho
 > 0$. The GIT quotient $\mathbb{A}^{\Sigma(1)}_\kk\git_{L}G$ is the
 geometric quotient of the $L$-stable locus
 $\mathbb{A}^{\Sigma(1)}_\kk\setminus \mathbb{V}((s_\sigma : \sigma \in
 \Sigma))$ by the action of $G$. The result follows since the radical
 of the ideal $(s_\sigma : \sigma \in \Sigma)$ is $B_X$.
 \end{proof}

 \begin{example}
 \label{ex:fundiagramF1}
 The toric fan of $X=\mathbb{F}_1$ is shown in
 Figure~\ref{fig:F1}. For $i=1,\dots, 4$, write $D_i$ for the toric
 divisor corresponding to the ray indexed by $i$.  For the standard
 basis $e_1,e_2\in \ZZ^2=M$, we have $\div(e_1)= D_1-D_3$ and
 $\div(e_2)=D_2+D_3-D_4$. If we choose $\mathscr{O}_X(D_1)$ and
 $\mathscr{O}_X(D_4)$ as the basis for $\Pic(X)\cong A_1(X)$, then the
 sequence \eqref{eq:fundiagram} becomes
 \begin{equation} 
  \label{eq:fundiagramF1}
  \begin{CD}   
    0 @>>> \ZZ^2 @>\left[ \text{\scriptsize $\begin{array}{rrrr} 1 & 0
    \\ 0 & 1 \\ -1 & 1 \\ 0 & -1 \end{array}$} \right]>> \ZZ^{4}
    @>\bigl[ \text{\scriptsize $\begin{array}{rrrr} 1 & -1 & 1 & 0 \\
    0 & 1 & 0 & 1 \end{array}$} \bigr]>> \ZZ^2 @>>> 0.
 \end{CD}
\end{equation}
 The line bundle $L=\mathscr{O}_X(D_1+D_4)$ is very ample. The
 monomials shown in Figure~\ref{fig:F1}(a) form both a $\kk$-vector
 space basis for the $L$-graded piece of $R_X=\kk[x_1,x_2,x_3,x_4]$
 and a set of $\kk$-algebra generators of $\bigoplus_{j\geq 0}
 (R_X)_{L^\otimes j}$. The GIT diagram \eqref{diag:GIT} shows that
 these $\kk$-algebra generators cut out the $L$-unstable locus, so
 $\mathbb{A}^{4}_\kk\git_{L}G$ is the categorical quotient of
 $\mathbb{A}^{4}_\kk\setminus \mathbb{V}(J)$ by the action of $G$
 for the ideal $J = (x_1x_3, x_1x_4, x_1^2x_2, x_1x_2x_3, x_2x_3^2)$.
 To compute $\mathbb{V}(J)$ we may of course replace $J$ by its
 radical, namely $\rad(J) = (x_1x_3, x_1x_4, x_1x_2, x_2x_3)$. In this
 case, the result boils down to the observation that $\rad(J)=B_X$.
 \end{example}

 \begin{example}
 \label{exa:res1/3(1,2)}
 Consider the toric variety $X$ whose fan $\Sigma\subset N\otimes_\ZZ
 \QQ$ is shown in Figure~\ref{fig:res1/3(1,2)}(a), where
 $N=\ZZ^2+\ZZ\cdot \frac{1}{3}(1,2)$. It follows from
 Remark~\ref{rem:smooth} that $X$ is smooth. If we list the rays
 $\rho_0,\dots,\rho_3$ as shown then $\mathscr{O}_X(D_0)$ and
 $\mathscr{O}_X(D_3)$ provide a basis for $\Pic(X)\cong A_1(X)$.
 \begin{figure}[!ht]
    \centering \mbox{
\subfigure[The fan $\Sigma$ in $N\otimes_\ZZ \QQ$]{
        \psset{unit=1cm}
        \begin{pspicture}(-1.0,-0.3)(3.5,3.2)
        \pspolygon[linecolor=white,fillstyle=hlines*,
          hatchangle=0](0,0)(0,3)(1.5,3)
          \pspolygon[linecolor=white,fillstyle=hlines*,
          hatchangle=135](0,0)(1.5,3)(3,1.5)
          \pspolygon[linecolor=white,fillstyle=hlines*,
          hatchangle=90](0,0)(3,1.5)(3,0)
          \cnode*(0,0){3pt}{P1}
          \cnode*(0.6667,1.3333){3pt}{P2}
          \cnode*(2,0){3pt}{P3}
          \cnode*(1.3333,0.6667){3pt}{P4}
          \cnode*(0,2){3pt}{P5}
          \cnode*(1.5,3){0pt}{P7}
          \cnode*(3,1.5){0pt}{P8}
          \cnode*(0,3){0pt}{P9}
          \cnode*(3,0){0pt}{P10}
          \ncline[linewidth=1.5pt]{-}{P1}{P7}
          \ncline[linewidth=1.5pt]{-}{P1}{P8}
          \ncline[linewidth=1.5pt]{-}{P1}{P9}
          \ncline[linewidth=1.5pt]{-}{P1}{P10}
          \rput(1.5,0){\psframebox*{$0$}}
          \rput(0.8,0.4){\psframebox*{$1$}}
          \rput(0.4,0.8){\psframebox*{$2$}}
          \rput(0,1.5){\psframebox*{$3$}}
        \end{pspicture}}
    \qquad \qquad  
 \subfigure[Polyhedron $P$]{ \psset{unit=1cm}
        \begin{pspicture}(-1,-0.2)(4,3.2)
         \pspolygon[linecolor=white,fillstyle=hlines*,
          hatchangle=45](0,0)(0,3)(1,3)(2.3,2.3)(3,1)(3,0)
         \cnode*(0,3){0pt}{Q1}
         \cnode*(1,3){3pt}{Q2}
         \cnode*(2.3,2.3){3pt}{Q3}
         \cnode*(3,1){3pt}{Q4}
         \cnode*(3,0){0pt}{Q5}
        \ncline[linewidth=1.5pt]{-}{Q1}{Q2}
        \ncline[linewidth=1.5pt]{-}{Q2}{Q3}
        \ncline[linewidth=1.5pt]{-}{Q3}{Q4}
        \ncline[linewidth=1.5pt]{-}{Q4}{Q5}
         \rput(3.6,1){\psframebox*{$x_2x_3^3$}}
          \rput(1.2,3.4){\psframebox*{$x_0^3x_1$}}
         \rput(2.6,2.6){$x_0x_3$}
          \end{pspicture}}
     }
    \caption{The minimal resolution of the singularity of type $\frac{1}{3}(1,2)$\label{fig:res1/3(1,2)}}
  \end{figure}
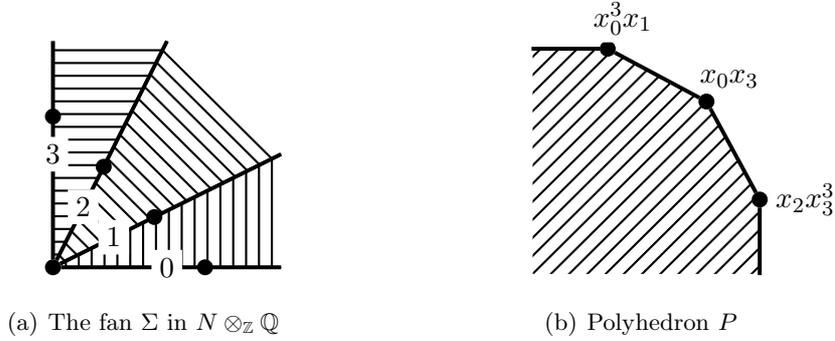
 The lattice $M$ dual to $N$ is the sublattice of $(\ZZ^2)^\vee$
 consisting of the exponents of monomials in $\kk[x,y]$ that are
 invariant under the group action of type $\frac{1}{3}(1,2)$ described
 in Section~\ref{sec:affinetoric}. If we choose the exponents of the
 monomials $x^3, xy$ as a $\ZZ$-basis for $M$ then sequence
 \eqref{eq:fundiagram} becomes
 \begin{equation} 
  \label{eq:fundiagramres1/3(1,2)}
  \begin{CD}   
    0 @>>> \ZZ^2 @>\left[ \text{\scriptsize $\begin{array}{rrrr} 3 & 1
    \\ 2 & 1 \\ 1 & 1 \\ 0 & 1 \end{array}$} \right]>> \ZZ^{4}
    @>\bigl[ \text{\scriptsize $\begin{array}{rrrr} 1 & -2 & 1 & 0 \\
    0 & 1 & -2 & 1 \end{array}$} \bigr]>> \ZZ^2 @>>> 0.
 \end{CD}
\end{equation}
In this case, $L=\mathscr{O}_X(D_0+D_3)$ is very ample, and the
monomials in the total coordinate ring corresponding to the vertices
of the polyhedron $P = \conv\big(\NN^4\cap \pi^{-1}(L)\big)$ appear in
Figure~\ref{fig:res1/3(1,2)}(b). The GIT quotient
$\mathbb{A}_\kk^{4}\git_{L}G$ is the categorical quotient of
$\mathbb{A}_\kk^{4}\setminus \mathbb{V}(x_0^3x_1, x_0x_3, x_2x_3^3)$
by the action of $G$. Again, the radical of the ideal is the
irrelevant ideal $B_X = (x_0x_1, x_0x_3, x_2x_3)$.
 \end{example}

 \begin{exercise}
   Show that the affine quotient $\mathbb{A}_\kk^{4}\git_{0}G \cong
   \Spec\big(\kk[\NN^4\cap M]\big)$ is isomorphic to the cyclic
   quotient singularity $\mathbb{A}^2_\kk/(\ZZ/3)$. Thus, the
   natural morphism $\mathbb{A}_\kk^{4}\git_{L}G\to
   \mathbb{A}_\kk^{4}\git_{0}G$ obtained by variation of GIT quotient
   is the minimal resolution $X\to\mathbb{A}^2_\kk/(\ZZ/3)$.
 \end{exercise}


 \section{Moduli spaces of quiver representations}
 \label{sec:quivers}
 We now turn our attention to quivers and to the abelian category of
 quiver representations. Certain toric varieties, known as toric
 quiver varieties, arise naturally as fine moduli spaces for
 representations of a quiver with a fixed dimension vector; the key to
 this construction is King's notion of $\theta$-stability. We extend
 the moduli construction to the case of bound quivers (= quivers with
 relations), and indicate some important examples which illustrate that
 the introduction of relations makes the study of the resulting moduli
 spaces much more delicate.

 \subsection{On quivers, path algebras and the incidence map}
 A \emph{quiver} $Q$ is a directed graph given by a set $Q_0$ of
 vertices, a set $Q_1$ of arrows, and maps $\tail,\head\colon Q_1\to
 Q_0$ that specify the tail and head of each arrow. We assume that
 both $Q_0, Q_1$ are finite sets, and that $Q$ is connected, i.e., the
 graph obtained by forgetting the orientation of arrows in $Q$ is
 connected.  A nontrivial \emph{path} $p$ in $Q$ of length $\ell\in
 \NN$ from vertex $i\in Q_0$ to vertex $j\in Q_0$ is a sequence of
 arrows $a_1\cdots a_\ell$ with $\head(a_{k}) = \tail(a_{k+1})$ for $1
 \leq k < \ell$; we set $\tail(p):= \tail(a_{1})$ and $\head(p):=
 \head(a_\ell)$.  In addition, each vertex $i \in Q_0$ gives a trivial
 path $e_i$ of length zero with $\tail(e_i) = \head(e_i) = i$.  A
 cycle is a nontrivial path with the same head and tail, and $Q$ is
 \emph{acyclic} if it contains no cycles. A \emph{tree} is a connected
 acyclic quiver. 

 Let $\kk$ be a field and $Q$ a finite, connected quiver.  The
 \emph{path algebra} $\kk Q$ of $Q$ is the associative $\kk$-algebra
 whose underlying $\kk$-vector space has as basis the set of paths in
 $Q$, where the product of basis elements equals the basis element
 defined by concatenation of the paths if possible, or zero otherwise.
 The identity element in $\kk Q$ is $\sum_{i\in Q_0} e_i$, where $e_i$
 is the trivial path of length zero at $i\in Q_0$. Note that $\kk Q$
 has finite dimension over $\kk$ if and only if $Q$ contains no
 cycles. The algebra $\kk Q$ is graded by path length, and the
 zero-graded subring $(\kk Q)_0\subset \kk Q$ spanned by the trivial
 paths $e_i$ for $i\in Q_0$ is a semisimple ring in which the elements
 $e_i$ are orthogonal idempotents.

 \begin{example}
 \label{ex:BeilinsonMcKay}
 The \emph{Be{\u\i}linson quiver for $\mathbb{P}^n$} is the quiver with
 $n+1$ vertices $Q_0=\{0,1,\dots,n\}$ that has $n+1$ arrows from the
 $i$th vertex to the $(i+1)$-st vertex for each $i=0,1, \dots, n-1$.
 If we augment $Q$ by adding $n+1$ arrows form the $n$th vertex to the
 0th vertex then we obtain a cyclic quiver, namely the \emph{McKay
 quiver} for the group action of type $\frac{1}{n+1}(1,1,\dots,1)$ on
 $\mathbb{A}^{n+1}_\kk$. For more on these quivers, see
 Examples~\ref{ex:boundBeilinson} and \ref{ex:boundMcKay}.
 \end{example}

 The characteristic functions $\chi_{i} \colon Q_0
 \to \ZZ$ for $i \in Q_0$ and $\chi_{a} \colon Q_1 \to \ZZ$ for $a \in
 Q_1$ form the standard integral bases of the vertex space $\ZZ^{Q_0}$
 and the arrows space $\ZZ^{Q_1}$ respectively.  The \emph{weight
   lattice} of $Q$ is the sublattice of $\ZZ^{Q_0}$ given by
 \[
 \Wt(Q):= \Big{\{}\textstyle{\sum_{i \in Q_0} \theta_i \chi_i \in
 \ZZ^{Q_0} : \sum_{i \in Q_0} \theta_i
 = 0}\Big{\}}.
 \]
 Since $Q$ is connected,
 the incidence map $\inc \colon \ZZ^{Q_1} \to \ZZ^{Q_0}$ defined by
 $\inc(\chi_{a})=\chi_{\head(a)} - \chi_{\tail(a)}$ determines a short exact sequence
\begin{equation} 
  \label{eq:graphses}
  \begin{CD}   
    0 @>>> \Circuit(Q) @>>> \ZZ^{Q_1} @>{\inc}>> \Wt(Q) @>>> 0
  \end{CD}
\end{equation}
where $\Circuit(Q)\subseteq \ZZ^{Q_1}$ denotes that sublattice of
elements $f=\sum_{a\in Q_1} f_a\chi_a\in \ZZ^{Q_1}$ such that
 \[
\sum_{\{a \in Q_1 : \; \tail(a) = i\}} f_a =
\sum_{\{a \in Q_1 : \; \head(a) = i\}} f_a
 \]
 holds for each vertex $i \in Q_0$. The incidence map plays a central
 role in what follows.

 \subsection{The category of quiver representations}
 As is typical in representation theory, our interest lies not just
 with $\kk Q$, but with modules over $\kk Q$. To help us visualise
 these modules we use the terminology of quiver representations. 
 
 A \emph{representation of the quiver} $Q$ consists of a $\kk$-vector
 space $W_i$ for each $i \in Q_0$ and a $\kk$-linear map $w_a \colon
 W_{\tail(a)} \to W_{\head(a)}$ for each $a \in Q_1$. More compactly,
 we write $W = \bigl((W_i)_{i\in Q_0}, (w_a)_{a\in Q_1}\bigr)$. A
 representation is finite dimensional if each vector space $W_i$ has
 finite dimension over $\kk$, and the dimension vector of $W$ is the
 tuple of nonnegative integers $(\dim_\kk W_i)_{i\in Q_0}$. The
 \emph{support} of a representation $W$ is the quiver with vertex set
 $\{i\in Q_0 : W_0\neq 0\}$ and arrow set $\{a\in Q_1 : w_a\neq 0\}$.
 A map between representations $W$ and $W^\prime$ is a family
 $\psi_{i} \colon W_i^{\,} \to W_i^\prime$ for $i \in Q_0$ of
 $\kk$-linear maps that are compatible with the structure maps, that
 is, such that the diagrams
\[\xymatrix@C=1.2em{
W_{\tail(a)}
\ar[d]^{\psi_{\tail(a)}} \ar[rr]^-{w_{a}} && W_{\head(a)} \ar[d]^{\psi_{\head(a)}} \\
W^\prime_{\tail(a)} \ar[rr]^-{w^\prime_{a}} && W^\prime_{\head(a)} \\ }
\] 
commute for all $a\in Q_1$. With composition defined componentwise, we
obtain the category of finite-dimensional representations of $Q$,
denoted $\rep_\kk(Q)$.  

 \begin{proposition}
 \label{prop:isomcats}
 The category $\rep_\kk(Q)$ is equivalent to the category of
 finitely-generated left $\kk Q$-modules. In particular, $\rep_\kk(Q)$
 is an abelian category.
 \end{proposition}
 \begin{proof}
   Let $W = \bigl((W_i)_{i\in Q_0}, (w_a)_{a\in Q_1}\bigr)$ be a
   finite-dimensional representation of $Q$.  Define a $\kk Q$-module
   structure on the $\kk$-vector space $M:= \bigoplus_{i \in Q_0} W_i$
   by extending linearly from
 $$
e_i m = 
\begin{cases}
m & \text{for }m \in W_i, \\
0 & \text{for }m \in W_j \text{ with }j\neq i,
\end{cases} \text{and}\quad
a \cdot m = 
\begin{cases}
w_{a}(m_{\tail(a)}) & \text{for }m \in W_{\tail(a)}, \\
0 & \text{for }m \in W_j \quad\text{with }j\neq \tail(a),
\end{cases}
$$
for $i \in Q_{0}$ and $a \in Q_{1}$.  In the opposite direction,
associate to each left $\kk Q$-module $M$ the quiver representation
with $W_i:= e_i M$ for $i \in Q_0$ and maps $w_{a} \colon W_{\tail(a)}
\to W_{\head(a)}$ for $a\in Q_1$ satisfying $w_a(m) = a\cdot m$.
These operations are inverse to each other, and maps of
representations of $Q$ correspond to $\kk Q$-module homomorphisms.
 \end{proof}
 
 \begin{remark}
 \label{rem:subreps}
   This equivalence gives the notion of a subrepresentation of a
   quiver representation $W$, and the proof shows that quiver
   representations of dimension vector $(\dim_\kk W_i)_{i\in Q_0}$
   determine left $\kk Q$-modules that are isomorphic as $\bigoplus_{i\in
   Q_0} \kk e_i$-modules to $\bigoplus_{i\in Q_0}
   \big(\kk^{\dim_\kk(W_i)}\big) e_i$.
 \end{remark}

 For an algebra $A$, let $\modA$ denote the category of finitely
 generated left $A$-modules. If $A^{\opp}$ denotes the opposite
 algebra where the product satisfies $a\cdot b:= ba$, then $\modAop$
 is the category of finitely generated right $A$-modules.  If
 $Q^{\opp}$ denotes the quiver obtained from $Q$ be reversing the
 orientation of the arrows then $(\kk Q)^{\opp} \cong \kk (Q^{\opp})$.
 
 \subsection{Bound quiver representations}
 \label{sec:boundquiver}
 In most geometric contexts, the algebra of interest is not isomorphic
 to the path algebra of a quiver $Q$, but is isomorphic to the
 quotient of a path algebra by an ideal of relations
 (see Section~\ref{sec:qos} for the geometric construction). 
 
 Formally, a \emph{relation} in a quiver $Q$ (with coefficients in
 $\kk$) is a $\kk$-linear combination of paths of length at least two,
 each with the same head and the same tail.  Any finite set of
 relations $\rel$ in $Q$ determines a two-sided ideal $\langle
 \rel\rangle$ in the algebra $\kk Q$. A \emph{bound quiver}
 $(Q,\rel)$, or equivalently a \emph{quiver with relations}, is a
 quiver $Q$ together with a finite set of relations $\rel$. A
 \emph{representation of the bound quiver} $(Q,\rel)$ is a
 representation of $Q$ where each relation is satisfied in the sense
 that the corresponding $\kk$-linear combination of homomorphisms
 between the vector spaces $(W_{i})_{i\in Q_0}$ should given the zero
 map.  As before, finite-dimensional representations of $(Q,\rel)$
 form a category denoted $\rep_\kk(Q,\rel)$.  The presence of
 relations $\rel$ in a quiver $Q$ leads to a refinement of the
 equivalence of abelian categories constructed in
 Proposition~\ref{prop:isomcats}.

 \begin{proposition}
 \label{prop:boundisomcats}
 The category $\rep_\kk(Q,\rel)$ is equivalent to the category of left
   $\kk Q/\langle \rel \rangle$-modules. 
 \end{proposition}
 \begin{proof}
 See \cite[Theorem~III.1.6]{ARS} for a proof.
 \end{proof}

 \begin{example}[The bound Be{\u\i}linson quiver]
 \label{ex:boundBeilinson}
 For $n\geq 2$, the Be{\u\i}linson quiver for $\mathbb{P}^n_\kk$ described
 in Example~\ref{ex:BeilinsonMcKay} admits a set of relations that has
 geometric significance. List the $n+1$ arrows from the $i$th vertex
 to the $(i+1)$st vertex as $a_{i,0},\dots, a_{i,n}$ and define
 \[
 \rel=\big\{a_{i,j}a_{(i+1),\ell} - a_{i,\ell}a_{(i+1),j} : 0\leq j<k\leq
 n, 0\leq i\leq n-1\big\}.
 \]
 Since the length two paths $a_{i,j}a_{(i+1),\ell}$ and
 $a_{i,\ell}a_{(i+1),j}$ share tail at vertex $i$ and head at vertex
 $i+1$, the difference defines a relation for $0\leq j<k\leq n$
 and $0\leq i\leq n-1$. The quotient $\kk Q/\langle \rel\rangle$ is
 isomorphic as a $\kk$-algebra to the endomorphism algebra of the
 vector bundle $\bigoplus_{0\leq i\leq n}
 \mathscr{O}_{\mathbb{P}^n_\kk}(i)$ on $\mathbb{P}^n_\kk$.
 \end{example}

 \begin{exercise}
 \label{ex:BeilinsonAlgebra}
   Prove the assertion from Example~\ref{ex:boundBeilinson},
   namely that the quotient algebra $\kk Q/\langle \rel\rangle$ is a
   isomorphic to the endomorphism algebra of $\bigoplus_{0\leq i\leq
     n} \mathscr{O}_{\mathbb{P}^n_\kk}(i)$.
 \end{exercise}

 \begin{example}[The bound McKay quiver]
 \label{ex:boundMcKay}
   Let $G\subset \GL(n,\kk)$ be a finite abelian subgroup. Irreducible
   representations of $G$ are one-dimensional and hence define
   elements of the character group $G^*\!  :=\Hom(G,\kk^\times)$. The given
   representation then decomposes as $\mathbb{A}^n_\kk=\rho_1\oplus
   \dots \oplus\rho_n$ for some $\rho_i\in G^*$.  The \emph{McKay
     quiver} of $G\subset \GL(n,\kk)$ is the quiver with a vertex for
   each character $\rho\in G^*$, and an arrow $a_i^\rho$ from
   $\rho\rho_i:=\rho\otimes\rho_i$ to $\rho$ for each $\rho\in G^*$
   and $1\leq i \leq n$. If we label the arrow
 $a_i^\rho$ by the monomial $x_i$, then the set
 \[
 \rel=\{a_{i}^{\rho\otimes \rho_j}a_j^\rho -
 a_{j}^{\rho\otimes\rho_i}a_i^\rho : \rho\in G^*, 1\leq i,j\leq n\}
 \]
 of relations in $\kk Q$ corresponds to the condition that the
 labelling monomials commute.  We call $(Q,\rel)$ the \emph{bound
   McKay quiver} of the abelian subgroup $G\subset \GL(n,\kk)$.  It is
 well known (see Craw--Maclagan--Thomas~\cite{CMT2}) that the quotient
 algebra $\kk Q/\langle \rel\rangle$ is isomorphic to the \emph{skew
   group algebra} $\kk[x_1,\dots,x_n]*G$, i.e, to the free
 $\kk[x_1,\dots,x_n]$-module with basis $G$, where the ring structure
 is given by
 \[
 (sg) \cdot(s'g'):=s(g \cdot s') gg' 
 \]
 for $s,s'\in \kk[x_1,\dots,x_n]$ and $g,g'\in G$. Under this
 isomorphism, the semisimple subalgebra $\bigoplus_{i\in Q_0}\kk e_i$
 that plays a role in Remark~\ref{rem:subreps} is isomorphic to the
 group algebra $\kk[G]$.
 \end{example}
 
 \begin{remark}
 \label{rem:BSW}
 Since these lectures were presented,
 Bocklandt--Schedler--Wemyss~\cite{BSW} provided an explicit
 description of an ideal of relations $\varrho$ in the path algebra of
 the McKay quiver $Q$ arising from an arbitrary finite subgroup
 $G\subset \SL(n,\kk)$ for which $\kk Q/\langle \rel\rangle$ is Morita
 equivalent to the skew group algebra $\kk[x_1,\dots,x_n]*G$.
 \end{remark}

  \subsection{Toric quiver varieties}
 \label{sec:tqv}
 We now restrict attention to those representations $W$ of a finite,
 connected quiver $Q$ with dimension vector $(1,1,\dots, 1)\in
 \ZZ^{Q_0}$. Such representations correspond via the equivalence of
 categories from Proposition~\ref{prop:isomcats} to finitely generated
 left $\kk Q$-modules that are isomorphic as $\bigoplus_{i\in Q_0} \kk
 e_i$-modules to $\bigoplus_{i\in Q_0} \kk e_i$. For simplicity, we
 sometimes refer to such modules as \emph{$Q$-constellations}.

 Let $W= \bigl((W_i)_{i\in Q_0}, (w_a)_{a\in Q_1}\bigr)$ be a
representation of $Q$.  If we pick a basis for each one-dimensional
vector space $W_i$ ($0\leq i\leq r$) then the representation space is
 \[
 \mathbb{A}_\kk^{Q_1} = \Spec \bigl( \kk[y_a : a\in Q_1]
\bigr) \cong \bigoplus\nolimits_{a \in Q_1} \Hom_{\kk}(W_{\tail(a)},
W_{\head(a)}).
 \]
 The isomorphism classes of such representations are orbits by the
 action of the change of basis group $(\kk^\times)^{Q_0} \cong
 \prod_{i \in Q_0} \operatorname{GL}(W_i)$, where $t=(t_i)_{i\in Q_0}$
 acts on $w=(w_a)_{a\in Q_1}$ as
 \[
 (t \cdot w)_{a} = t_{\head(a)}^{\,} w_{a}
 t_{\tail(a)}^{-1}.
 \]
 The weights of this action are encoded by the incidence map
 \eqref{eq:graphses}, and hence the corresponding $\ZZ^{Q_0}$-grading
 of $\kk[y_a : a\in Q_1]$ is obtained from the semigroup homomorphism
 $\inc\colon \NN^{Q_1}\to \ZZ^{Q_0}$.  Since the image of the incidence map is
 the weight lattice $\Wt(Q)$, the algebra $\kk[y_a : a\in Q_1]$ admits
 only a grading by $\Wt(Q)$. Geometrically, the diagonal
 one-dimensional subtorus of $(\kk^\times)^{Q_0}$ acts trivially,
 leading to a faithful action of the quotient torus $T:= \Hom_{\ZZ}
 \bigl( \Wt(Q),\kk^\times \bigr)\cong (\kk^\times)^{Q_0}/\kk^\times$
 on $\mathbb{A}_\kk^{Q_1}$. The effective characters $\theta\in
 T^*=\Wt(Q)$ are those in the subsemigroup $\inc(\NN^{Q_1})$, and for
 any such character the GIT quotient
 \[
 \mathbb{A}^{Q_1}_\kk\git_\theta T = \Proj\Big{(}\bigoplus_{j\geq 0}
 \kk[y_a : a\in Q_1]_{j\theta}\Big{)}
 \]
 is a normal, semiprojective toric variety by
 Proposition~\ref{prop:toricGIT} and Remark~\ref{rem:normal}.

 \begin{exercise}
   Use sequence \eqref{eq:graphses} to prove that
   $\mathbb{A}^{Q_1}_\kk\git_\theta T$ is projective if and only if
   $Q$ is acyclic.
 \end{exercise}
 
 Recall from Section~\ref{sec:projectiveGIT} that a weight $\theta\in
 \Wt(Q)$ is \emph{generic} if every $\theta$-semistable point of
 $\mathbb{A}^{Q_1}$ is $\theta$-stable. Generalising only slightly the
 work of Hille~\cite{Hille}, we call
 \[
 \mathcal{M}_\theta(Q):= \mathbb{A}^{Q_1}_\kk\git_\theta T
 \]
 a \emph{toric quiver variety} whenever $\theta\in \Wt(Q)$ is
 generic. As Hille remarks, the normal toric variety
 $\mathcal{M}_\theta(Q)$ is smooth, and in our case it is projective
 over $\Spec\big(\kk[\NN^{Q_1}\cap \Circuit(Q)]\big)$.  Moreover, the
 exact sequence \eqref{eq:graphses} arising from the incidence map of
 $Q$ coincides with the sequence \eqref{eq:fundiagram} from the Cox
 construction of $\mathcal{M}_\theta(Q)$ as a toric variety, giving
 $\Wt(Q)=\Pic(\mathcal{M}_\theta(Q))$.

 \begin{example}
 \label{ex:quiver3vertices}
   For the quiver $Q$ from Figure~\ref{fig:quiver3vertices}, consider
   the weight $\vartheta:=(-2,1,1)\in \Wt(Q)$.
  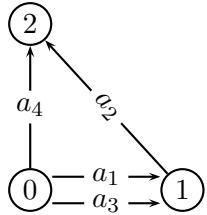
\begin{figure}[!ht]
    \centering
    \mbox{
        \psset{unit=1cm}
        \begin{pspicture}(-0.5,-0.5)(2.5,2.4)
          \cnodeput(0,0){A}{0}
          \cnodeput(2,0){B}{1} 
          \cnodeput(0,2.2){C}{2}
          \psset{nodesep=0pt}
          \ncline[offset=5pt]{->}{A}{B} \lput*{:U}{$a_1$}
          \ncline[offset=-5pt]{->}{A}{B} \lput*{:U}{$a_3$}
          \ncline{->}{A}{C} \lput*{:270}{$a_4$}
          \ncline{<-}{C}{B} \lput*{:U}{$a_2$}
        \end{pspicture}
    }
    \caption{A quiver defining the Hirzebruch surface $\mathbb{F}_{1}$ \label{fig:quiver3vertices}}
  \end{figure}
  The algebra $\bigoplus_{j\geq 0} \kk[y_1,y_2,y_3,y_4]_{j\vartheta}$
  is the semigroup algebra of $S_\vartheta = \bigoplus_{j\geq 0}
  \big(\NN^4\cap \inc^{-1}(j\vartheta)\big)$ which is generated in
  degree $j=1$. The computation of the set $\NN^4\cap
  \inc^{-1}(\vartheta)$ should be familiar from
  Example~\ref{ex:F1fan}. Indeed, the incidence matrix of the quiver
  $Q$ is obtained by adding a row at the top of the mztrix from that
  example, where the new entry in each column is added to ensure that
  the column sum is zero. The parameter $\chi=(1,1)$ from
  Example~\ref{ex:F1fan} corresponds to $\vartheta=(-2,1,1)$ here, and
  hence $\mathcal{M}_\vartheta(Q)\cong \mathbb{F}_1$.
 \end{example}


 \begin{exercise}
 \label{ex:challenge}
   For the weight $\vartheta=(-2,1,1)$, identify the toric quiver
   variety defined by the quiver $Q$ with incidence matrix
 \[
 \begin{pmatrix}
 -1 & -1 & -1 & 0 & 0 \\ 1 & 1 & 0 & -1 & -1 \\ 0 & 0 & 1 & 1 & 1 
 \end{pmatrix}.
 \]
 [The solution to this exercise does appear later in these notes.]
 \end{exercise}

 \subsection{Moduli spaces of $\theta$-stable representations}
 \label{sec:thetastability}
 The importance of toric quiver varieties $\mathcal{M}_\theta(Q)$
 becomes apparent only after establishing that each represents a
 functor. To describe this construction we first recall King's notion
 of $\theta$-stability. Let $\theta\in T^*\otimes_\ZZ \QQ$ be a
 fractional character that is not necessarily generic. A
 representation $W$ of $Q$ with dimension vector $(1,\dots,1)\in
 \ZZ^{Q_0}$ is said to be \emph{$\theta$-semistable} if $\theta(W)=0$
 and if for every proper, nonzero subrepresentation $W^\prime \subset
 W$ we have $\theta(W^\prime) \geq 0$.  The notion of
 \emph{$\theta$-stability} is obtained by replacing $\geq$ with $>$.
 This translates immediately via the equivalence of categories from
 Proposition~\ref{prop:isomcats} into a notion of
 $\theta$-(semi)stability for $Q$-constellations.

 \begin{exercise}
 \label{ex:generic}
 Let $Q$ be a quiver with a distinguished source, denoted $0\in Q_0$,
 that admits a path to every other vertex in $Q$.  Consider any weight
 $\vartheta\in \Wt(Q)\otimes_\ZZ \QQ$ with $\vartheta_i>0$ for
 $i\neq 0$ (and hence $\vartheta_0<0$).  Show that:
 \begin{enumerate} 
 \item[\one] every $\vartheta$-semistable representation is
 $\vartheta$-stable.
\item[\two] a $Q$-constellation is $\vartheta$-stable if and only if
  it is a cyclic $\kk Q$-module with generator $e_0$.
 \end{enumerate}
 \end{exercise}
 
 An abelian category has finite length if every object $W$ has a
 Jordan-H{\"o}lder filtration
 \[
 0=E_0\subset E_1\subset \cdots\subset E_{n-1}\subset E_n=W
 \]
 such that each factor object $F_i=E_i/E_{i-1}$ is simple. The set of
 $\theta$-semistable representations form the objects in a
 finite-length abelian subcategory of $\rep_\kk(Q)$, where the simple
 objects are precisely the $\theta$-stable representations. Moreover,
 two $\theta$-semistable representations of $Q$ are said to be
 \emph{$S$-equivalent} (with respect to $\theta$) if their
 Jordan--H\"{o}lder filtrations have isomorphic composition factors.

 \begin{proposition}
  \label{prop:King}
 Let $W$ be a representation of $Q$ defining the point $w\in
   \mathbb{A}^{Q_1}_\kk$ in the space of representations and let
   $\theta\in \Wt(Q)\otimes_\ZZ \QQ$. Then:
 \begin{enumerate}
 \item[\one] the point $w\in \mathbb{A}^{Q_1}_\kk$ is
   $\theta$-semistable in the sense of GIT if and only if the
   representation $W$ is $\theta$-semistable in the sense defined above; and
 \item[\two] two $\theta$-semistable points of $\mathbb{A}^{Q_1}_\kk$
   are identified under the map $\pi\colon
   (\mathbb{A}^{Q_1})^{\textrm{ss}}_\theta\to
   \mathbb{A}^{Q_1}\git_\theta T$ from Diagram~\ref{diag:GITcharacter}
   if and only if the corresponding representations are $S$-equivalent
   (w.r.t.\ $\theta$).
 \end{enumerate}
 \end{proposition}

 Combining this result with Theorem~\ref{thm:projectiveGIT} gives the
 following.

 \begin{corollary}
 The toric variety $\mathcal{M}_\theta(Q)$ parametrises
 $S$-equivalence classes of $\theta$-semistable representations of
 $Q$. If $\theta$ is generic then $\mathcal{M}_\theta(Q)$ parametrises
 isomorphism classes of $\theta$-stable representations of $Q$.
 \end{corollary}

 For the moduli description of $\mathcal{M}_\theta(Q)$, we choose our
 favourite vertex of the quiver $Q$ and denote it $0\in Q_0$; at
 present this looks arbitrary, but we shall see in
 Section~\ref{sec:qos} that a natural candidate often presents
 itself. Consider the functor $\rep_\theta(Q)(-)$ that assigns to each
 connected scheme $B$ the set of locally-free sheaves on $B$ of the
 form
 \[
 \left\{\bigoplus_{i \in Q_0} \mathscr{W}_i : \begin{array}{c}
 \rank(\mathscr{W}_i)=1 \;\;\forall\; i\in Q_0, \text{ with
 }\mathscr{W}_0\cong \mathscr{O}_B\\ \exists \text{ morphism
 }\mathscr{W}_{\tail(a)}\to \mathscr{W}_{\head(a)}\;\;\forall\; a\in
 Q_1 \\ \text{each fibre }\bigoplus_{i\in Q_0}W_i \text{ is
 }\theta\text{-semistable (up to
 }S\text{-equiv)}\end{array}\right\}/\sim_{\text{isom}};
 \]
 the requirement on the existence of the morphisms
 $\mathscr{W}_{\tail(a)}\to \mathscr{W}_{\head(a)}$ for all $a\in Q_1$ is
 equivalent to the existence of a $\kk$-algebra homomorphism $\kk Q\to
 \End\big(\bigoplus_{i \in Q_0} \mathscr{W}_i\big)$.

\begin{theorem}[King~\cite{King1}]
 \label{thm:Kingmoduli}
 If $\theta$ is generic then $\mathcal{M}_\theta(Q)$ represents the
 functor $\rep_\theta(Q)$, i.e., $\mathcal{M}_\theta(Q)$ is the
 \emph{fine moduli space of $\theta$-stable representations of
 $Q$}. In particular, $\mathcal{M}_\theta(Q)$ carries 
 \begin{enumerate}
 \item[\one] a tautological
 locally free sheaf $\bigoplus_{i \in Q_0} \mathscr{W}_i$ with $\rank(\mathscr{W}_i)=1$ and $\mathscr{W}_0\cong
\mathscr{O}_{\mathcal{M}_\theta(Q)}$; and 
 \item[\two] a $\kk$-algebra homomorphism $\kk Q\to
 \End\big(\bigoplus_{i \in Q_0} \mathscr{W}_i\big)$.
 \end{enumerate}
\end{theorem}
\begin{sketch}
 Since $\theta$ is generic, the $S$-equivalence classes of
 $\theta$-semistable representations coincide with the isomorphism
 classes of $\theta$-stable representations. Write $\bigoplus_{i\in
 Q_0}\mathscr{O}_{\mathbb{A}^{Q_1}}$ for the universal locally free
 sheaf on the representation space $\mathbb{A}^{Q_1}_\kk$ whose fibre
 at a point is the corresponding representation of $Q$. One expects
 that the restriction of $\bigoplus_{i\in
 Q_0}\mathscr{O}_{\mathbb{A}^{Q_1}}$ to the $\theta$-stable locus in
 $\mathbb{A}^{Q_1}_\kk$ will descend under the geometric quotient map
 $(\mathbb{A}^{Q_1}_\kk)^{\textrm{s}}_\theta\to
 (\mathbb{A}^{Q_1}_\kk)^{\textrm{s}}_\theta/T$ to give the required
 tautological sheaf. This is not immediate since the change of basis
 group $(\kk^\times)^{Q_0}$ does not act faithfully on
 $\mathbb{A}^{Q_1}_\kk$. However, our choice of the distinguished
 vertex $0\in Q_0$ and our condition $\mathscr{W}_0\cong
 \mathscr{O}_B$ on the functor $\rep_\theta(Q)(-)$ removes this
 ambiguity, at the expense of imposing the restriction
 $\mathscr{W}_0\cong \mathscr{O}_{\mathcal{M}_\theta(Q)}$ on the
 resulting tautological bundle.
\end{sketch}

 \subsection{Moduli for bound quiver representations}
 \label{sec:moduliboundquiver}
 The GIT and moduli constructions for representations of $Q$ lead
naturally to similar constructions for representations of a bound
quiver $(Q,\rel)$ as follows. 

 Let $(Q,\rel)$ be a bound quiver (see Section~\ref{sec:boundquiver}).
The map sending a path $p = a_1 \dotsb a_j$ to the monomial $y_{a_1}
\dotsb y_{a_j} \in \kk[y_a : a\in Q_1]$ extends to a $\kk$-linear map
from $\kk Q$ to $\kk[y_a : a\in Q_1]$, and let $\Irel$ denote the ideal
in $\kk[y_a : a\in Q_1]$ generated by the image of $\langle \rel\rangle$
under this map.  A point in $\mathbb{A}_\kk^{Q_1}$ corresponds to a
representation of the bound quiver $(Q,\rel)$ if and only it lies in the
subscheme $\mathbb{V}(\Irel)$ cut out by $\Irel$. Since $\rel$ consists of
$\kk$-linear combinations of paths with the same head and tail, the
ideal $\Irel$ is homogeneous with respect to the $\Wt(Q)$-grading of
$\kk[y_a : a\in Q_1]$ and hence the subscheme $\mathbb{V}(\Irel)$ is
invariant under the action of $T$. For $\theta \in
\Wt(Q)\otimes_{\ZZ}\QQ$, let $(\kk[y_a : a\in Q_1]/\Irel)_{\theta}$
denote the $\theta$-graded piece. Then
 \begin{equation}
 \label{eqn:quiverGIT}
 \mathbb{V}(\Irel) \git_{\theta}T = \Proj \Big(\bigoplus_{j \in \NN} \bigl( \kk[y_a : a\in Q_1]/\Irel \bigr)_{j\theta} \Big)
 \end{equation}
 is the categorical quotient of the open subscheme
 $\mathbb{V}(\Irel)^{\text{ss}}_\theta \subseteq \mathbb{V}(\Irel)$
 parametrising $\theta$-semistable representations of $(Q,\rel)$.

 \begin{example}
 \label{exa:reducible}
 The ideal of relations $\Irel\subset \kk[y_a : a\in Q_1]$ need not be
 prime, so the scheme $\mathbb{V}(\Irel)$ may be reducible. For
 example, let $(Q,\rel)$ be the bound McKay quiver (see
 Example~\ref{ex:boundMcKay}) for the finite group action of type
 $\frac{1}{7}(1,2)$. Denote the arrows by $a_i^{\rho_j}$ for $1\leq
 i\leq 2$ and $0\leq j\leq 6$, so the representation space
 $\mathbb{A}^{14}_\kk$ has coordinate ring $\kk[y_i^{\rho_j}:1\leq
 i\leq 2, 0\leq j\leq 6]$. The set of relations $\rel$ from
 Example~\ref{ex:boundMcKay} defines the ideal
 \[ 
\Irel=\big\langle z_2^{\rho_j}z_1^{\rho_{j-1}} -
z_1^{\rho_{j+1}}z_2^{\rho_{j-1}} : 0\leq j\leq 6 \big\rangle,
\]
 where addition on the indices of exponents is taken modulo 7. This
 ideal has 8 associated primes, so $\mathbb{V}(\Irel)$ has 8
 irreducible components.
 \end{example}

 \begin{remark}
   The brief introduction to GIT in Section~\ref{sec:GIT} constructed
   quotient spaces arising from a diagonalisable group action on an
   affine variety. However, if the scheme $\mathbb{V}(\Irel)$ is
   reducible then the GIT quotient $\mathbb{V}(\Irel) \git_{\theta}T$
   may also be reducible.  This highlights a crucial difference
   between toric quiver varieties $\mathbb{A}_\kk^{Q_1}\git_\theta T$
   and the subschemes $\mathbb{V}(\Irel) \git_{\theta}T$. The former
   are normal semiprojective toric varieties, but the latter need not
   a priori be toric as Example~\ref{exa:reducible} shows. In fact,
   since the ideal $\Irel\subset \kk[y_a : a\in Q_1]$ is generated by
   binomials, Eisenbud--Sturmfels~\cite{EisenbudSturmfels} shows that
   each irreducible component of the scheme $\mathbb{V}(\Irel)$ is an
   affine toric variety, so each irreducible component of the GIT
   quotient $\mathbb{V}(\Irel)\git_\theta T$ is a semiprojective toric
   variety.  These components need not be normal as
   Example~\ref{exa:notnormal} demonstrates.
 \end{remark}

 If $\theta$ is generic then we write
 \[
 \mathcal{M}_\theta(Q,\rel):=
 \mathbb{V}(\Irel) \git_{\theta}T
 \]
 for the geometric quotient of the open subscheme
 $\mathbb{V}(\Irel)^{\text{s}}_\theta \subseteq \mathbb{V}(\Irel)$
 parametrising $\theta$-stable representations of $(Q,\rel)$ by the
 action of the torus $T=\Hom_\ZZ(\Wt(Q),\kk^\times)$. The
 analogue of Theorem~\ref{thm:Kingmoduli} holds for bound quiver
 representations, making $\mathcal{M}_\theta(Q,\rel)$ the \emph{fine
 moduli space of $\theta$-stable representations of $(Q,\rel)$}.

 \begin{exercise}
 \label{ex:boundfunctor}
 Establish the bound quiver analogue of Theorem~\ref{thm:Kingmoduli}
 by introducing a functor $\rep_\kk(Q,\varrho)$, and hence give the
 fine moduli space interpretation of $\mathcal{M}_\theta(Q,\rel)$.
 \end{exercise}

\begin{example}
 \label{exa:notnormal}
 Let $(Q,\rel)$ denote the bound McKay quiver associated to a finite
 abelian subgroup $G\subset \GL(n,\kk)$ as constructed in
 Example~\ref{ex:boundMcKay}, so $\kk Q/\langle \rel\rangle$ is
 isomorphic to the skew group algebra $\kk[x_1,\dots,x_n]*G$ and
 $\bigoplus_{i\in Q_0} \kk e_i$ is isomorphic to the group algebra
 $\kk[G]$. Write $0\in Q_0$ for the vertex arising from the trivial
 representation of $G$, and let $\vartheta\in \Wt(Q)$ be any weight
 satisfying $\vartheta_i>0$ for $i\neq 0$.
 Exercises~\ref{ex:generic}\one\ and \ref{ex:boundfunctor} show that
 $\mathcal{M}_\vartheta(Q,\rel)$ is the fine moduli space of
 $\vartheta$-stable $\kk[x_1,\dots,x_n]*G$-modules that are isomorphic
 as $\kk[G]$-modules to $\kk[G]$; in this context, such modules are
 called $\vartheta$-stable \emph{$G$-constellations} (see
 Craw~\cite{Crawthesis}). Craw--Maclagan--Thomas~\cite[Examples~4.12
 and~5.7]{CMT2} show that the scheme $\mathcal{M}_\vartheta(Q,\rel)$
 is:
 \begin{enumerate}
 \item[\one] reducible for the subgroup of $\GL(3,\kk)$ that defines
   the action of type $\frac{1}{14}(1,9,11)$; and
 \item[\two] not normal for a subgroup of $\GL(6,\kk)$ that is
   isomorphic to $(\ZZ/5)^{4}$.
 \end{enumerate}
 \end{example}

 \begin{remark}[The $G$-Hilbert scheme]
 \label{rem:G-Hilbert}
 For the bound McKay quiver $(Q,\rel)$ associated to a finite abelian
 subgroup $G\subset \GL(n,\kk)$, and for any weight $\vartheta\in
 \Wt(Q)$ satisfying $\vartheta_i>0$ for $i\neq 0$, the scheme
 $\mathcal{M}_\vartheta(Q,\rel)$ represents the same functor as the
 $G$-Hilbert scheme $\ghilb(\mathbb{A}^n_\kk)$.  Indeed, a
 $\kk[x_1,\dots,x_n]$-module $M$ is $G$-equivariant if and only if it
 carries a $G$-action such that
 \[
 g \cdot (sm)=(g \cdot s) (g \cdot m)
 \]
 for $g \in G, s \in \kk[x_1,\dots,x_n]$ and $m \in M$, from which it
 follows that left $\kk[x_1,\dots,x_n]*G$-modules are precisely
 $G$-equivariant $\kk[x_1,\dots,x_n]$-modules (compare also
 Lemma~\ref{lem:abelianequivs}). If we consider only
 $G$-constellations, then Exercise~\ref{ex:generic}\two\ asserts that
 a $G$-constellation is $\vartheta$-stable if and only it is a cyclic
 $\kk[x_1,\dots,x_n]*G$-module with generator the trivial
 representation. Combining these observations, we see that
 $\mathcal{M}_\vartheta(Q,\rel)$ is isomorphic to the
 \emph{$G$-Hilbert scheme}
  \[
 \ghilb(\mathbb{A}^n_\kk):= \Big\{G\text{-equivariant quotient modules
 }M = \kk[x_1,\dots,x_n]/I \text{ with }
 M \cong_{\kk[G]}\kk[G]\Big\}
 \]
 that parametrises (coordinate rings of) $G$-homogeneous ideals
 $I\subset \kk[x_1,\dots,x_n]$ for which the quotient module
 $\kk[x_1,\dots,x_n]/I$ is isomorphic as a $\kk[G]$-module to
 $\kk[G]$. Moreover, the universal bundles for each moduli space are
 isomorphic as $G$-equivariant locally free sheaves. This moduli space
 (and its nonabelian analogue) plays the key role in
 Section~\ref{sec:McKay}.
\end{remark}

 \section{Noncommutative construction of toric varieties}
 \label{sec:noncommutative}
 This section demonstrates how quivers arise naturally in the study of
 normal toric varieties, and provides a quiver-theoretic (and hence
 noncommutative) construction of normal semiprojective toric
 varieties.  Following Craw--Smith~\cite{CrawSmith}, we extend the
 classical notion of the linear series $\vert L\vert $ of a single
 line bundle $L$, obtaining the multilinear series $\vert
 \mathcal{L}\vert $ of a collection of line bundles
 $\mathcal{L}=(\mathscr{O}_X,L_1,\dots, L_r)$ on a normal
 semiprojective toric variety $X$. In particular, we show that every
 such toric variety is a fine moduli space of bound quiver
 representations for many sequences of line bundles.

 \subsection{Quivers of sections}
 \label{sec:qos}
 Let $X$ be a normal semiprojective toric variety with fan $\Sigma$
 and dense torus $T_X$.  For $r\geq 1$, let $\mathcal{L}:=
 (L_0,\dotsc, L_r)$ be a sequence of distinct, basepoint-free line
 bundles on $X$ with $L_0:=\mathscr{O}_X$. A $T_X$-invariant section
 $s$ of $L_j^{\;} \otimes L_{i}^{-1}$ is \emph{indecomposable} if the
 divisor of zeros $\div(s)$ cannot be expressed as $\div(s') +
 \div(s'')$ where $s'$ and $s''$ are nonzero $T_X$-invariant sections
 of $L_\ell^{\;} \otimes L_{i}^{-1}$ and $L_j^{\;} \otimes
 L_{\ell}^{-1}$ respectively, and $0 \leq \ell \leq r$.  The complete
 \emph{quiver of sections} of the sequence $\mathcal{L}$ is the quiver $Q$
 whose vertices $Q_0 = \{ 0, \dotsc, r \}$ correspond to the line
 bundles in $\mathcal{L}$, and the arrows from $i$ to $j$ correspond
 to the indecomposable $T_X$-invariant elements of $H^0(X,L_j^{\;}
 \otimes L_{i}^{-1})=\Hom(L_i,L_j)$. 

 \begin{lemma}
   Let $Q$ be a complete quiver of sections on $X$. For each $i\in
   Q_0$, there is a path in $Q$ from vertex 0 to vertex $i$. Moreover,
   $Q$ is acyclic if and only if $X$ is projective.
 \end{lemma}
 \begin{proof}
   The first statement holds since each $L_i$ is effective. For the
   second, write $\nu\colon S\to \ZZ$ for the $\ZZ$-graded semigroup
   giving $X=\Proj\big(\bigoplus_{j\geq 0} \kk[S]_j\big)$, so $X$ is
   projective over $\Spec\big(\kk[S]_0\big)$. It follows for each
   $0\leq i\leq r$ that $H^0(L_i\otimes L_i^{-1}) \cong \kk[S]_0$, so
   directed cycles in $Q$ based at vertex $i\in Q_0$ correspond to
   elements of the $\kk$-vector space basis $S_0:=S\cap\Ker(\nu)$ of
   the algebra $\kk[S]_0$. In particular, $Q$ is acyclic if and only
   if $\kk[S]_0= \kk$.
  \end{proof}

 Each quiver of sections carries naturally a set of relations.  For
 each arrow $a\in Q_1$ we write $\div(a)$ for the divisor of zeros of
 the section $s \in H^0(X,L_j^{\;} \otimes L_{i}^{-1})$ that defines
 $a$. Extend this to paths $p = a_1 \dots a_\ell$ in
 $Q$ by setting $\div(p) := \sum_{j=1}^\ell \div(a_j)$, and hence
 define
 \begin{equation}
 \label{eqn:varrho}
 \varrho = \big\{ p-p'\in \kk Q : \tail(p) = \tail(p'),\head(p) =
 \head(p'), \div(p) = \div(p')\big\}.
 \end{equation}
 Only indecomposable sections determine arrows in $Q$, so the paths
 whose differences form the elements of $\varrho$ each comprise at
 least two arrows.  We call $(Q,\varrho)$ is the \emph{complete bound
   quiver of sections} for $\mathcal{L}$. The significance of the
 relations $\varrho$ is evident from the following result.

 \begin{proposition}
 \label{prop:algebra}
 The quotient algebra $\kk Q/\langle\varrho\rangle$ is isomorphic to
 $\End\big( \bigoplus_{i\in Q_0} L_i \big)$.
 \end{proposition}
 \begin{proof}
   Each path in $Q$ arises from an element of $\Hom(L_i,L_j)$ for some
   $i,j\in Q_0$.  The assignment sending each path to the
   corresponding homomorphism of line bundles extends to a surjective
   homomorphism of $\kk$-algebras $\eta \colon \kk Q \rightarrow
 \End\big(\bigoplus_{i\in Q_0} L_i\big)$ with kernel $\langle \varrho\rangle$.
 \end{proof}

 \begin{exercise}
  \label{ex:Beilinsoncqos}
 Prove that the bound Be{\u\i}linson quiver from
   Example~\ref{ex:boundBeilinson} is the complete bound quiver of
   sections of
   $\mathcal{L}=\big(\mathscr{O}_{\mathbb{P}^n_\kk},\mathscr{O}_{\mathbb{P}^n_\kk}(1),\dots,
   \mathscr{O}_{\mathbb{P}^n_\kk}(n)\big)$.
 \end{exercise}
  
 In light of Exercise~\ref{ex:Beilinsoncqos},
 Proposition~\ref{prop:algebra} provides a solution to
 Exercise~\ref{ex:BeilinsonAlgebra}.

 \begin{example}
 \label{exa:qos1/3(1,2)}
 Let $X$ denote the minimal resolution of the singularity of type
 $\frac{1}{3}(1,2)$ from Example~\ref{exa:res1/3(1,2)}. The bound
 McKay quiver $(Q,\varrho)$ for the action of type $\frac{1}{3}(1,2)$
 is the complete quiver of sections of
 $\mathcal{L}=(\mathscr{O}_X,L_1,L_2)$ on $X$, where $L_1 =
 \mathscr{O}_X(D_0)$ and $L_2=\mathscr{O}_X(D_3)$ form the $\ZZ$-basis
 of $\Pic(X)$ dual to the classes of the irreducible exceptional
 curves.
 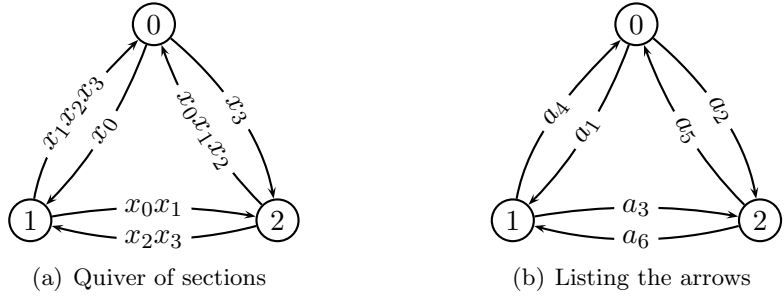
\begin{figure}[!ht]
    \centering \mbox{ \subfigure[Quiver of sections]{
    \psset{unit=1.3cm}
        \begin{pspicture}(-0.5,-0.25)(3,2.25)
        \cnodeput(0,0){A}{1} \cnodeput(2.5,0){B}{2}
        \cnodeput(1.25,2){C}{0} \psset{nodesep=0pt}
        \nccurve[angleA=10,angleB=170]{->}{A}{B}\lput*{:U}{$x_0x_1$}
        \nccurve[angleA=345,angleB=195]{<-}{A}{B}\lput*{:U}{$x_2x_3$}
        \nccurve[angleA=250,angleB=45]{->}{C}{A}\lput*{:180}{$x_0$}
        \nccurve[angleA=225,angleB=80]{<-}{C}{A}\lput*{:180}{$x_1x_2x_3$}
        \nccurve[angleA=320,angleB=100]{->}{C}{B}\lput*{:U}{$x_3$}
        \nccurve[angleA=290,angleB=135]{<-}{C}{B}\lput*{:U}{$x_0x_1x_2$}
        \end{pspicture}}
      \qquad \qquad
      \subfigure[Listing the arrows]{
        \psset{unit=1.3cm}
        \begin{pspicture}(-0.5,-0.25)(3,2.25)
        \cnodeput(0,0){A}{1}
        \cnodeput(2.5,0){B}{2} 
        \cnodeput(1.25,2){C}{0}
        \psset{nodesep=0pt}
        \nccurve[angleA=10,angleB=170]{->}{A}{B}\lput*{:U}{$a_3$}
        \nccurve[angleA=345,angleB=195]{<-}{A}{B}\lput*{:U}{$a_6$}
        \nccurve[angleA=250,angleB=45]{->}{C}{A}\lput*{:180}{$a_1$}
        \nccurve[angleA=225,angleB=80]{<-}{C}{A}\lput*{:180}{$a_4$}
         \nccurve[angleA=320,angleB=100]{->}{C}{B}\lput*{:U}{$a_2$}
        \nccurve[angleA=290,angleB=135]{<-}{C}{B}\lput*{:U}{$a_5$}
        \end{pspicture}}
    }
    \caption{McKay quiver for the action of type $\frac{1}{3}(1,2)$ \label{fig:qos1/3(1,2)}}
  \end{figure}
  Figure~\ref{fig:qos1/3(1,2)} shows the quiver, where each arrow
  $a\in Q_1$ is labelled with the monomial $x^{\div(a)}\in R_X$.
  Using sequence \eqref{eq:fundiagramres1/3(1,2)} from
  Example~\ref{exa:res1/3(1,2)}, we see that the sections
  $x_0,x_2x_3^2\in H^0(X,L_1)$ correspond to the pair of vertices of
  the polyhedron $\conv(\NN^4\cap \pi^{-1}(L_1))$, while
  $x_3,x_0^2x_1\in H^0(X,L_2)$ correspond to the pair of vertices of
  $\conv(\NN^4\cap \pi^{-1}(L_2))$. Clearly $x_0\in H^0(X,L_1)$ is
  indecomposable, while $x_2x_3^2\in H^0(X,L_1)$ decomposes as the
  product of $x_3\in H^0(L_2)$ and $x_2x_3\in H^0(L_1\otimes
  L_2^{-1})$, and similarly for the section of $L_2$. If we list the
  arrows as shown in Figure~\ref{fig:qos1/3(1,2)}(b), then the relations
  $\varrho=\{a_1a_4 - a_{2}a_5, a_{3}a_6 -
  a_{4}a_1, a_{5}a_2 - a_{6}a_3\}$ each derive from
  decomposing the section $x_0x_1x_2x_3\in H^0(\mathscr{O}_X)$ in two
  different ways when viewed as a cycle from vertex $i\in Q_0$.
 \end{example}

 \begin{exercise}
   \label{ex:qos1/r(1,r-1)}
   Let $X$ denote the minimal resolution of the singularity of type
   $\frac{1}{r}(1,r-1)$. By first generalising
   Example~\ref{exa:res1/3(1,2)}, show that the bound McKay quiver for
   the action of type $\frac{1}{r}(1,r-1)$ is the complete quiver of
   sections of $\mathcal{L}=(\mathscr{O}_X,L_1,\dots,L_{r-1})$ on $X$,
   where $L_1,\dots,L_{r-1}$ form the $\ZZ$-basis of $\Pic(X)$ dual to
   the classes of the irreducible exceptional curves.
  \end{exercise}

 \begin{remark}
   Exercise~\ref{ex:qos1/r(1,r-1)} can be generalised to the case
   $\frac{1}{r}(1,a)$ where $\gcd(a,r)=1$. Let $X$ denote the minimal
   resolution of the singularity of type $\frac{1}{r}(1,a)$ and let
   $L_1,\dots, L_m$ be the $\ZZ$-basis of $\Pic(X)$ dual to the
   irreducible exceptional curves in $X$. The complete bound quiver of
   sections $(Q,\varrho)$ for
   $\mathcal{L}=(\mathscr{O}_X,L_1,\dots,L_m)$ is the \emph{bound
   Special McKay quiver} of type $\frac{1}{r}(1,a)$ introduced in
   Craw~\cite{CrawSM} and the quotient algebra $\kk Q/\langle
   \varrho\rangle$ is the \emph{reconstruction algebra} of
   Wemyss~\cite{Wemyss}.
 \end{remark}

 \subsection{Multilinear series}
 Since the vertex $0\in Q_0$ in any complete quiver of sections $Q$
 provides a distinguished source admitting a path to every other
 vertex in $Q$, Exercise~\ref{ex:generic} shows that any weight
 $\vartheta\in \Wt(Q)\otimes_\ZZ \QQ$ satisfying $\vartheta_i>0$ for
 $i\neq 0$ (and hence $\vartheta_0<0$) is generic.  In this case, we
 write 
 \[
 \vert\mathcal{L}\vert:=
 \mathcal{M}_\vartheta(Q)=\mathbb{A}^{Q_1}\git_\vartheta T
 \]
 for the fine moduli space of $\vartheta$-stable representations of
 $Q$ of dimension vector $(1,\dots,1)\in \ZZ^{Q_0}$, and call this
 toric variety the \emph{multilinear series} of the sequence
 $\mathcal{L}$. Results from Section~\ref{sec:tqv} show that
 $\vert\mathcal{L}\vert$ is smooth and projective over
 $\Spec\big(\kk[\NN^{Q_1}\cap \Circuit(Q)]\big)$, and
 $\Pic(\vert\mathcal{L}\vert)\cong \Wt(Q)$. Moreover, our
 normalisation for the tautological bundle on the moduli space
 $\mathcal{M}_\vartheta(Q)$ (see Section~\ref{sec:thetastability})
 implies that the isomorphism $\Wt(Q)\to \Pic(\vert\mathcal{L}\vert)$
 identifies $\sum_{i\in Q_0} \theta_i\chi_i$ with $\bigotimes_{i\in
   Q_0} \mathscr{W}_i^{\otimes \theta_i}$. This leads to the
 following mild generalisation of Craw--Smith~\cite[Proposition~3.8]{CrawSmith}
 from the case when $X$ is a normal projective toric variety:

 \begin{proposition}
  \label{pro:smooth}
  Let $Q$ be the complete quiver of sections for a sequence $\mathcal{L}$
  as above on a normal semiprojective toric variety $X$. The multilinear series
  $\vert\mathcal{L}\vert$ is a smooth semiprojective toric variety
  obtained as the geometric quotient of $\mathbb{A}^{Q_1} \setminus
  \mathbb{V}(B_Q)$ by $T$ where
    \[
    B_Q:= \Biggl( \prod\limits_{a \in Q'_1} y_a : \text{$Q'$ is a
      spanning tree of $Q$ rooted at $0$} \Biggl).
    \]
    Moreover, the tautological line bundles $\mathscr{W}_i$ for $i\in
    Q_0$ on the fine moduli space $\vert \mathcal{L}\vert =
    \mathcal{M}_\vartheta(Q)$ correspond to the elements
    $\chi_i-\chi_0$ for $i\in Q_0$ in $\Wt(Q)$ under the isomorphism
    decribed above.
 \end{proposition}
 \begin{proof}
 It remains to establish that the $\vartheta$-unstable locus is
   $\mathbb{V}(B_Q)$. The irrelevant ideal for the quotient
   $\mathbb{A}^{Q_1}\git_\vartheta T$ is generated by those monomials
   $y^u\in \kk[y_a : a\in Q_1]$ with $u\in \NN^{Q_1}\cap
   \inc^{-1}(\vartheta)$. If we set $\vartheta:= (-r,1,1,\dots,1)\in
   \Wt(Q)$, then the set of arrows supporting every such exponent $u$
   also supports paths from the distinguished vertex 0 to every other
   vertex $i\in Q_0\setminus \{0\}$. Forgetting multiplicities and
   working only with the supporting arrows is equivalent to taking the
   radical of the ideal. Therefore, as in the proof of
   Theorem~\ref{thm:Cox}, the ideal $B_Q$ is the radical of the ideal
   that cuts out the $\vartheta$-unstable locus in $\mathbb{A}^{Q_1}$.
 \end{proof} 

\begin{example}
  \label{exa:2pt}
  Let $L_1$ be a basepoint-free line bundle on $X$ and set
  $\mathcal{L} = (\mathscr{O}_{X}, L_1)$. The complete quiver of
  sections of $\mathcal{L}$ has arrow set $Q_1 = \{ a_0, \dotsc, a_m
  \}$ corresponding to a $\kk$-vector space basis of $H^0(X,L_1)$.
  Since every arrow forms a spanning tree in $Q$, we have $B_Q= (
  y_{a_0}, \dotsc, y_{a_{m}})$ and hence, $\vert \mathcal{L}\vert$ is
  the geometric quotient of $\mathbb{A}^{m+1} \setminus \{ 0 \}$ by $T
  := \Hom_{\ZZ}\bigl( \Wt(Q), \kk^{*})$.  Choosing $\chi_1 - \chi_0$
  as a basis for $\Wt(Q)\cong \ZZ$, the action of $T \cong \kk^\times$
  on $\mathbb{A}^{m+1}$ is induced by 
 \[
  \begin{CD}   
    0 @>>> \Circuit(Q) @>>> \ZZ^{m+1} @>{\left[
    \begin{smallmatrix} 1 & \dotsb & 1 \end{smallmatrix} \right]}>> \ZZ @>>> 0.
  \end{CD}
 \]
 It follows that $|\mathcal{L}|$ is isomorphic to the classical linear
 series $|L_1|\cong\mathbb{P}^m_\kk$; whence the name. In our case,
 the tautological bundle is $L_1$ rather than $L_1^{-1}$; we view
 projective space as parametrising hyperplanes rather than lines.
\end{example}

\begin{example}
  \label{exa:F1asmultilinear}
  For $X = \mathbb{F}_1$, consider the sequence $\mathcal{L} = \bigl(
  \mathscr{O}_X, \mathscr{O}_{X}(D_1), \mathscr{O}_{X}(D_4) \bigr)$
  where we adopt the notation of Example~\ref{ex:fundiagramF1}. The
  complete quiver of sections for $\mathcal{L}$ appears in
  Figure~\ref{fig:quiver3vertices}, and
  Example~\ref{ex:quiver3vertices} shows that the
  $\vert\mathcal{L}\vert$ is isomorphic to $\mathbb{F}_1$. In this
  case, the tautological line bundles coincide with the bundles of
  $\mathcal{L}$.
\end{example}

\begin{example}
  \label{ex:3foldflop}
  Let $X$ be the smooth toric threefold determined by the following
  fan $\Sigma$ in $\QQ^3$: the rays $\Sigma(1)$ are generated by the
  vectors $v_1 := (1,0,0)$, $v_2 := (0,1,0)$, $v_3 := (-1,-1,-1)$,
  $v_4 := (0,1,1)$, $v_5 := (1,0,1)$ and the two-dimensional cones are
  represented in Figure~\ref{fig:3fold}~(a).  There is a flop $X
  \dashrightarrow X'$ where the toric variety $X'$ is the determined
  by the triangulation of $\Sigma(1)$ where the cone generated by
  $v_1,v_4$ is replaced by the cone with generators $v_2, v_5$. For
  $(k, \ell) \in \ZZ^2$, write $\mathscr{O}_{X}(k,\ell) :=
  \mathscr{O}_{X}(kD_3+\ell D_2) \in \Pic(X)$.  The complete quiver of
  sections for $\mathcal{L} = \bigl( \mathscr{O}_{X},
  \mathscr{O}_{X}(0,1), \mathscr{O}_{X}(1,0)\bigr)$ satisfies
  $X=\vert\mathcal{L}\vert$ in this case (compare Exercise~\ref{ex:challenge}).
  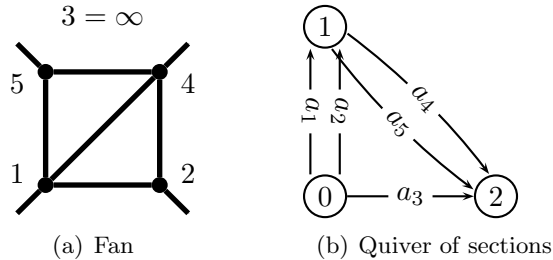
\begin{figure}[!ht]
    \centering
    \mbox{
     \subfigure[Fan]{
        \psset{unit=0.75cm}
        \begin{pspicture}(3,3){
            \cnode*[fillcolor=black](0,0){0pt}{Q1}
            \cnode*[fillcolor=black](3,0){0pt}{Q2}
            \cnode*[fillcolor=black](0,3){0pt}{Q3}
            \cnode*[fillcolor=black](3,3){0pt}{Q4}
            \cnode*[fillcolor=black](0.5,0.5){3pt}{P1}
            \cnode*[fillcolor=black](2.5,0.5){3pt}{P2}
            \cnode*[fillcolor=black](0.5,2.5){3pt}{P5}
            \cnode*[fillcolor=black](2.5,2.5){3pt}{P4}
            \ncline[linewidth=2pt]{-}{P1}{Q1}
            \ncline[linewidth=2pt]{-}{P2}{Q2}
            \ncline[linewidth=2pt]{-}{P5}{Q3}
            \ncline[linewidth=2pt]{-}{P4}{Q4}
            \ncline[linewidth=2pt]{-}{P1}{P2}
            \ncline[linewidth=2pt]{-}{P1}{P5}
            \ncline[linewidth=2pt]{-}{P2}{P4}
            \ncline[linewidth=2pt]{-}{P5}{P4}
            \ncline[linewidth=2pt]{-}{P1}{P4}
            \rput(0,0.7){\psframebox*{$1$}}
            \rput(3,0.7){\psframebox*{$2$}}
            \rput(0,2.3){\psframebox*{$5$}}
            \rput(3,2.3){\psframebox*{$4$}}
            \rput(1.5,3.5){\psframebox*{$3=\infty$}}
          }
      \end{pspicture}}
      \qquad \qquad
      \subfigure[Quiver of sections]{
        \psset{unit=0.75cm}
        \begin{pspicture}(4,3.5)
          \rput(0,0.3){
            \cnodeput(0,0){A}{0}
            \cnodeput(0,3){B}{1} 
            \cnodeput(3,0){C}{2}
             \psset{nodesep=0pt}
            \ncline[offset=5.5pt]{->}{A}{B} \lput*{:180}{$a_1$}
            \ncline[offset=5.5pt]{<-}{B}{A} \lput*{:U}{$a_2$}
            \ncline{->}{A}{C} \lput*{:U}{$a_3$}
             \ncarc[offset=3pt]{->}{B}{C} \lput*{:U}{$a_4$}
            \ncarc[offset=3pt]{<-}{C}{B} \lput*{:180}{$a_5$}
           }
      \end{pspicture}}
    }
    \caption{Projective threefold admitting a flop \label{fig:3fold}}
  \end{figure}
\end{example}

 \subsection{Morphism to the multilinear series}
 We can provide the multigraded analogue of the morphism
 $\varphi_{\vert L\vert}\colon X\to \vert L\vert\cong \mathbb{P}^m$ to
 the classical linear series for a basepoint-free line bundle $L$
 on a toric variety $X$.  To motivate our generalisation we reconsider
 Example~\ref{ex:O3}.
 
 \begin{example}
   For $X=\mathbb{P}^1$, consider $\mathcal{L}=
   \big(\mathscr{O}_{\mathbb{P}^1},\mathscr{O}_{\mathbb{P}^1}(3)\big)$
   with complete quiver of sections $Q$. The image of $\varphi_{\vert
     \mathscr{O}_{\mathbb{P}^1}(3)\vert}$ is obtained by taking
   $\Proj$ of the section ring
 \[
 \bigoplus_{j\geq 0}
 H^0\big{(}\mathbb{P}^1,\mathscr{O}_{\mathbb{P}^1}(3j)\big{)}\cong \kk[y_1,y_2,y_3,y_4]/I,
 \]
 where $I = \big(y^u-y^v : u,v\in \NN^4, u-v\in \Ker(\pi)\big)$ is the
 \emph{toric ideal} defined by the matrix
\[
 \pi = \begin{pmatrix} 1 & 1 & 1 & 1 \\ 3 & 2 & 1 & 0 \\ 0 & 1 & 2 & 3
   \end{pmatrix},
 \]
 where the grading is recorded in the first variable.
 \end{example}
 
 We now associate a toric ideal to every complete quiver of sections.
 The map sending $a \in Q_1$ to $\div(a) \in \ZZ^{\Sigma(1)}$ extends
 to give a $\ZZ$-linear map $\div \colon \ZZ^{Q_1} \rightarrow
 \ZZ^{\Sigma(1)}$.  The \emph{section lattice} $\ZZ(Q)$ is defined to
 be the image of the lattice map $\pi := (\inc,\div) \colon
 \ZZ^{Q_{1}} \rightarrow \Wt(Q)\oplus \ZZ^{\Sigma(1)}$ sending
 $\chi_a\in \ZZ^{Q_1}$ to $\pi(\chi_a) = \bigl( \chi_{\head(a)} -
 \chi_{\tail(a)}, \div(a) \bigr)$. The projections onto the components
 are denoted $\pi_{1} \colon \ZZ(Q) \to \Wt(Q)$ and $\pi_{2} \colon
 \ZZ(Q) \to \ZZ^{\Sigma(1)}$ respectively, and fit in to the
 commutative diagram
\begin{equation}
 \label{eqn:diagram}
\xymatrix{\ZZ(Q)\ar[r]^*{\text{\footnotesize{$\pi_1$}}}\ar[d]^*{\text{\footnotesize{$\pi_2$}}} & \Wt(Q)\ar[d]^*{\text{\footnotesize{$\pic$}}} \\
 \ZZ^{\Sigma(1)} \ar[r]^*{\text{\footnotesize{$\deg$}}} & \Pic(X)}
\end{equation}
where $\deg\colon \ZZ^{\Sigma(1)}\to \Pic(X)$ appears in
\eqref{eq:fundiagram}, and where $\pic\colon \Wt(Q)\to \Pic(X)$ sends
$\sum_{i\in Q_0} \theta_i\chi_i$ to $\bigotimes_{i\in Q_0}L_i^{\otimes
\theta_i}$. Write $\NN^{Q_1}$ and $\NN(Q)$ for the subsemigroups of
$\ZZ^{Q_1}$ and $\ZZ(Q)$ generated by $\{\chi_a : a\in Q_1\}$ and
$\{\pi(\chi_a) :a\in Q_1\}$ respectively.  The coordinate ring
$\kk[y_a : a\in Q_1]$ of $\mathbb{A}^{Q_1}_\kk$ is the semigroup
algebra of $\NN^{Q_1}$, and the map $\pi$ induces a surjective
$\kk$-algebras homomorphism $\pi_*\colon \kk[y_a: a\in Q_1]\to
\kk[\NN(Q)]$ with
 \[
 I_Q := \Ker(\pi_*)=\bigl( y^u- y^v \in \kk[y_a : a\in Q_1] : u - v
 \in \Ker(\pi) \bigr);
 \]
 this is the \emph{toric ideal} for the labeled quiver $Q$. The
 incidence map factors through $\NN(Q)$, so the action of $T$ on
 $\mathbb{A}_\kk^{Q_1}$ restricts to an action on the affine toric
 variety $\mathbb{V}(I_Q) = \Spec\big(\kk[\NN(Q)]\big)$ cut out by the prime
 $I_Q$, and we obtain
 \[
 \mathbb{V}(I_Q) \git_{\vartheta}T = \Proj
 \Bigl( \bigoplus_{j\geq 0} \kk[\NN(Q)]_{j \vartheta}
 \Bigr).
 \]
 The significance of this semiprojective toric variety is shown by the
 next result; the projective case is the first main result of
 Craw--Smith~\cite{CrawSmith}.
 
 \begin{theorem}
 \label{thm:image}
   Let $X$ be a semiprojective toric variety and let
   $\mathcal{L}=(\mathscr{O}_X,L_1,\dots, L_r)$ be a set of
   basepoint-free line bundles with complete quiver of sections $Q$.
   There is a morphism
 \[
 \varphi_{\vert \mathcal{L}\vert}\colon
   X\longrightarrow
   \vert\mathcal{L}\vert=\mathbb{A}_\kk^{Q_1}\git_\vartheta T
 \]
 whose image $\mathbb{V}(I_Q)\git_\vartheta T$ is equal to the
   geometric quotient of $\mathbb{V}(I_Q)\setminus\mathbb{V}(B_Q)$ by
   $T$.
 \end{theorem}
 \begin{sketch}
   The sheaf $\bigoplus_{i\in Q_0}L_i$ on $X$ defines a family of
   $\vartheta$-stable representations over $X$, and the morphism to
   the multilinear series is induced by the universal property of
   $\mathcal{M}_\vartheta(Q)=|\mathcal{L}|$. More concretely, the
   morphism $\varphi_{\vert \mathcal{L}\vert}$ is induced by the
   $\kk$-algebra homomorphism
 \[
 \Phi_Q\colon \kk[y_a : a\in Q_1]\to \kk[x_\rho : \rho\in \Sigma(1)]
 \]
 between the total coordinate rings of $\vert \mathcal{L}\vert$ and
 $X$ that sends $y_a$ for $a\in Q_1$ to the monomial
 $x^{\div(a)}$. The subscheme of $\mathbb{A}^{Q_1}_\kk$ cut out by the
 kernel of $\Phi_Q$ need not be invariant under the action of $T$, but
 by shrinking the defining ideal from $\Ker(\Phi_Q)$ to $I_Q$, we
 enlarge the subscheme and hence capture in $\mathbb{V}(I_Q)$ the
 $T$-orbits of all points from $\mathbb{V}\big(\Ker(\Phi_Q)\big)$.
 The $\Wt(Q)$-grading of the coordinate ring $\kk[\NN(Q)]$ of
 $\mathbb{V}(I_Q)$ and the $A_{n-1}(X)$-grading of the total
 coordinate ring of $X$ are each encoded by a horizontal map in
 Diagram~\eqref{eqn:diagram}. Commutativity of \eqref{eqn:diagram}
 shows that the map $\mathbb{A}^{\Sigma(1)}\to\mathbb{V}(I_Q)\subseteq
 \mathbb{A}^{Q_1}_\kk$ induced by $\Phi_Q$ is equivariant with
 respect to the actions of $G=\Hom(A_{n-1}(X),\kk^\times)$ on the
 domain and $T=\Hom(\Wt(Q),\kk^\times)$ on the target, and by the
 discussion above it is surjective after passing to group orbits. The
 statement of \one\ follows after verifying that the preimage of the
 irrelevant subscheme $\mathbb{V}(B_Q)$ is contained in
 $\mathbb{V}(B_X)$. This condition follows from the assumption that
 each line bundle in the sequence $\mathcal{L}$ is basepoint-free. For
 details, see
 Craw--Smith~\cite[Corollary~4.1,Proposition~4.3]{CrawSmith}.\qed
 \end{sketch}

 \begin{exercise}
   For the sequence $\mathcal{L} = (\mathscr{O}_X, L_1)$ with complete
   quiver of sections $\mathcal{L}$ from Example~\ref{exa:2pt}, compute
   the generators of the semigroup $\NN(Q)$ and hence show that
   $\varphi_{|\mathcal{L}|} = \varphi_{|L_1|}$.
 \end{exercise}
 
 In order to reconstruct the toric variety $X$ using
 Theorem~\ref{thm:image} we must establish when the morphism to the
 multilinear series is a closed immersion. The following result
 provides the multigraded analogue of the classical result that
 relates very ample line bundles with closed immersions.

 \begin{theorem}
 \label{thm:surjcriterion}
 With the assumptions of Theorem~\ref{thm:image}, assume in addition
 that the multiplication map $H^0(X,L_1)\otimes \dots \otimes
 H^0(X,L_r)\longrightarrow H^0\big{(}X,\textstyle{\bigotimes_{i\in
 Q_0} L_i}\big{)}$ is surjective. Then
 \begin{enumerate}
 \item[\one] the morphism $\varphi_{\vert \mathcal{L}\vert}$ is a closed
 immersion if and only if $\bigotimes_{i\in Q_0} L_i$ is very ample; and
 \item[\two] for each $i\in Q_0$, the tautological line bundle
  $\mathscr{W}_i$ on $\vert\mathcal{L}\vert=\mathcal{M}_\vartheta(Q)$
  satisfies $ \varphi_{\vert \mathcal{L}\vert}^*(\mathscr{W}_i)=L_i$.
 \end{enumerate}
 \end{theorem}
 \begin{sketch}
   The proof of \cite[Proposition~4.9]{CrawSmith} shows that the
   morphism $\varphi_{|\mathcal{L}|}$ is a closed immersion if and
   only if the map to projective space (over the ring
   $\kk[\NN^{Q_1}\cap \Circuit(Q)]$) determined by the line bundle
   $L=\bigotimes_{i\in Q_0} L_i$ and the subspace
   $\kk\langle\pi_2\big{(}\NN(Q)\cap
   \pi_1^{-1}(\vartheta)\big{)}\rangle \subseteq H^0(X,L)$ is a closed
   immersion.  The sections $ \pi_2\big{(}\NN(Q)\cap
   \pi_1^{-1}(\chi_i-\chi_0)\big{)}$ that label the set of paths in
   $Q$ from $0\in Q_0$ to $i\in Q_0$ form a $\kk$-vector space basis
   of $H^0(X,L_i)$. By taking the product of these sections over all
   vertices in $Q$, we see that the divisors labelling paths in the
   quiver encode the basis $\pi_2\big{(}\NN(Q)\cap
   \pi_1^{-1}(\vartheta)\big{)}$ of $\bigotimes_{i\in Q_0}H^0(X,L_i)$.
   By the assumption, these sections span $H^0(X,\bigotimes_{i\in Q_0}
   L_i)$, and part \one\ follows. For part \two, see
   Craw--Smith~\cite[Theorem~4.15]{CrawSmith}.\qed
 \end{sketch}

 \begin{example}
   For a sequence $\mathcal{L}=(\mathscr{O}_X,L_1)$ with $L_1$ very ample
   of $X$, Theorem~\ref{thm:surjcriterion}\two\ shows that the
   restriction of the tautological bundle
   $\mathscr{W}_1=\mathscr{O}_{\vert \mathcal{L}_1\vert}(1)$ on $\vert
   L_1\vert$ to the image of $X$ under the closed immersion
   $\varphi_{\vert L_1\vert}\colon X\to \vert L_1\vert$ is equal to
   $L_1$. 
 \end{example}

 \begin{remark}
   Theorem~\ref{thm:surjcriterion} enables one to show that every
   projective toric variety $X$ admits many collections $\mathcal{L}$
   for which $\varphi_{\vert \mathcal{L}\vert}$ is a closed immersion.
   The key is the notion of \emph{multigraded regularity} introduced
   by Maclagan--Smith~\cite{MaclaganSmith}. Using an application of
   regularity due to Hering-Schenck--Smith~\cite{HSS}, one can show
   that if $L_1,\dots, L_{r-1}$ are basepoint-free line bundles that
   do not all lie in the same face of the nef cone of $X$, then there
   exists $L_r\in \Pic(X)$ that is $\mathscr{O}_X$-regular with
   respect to $L_1,\dots, L_{r-1}$ such that $\varphi_{\vert
     \mathcal{L}\vert}$ is a closed immersion for
   $\mathcal{L}=(\mathscr{O}_X,L_1,\dots, L_r)$.
 \end{remark}

 \subsection{Toric varieties as fine moduli spaces}
 We now use the morphism to the multilinear series $\varphi_{\vert
 \mathcal{L}\vert}\colon X\to \vert\mathcal{L}\vert$ to give the fine
 moduli space construction of $X$. Recall from
 Section~\ref{sec:moduliboundquiver} that for any bound quiver
 $(Q,\varrho)$, the set of relations $\varrho$ determine an ideal
 $I_\varrho\subseteq \kk[y_a : a\in Q_1]$ that cuts out the fine
 moduli space
 \[
 \mathcal{M}_\vartheta(Q,\varrho)=
 \mathbb{V}(\Irel)\git_\vartheta T
 \]
 of $\vartheta$-stable bound representations of $(Q,\varrho)$ of
 dimension vector $(1,\dots,1)\in \ZZ^{Q_0}$ . For a complete bound
 quiver of sections, the set $\varrho$ of relations comprises path
 differences $p-p'\in \kk Q$, so the ideal $I_\varrho$ is a binomial
 ideal. It is easy to see that the toric ideal of equations $I_Q$ and
 the binomial ideal of relations $I_\varrho$ satisfy
 $I_\varrho\subseteq I_Q$, and hence we obtain
 \begin{equation}
 \label{eqn:inclusions}
 \mathbb{V}(I_Q)\subseteq \mathbb{V}(I_\varrho)\subseteq
 \mathbb{A}^{Q_1}_\kk.
 \end{equation}
 It follows from Theorem~\ref{thm:image} that the image
 $\mathbb{V}(I_Q)\git_\vartheta T$ of the morphism $\varphi_{\vert
 \mathcal{L}\vert}\colon X\to\vert\mathcal{L}\vert$ is a subscheme of
 $ \mathcal{M}_\vartheta(Q,\varrho)$. Both inclusions from
 \eqref{eqn:inclusions} can be proper, both can be equality, and
 either one but not the other can be proper; in short, anything goes!

 \begin{exercise}
    For each sequence below, decide which of the inclusions
    from \eqref{eqn:inclusions} is proper: 
 \begin{enumerate}
 \item[\one] $\mathcal{L}=
   \big(\mathscr{O}_{\mathbb{P}^1},\mathscr{O}_{\mathbb{P}^1}(3)\big)$
   on $\mathbb{P}^1_\kk$ defining the rational
   normal curve of degree three;
 \item[\two] $\mathcal{L}=
   \big(\mathscr{O}_{\mathbb{P}^2},\mathscr{O}_{\mathbb{P}^2}(1),\mathscr{O}_{\mathbb{P}^2}(2)\big)$
   on $\mathbb{P}^2$ that defines the bound Be{\u\i}linson quiver for
   $\mathbb{P}^2_\kk$;
 \item[\three] $\mathcal{L} = \big(\mathscr{O}_{X},
  \mathscr{O}_{X}(0,1), \mathscr{O}_{X}(1,0)\big)$ on the 3-fold $X$
  from Example~\ref{ex:3foldflop};
 \end{enumerate}
 \end{exercise}
 
 For a more involved example we augment the sequence $\mathcal{L}$ on
 $\mathbb{F}_1$ from Example~\ref{exa:F1asmultilinear} by adding an
 extra line bundle.

 \begin{example}
 \label{ex:F1four}
 For $X=\mathbb{F}_1$ and $\mathcal{L} = \bigl( \mathscr{O}_{X},
 \mathscr{O}_{X}(1,0), \mathscr{O}_{X}(0,1), \mathscr{O}_{X}(1,1)
 \bigr)$, the complete quiver of sections is shown in
 Figure~\ref{fig:F1tilting}: the sections from the total coordinate
 ring of $X$ that determines the arrows are illustrated in
 Figure~\ref{fig:F1tilting}(a); the arrows are listed in
 Figure~\ref{fig:F1tilting}(b);
 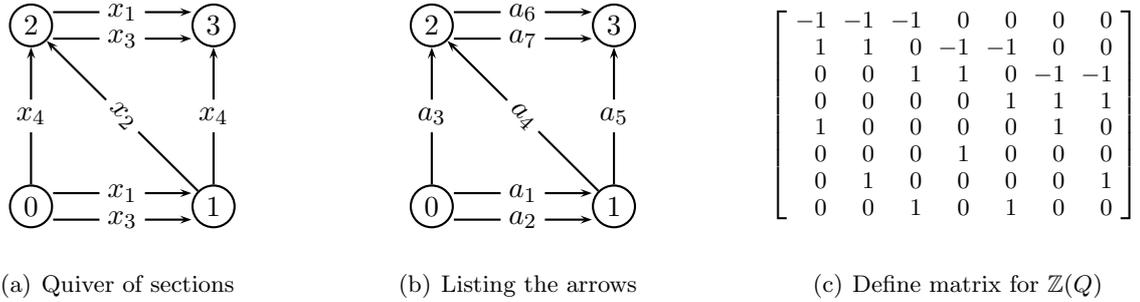
\begin{figure}[!ht]
    \centering
    \mbox{
      \subfigure[Quiver of sections]{
        \psset{unit=1.2cm}
        \begin{pspicture}(-0.5,-0.5)(2.5,2.2)
          \cnodeput(0,0){A}{0}
          \cnodeput(2,0){B}{1} 
          \cnodeput(0,2){C}{2}
          \cnodeput(2,2){D}{3}
          \psset{nodesep=0pt}
          \ncline[offset=5pt]{->}{A}{B} \lput*{:U}{$x_1$}
          \ncline[offset=-5pt]{->}{A}{B} \lput*{:U}{$x_3$}
          \ncline{->}{A}{C} \lput*{:270}{$x_4$}
          \ncline{->}{B}{D} \lput*{:270}{$x_4$}
          \ncline{->}{B}{C} \lput*{:180}{$x_2$}
          \ncline[offset=5pt]{->}{C}{D} \lput*{:U}{$x_1$}
          \ncline[offset=-5pt]{->}{C}{D} \lput*{:U}{$x_3$}
        \end{pspicture}}\qquad \qquad  
      \subfigure[Listing the arrows]{
        \psset{unit=1.2cm}
        \begin{pspicture}(-0.5,-0.5)(2.5,2.2)
          \cnodeput(0,0){A}{0}
          \cnodeput(2,0){B}{1} 
          \cnodeput(0,2){C}{2}
          \cnodeput(2,2){D}{3}
          \psset{nodesep=0pt}
          \ncline[offset=5pt]{->}{A}{B} \lput*{:U}{$a_1$}
          \ncline[offset=-5pt]{->}{A}{B} \lput*{:U}{$a_2$}
          \ncline{->}{A}{C} \lput*{:270}{$a_3$}
          \ncline{->}{B}{D} \lput*{:270}{$a_5$}
          \ncline{->}{B}{C} \lput*{:180}{$a_4$}
          \ncline[offset=5pt]{->}{C}{D} \lput*{:U}{$a_6$}
          \ncline[offset=-5pt]{->}{C}{D} \lput*{:U}{$a_7$}
        \end{pspicture}}\qquad \qquad\quad
      \subfigure[Define matrix for $\ZZ(Q)$]{
 \renewcommand{\arraystretch}{0.8}
   \renewcommand{\arraycolsep}{3pt}
\psset{unit=1.2cm}
        \begin{pspicture}(-0.5,-0.5)(2.5,2.2)
  \rput(0.8,1){$\left[ \text{\footnotesize $\begin{array}{rrrrrrr}
        -1 & -1 & -1 &  0 &  0 &  0 &  0\\
         1 &  1 &  0 & -1 & -1 &  0 &  0\\
         0 &  0 &  1 &  1 &  0 & -1 & -1\\
         0 &  0 &  0 &  0 &  1 &  1 &  1\\ 
         1 &  0 &  0 &  0 &  0 &  1 &  0\\
         0 &  0 &  0 &  1 &  0 &  0 &  0\\
         0 &  1 &  0 &  0 &  0 &  0 &  1\\
         0 &  0 &  1 &  0 &  1 &  0 &  0
      \end{array}$} \right]$}
 \end{pspicture}}
    }    \caption{A tilting quiver on the Hirzebruch surface $\mathbb{F}_{1}$ \label{fig:F1tilting}}
  \end{figure}
  and the section lattice $\ZZ(Q)$ is generated by the columns of the
  matrix presented in Figure~\ref{fig:F1tilting}(c).  The toric ideal
  is $I_Q = \big{(}y_6y_3-y_5y_1, y_7y_3-y_5y_2, y_2y_6 - y_1y_7
  \big{)}$, and $\mathbb{V}(I_Q)\subset\mathbb{A}^7_\kk$ is a normal
  affine toric variety of dimension 5.  In this case, the ideal of
  relations is
 \[
 I_\varrho = \big{(}y_3y_6-y_1y_5, y_3y_7-y_2y_5, y_2y_4y_6 -
 y_1y_4y_7\big{)}
  = I_Q \cap (y_3, y_4, y_5),
 \]
 so $I_Q$ is a primary component of $I_\varrho$.  Geometrically, the
 toric variety $\mathbb{V}(I_Q)$ is the unique irreducible component
 of the binomial subscheme $\mathbb{V}(I_\varrho)\subseteq
 \mathbb{A}^{7}_\kk$ that does not lie in the $\vartheta$-unstable
 locus $\mathbb{V}(B_Q)$ in $\mathbb{A}^7_\kk$ for
 $\vartheta=(-3,1,1,1)$. In particular, we have that
 $\mathbb{V}(I_Q)\git_\vartheta
 T\cong\mathbb{V}(I_\varrho)\git_\vartheta T$ even although
 $\mathbb{V}(I_Q)\neq \mathbb{V}(I_\varrho)$.
\end{example}

For $I_\varrho \subset \kk[y_a : a\in Q_1]$, the saturation of
$I_\varrho$ by $B_Q$ is
 \[
 (I_\varrho : B_Q^\infty):= \big{\{} f\in\kk[y_a : a\in Q_1] : b^nf\in
 I_\varrho\text{ for some }b\in B_Q, n\geq 0\big{\}}.
 \]
 The important point for our construction is that the ideal $B_Q$ is
 given explicitly, so one can check when the image is a fine moduli
 space of stable representations of a bound quiver:
 
 \begin{theorem}
  \label{thm:finequotient}
  Let $(Q,\varrho)$ be the complete bound quiver of sections for a
  sequence $\mathcal{L}$ of basepoint-free line bundles on $X$, and let
  $\vartheta\in \Wt(Q)$ satisfy $\vartheta_i>0$ for $i\neq 0$. The
  following are equivalent:
  \begin{enumerate}
  \item[\one] the image of $\varphi_{|\mathcal{L}|}$ is equal to the
    moduli space $\mathcal{M}_{\vartheta}(Q,\varrho)$;
  \item[\two] the ideal $I_Q$ equals the ideal quotient $(I_\varrho :
    B_Q^\infty)$.
  \end{enumerate}
  Moreover, every projective toric variety $X$ admits many sequences
  $\mathcal{L}$ for which the morphism $\varphi_{\vert \mathcal{L}\vert}$ is a
  closed immersion, and where one and hence both of the above
  conditions holds.
 \end{theorem}
 \begin{proof}
   The image of $\varphi_{|\mathcal{L}|}$ is $\mathbb{V}(I_Q)
   \git_\vartheta T$ by Theorem~\ref{thm:image}, while the moduli
   space $\mathcal{M}_{\theta}(Q,\varrho)$ is defined to be
   $\mathbb{V}(I_\varrho) \git_\vartheta T$. These GIT quotients
   coincide precisely when the $\vartheta$-stable loci in
   $\mathbb{V}(I_Q)$ and $\mathbb{V}(I_\varrho)$ coincide. The former
   is $\mathbb{V}(I_Q)\setminus \mathbb{V}(B_Q)$, the latter is
   $\mathbb{V}(I_\varrho)\setminus \mathbb{V}(B_Q)$, and since
   $\mathbb{V}(I_Q)\subseteq \mathbb{V}(I_\varrho)$ holds by
   \eqref{eqn:inclusions}, these $\vartheta$-stable loci coincide precisely
   when $I_Q = (I_\varrho : B_Q^\infty)$.
   
   The existence result is the hardest part of
   Craw--Smith~\cite{CrawSmith}.  To obtain appropriate sequences
   $\mathcal{L}$, we provide an explicit construction as follows: for
   basepoint-free line bundles $L_1, \dotsc, L_{r-2}$ on $X$, if
   the subsemigroup of $\Pic(X)$ generated by $L_1, \dotsc, L_{r-2}$
   contains an ample line bundle, then there exist line bundles
   $L_{r-1}$ and $L_r$ such that $\mathcal{L} = (\mathscr{O}_X, L_1,
   \dotsc, L_r)$ satisfies the condition. The proof use multigraded
   regularity to show that a particular ideal has only quadratic
   generators, and then employs an efficient saturation technique.
\end{proof}
 
\begin{corollary}
  Every normal semiprojecive toric variety $X$ is isomorphic to a fine
  moduli space $\mathcal{M}_\vartheta(Q,\varrho)$ for some complete
  bound quiver of sections and for $\vartheta\in \Wt(Q)$ satisfying
  $\vartheta_i>0$ for $i\neq 0$.
 \end{corollary}


\begin{example}
 \label{exa:GHilb1/3(1,2)}
  For the minimal resolution $X$ of the quotient singularity of type
  $\frac{1}{3}(1,2)$, the complete bound quiver of sections
  $(Q,\varrho)$ of the sequence $\mathcal{L}=(\mathscr{O}_X,L_1,L_2)$ on
  $X$ from Example~\ref{exa:qos1/3(1,2)} is the bound McKay quiver
  shown in Figure~\ref{fig:qos1/3(1,2)}(a). Using sequence
  \eqref{eq:fundiagramres1/3(1,2)}, compute the lattice points of the
  polyhedra $\conv(\NN^4\cap\pi^{-1}(L_i))$ for $i=1,2$ and compare
  $\conv(\NN^4\cap\pi^{-1}(L_1\otimes L_2))$ from
  Figure~\ref{fig:res1/3(1,2)}(b). It follows that the multiplication
  map $H^0(L_1)\otimes_\kk H^0(L_2)\to H^0(L_1\otimes L_2)$ is
  surjective.  Since $L_1\otimes L_2$ is very ample,
  Theorem~\ref{thm:surjcriterion} implies that $\varphi_{\vert
    \mathcal{L}\vert}$ is a closed immersion. The section lattice
  $\ZZ(Q)$ is generated by the columns of the matrix
 \[
 \left[ \text{\footnotesize $\begin{array}{rrrrrrr}
        -1 & -1 &  0 &  1 &  1 &  0 \\
         1 &  0 & -1 & -1 &  0 &  1 \\
         0 &  1 &  1 &  0 & -1 & -1 \\
         1 &  0 &  1 &  0 &  1 &  0 \\ 
         0 &  0 &  1 &  1 &  1 &  0 \\
         0 &  0 &  0 &  1 &  1 &  1 \\
         0 &  1 &  0 &  1 &  0 &  1 \\
      \end{array}$} \right]
 \]
 and the toric ideal of equations $I_Q=(y_1y_4-y_2y_5, y_3y_6-y_1y_4,
 y_2y_5-y_3y_6)$ coincides with the binomial ideal $I_\varrho$ cut out
 by the relations, so $\mathbb{V}(I_Q)\git_\vartheta T =
 \mathcal{M}_\vartheta(Q,\varrho)$ for the weight $\vartheta =
 (-2,1,1)\in \Wt(Q)$.  This establishes that the morphism
 $\varphi_{\vert \mathcal{L}\vert}\colon X\to \vert\mathcal{L}\vert$
 identifies the minimal resolution $X$ with the fine moduli space $
 \mathcal{M}_\vartheta(Q,\varrho)$. Since $(Q,\varrho)$ is the bound
 McKay quiver, Remark~\ref{rem:G-Hilbert} shows that the
 minimal resolution $X$ is isomorphic to the $G$-Hilbert scheme. This
 result is due originally to Ito--Nakamura~\cite{ItoNakamura}
 \end{example}

 \begin{remark}
 \label{rem:CMT}
 For the bound McKay quiver $(Q,\varrho)$ of a finite abelian subgroup
 $G\subset \GL(n,\kk)$, Craw--Maclagan--Thomas~\cite{CMT1,CMT2}
 constructed a toric ideal of equations $I_Q\subset \kk[y_a : a\in
 Q_1]$ using an especially simple matrix, namely, the incidence matrix
 of $X$ augmented by $\vert G\vert$ blocks of $n\times n$ identity
 matrices.  For any weight $\vartheta\in \Wt(Q)$ satisfying
 $\vartheta_i>0$ for $i\neq 0$, the GIT quotient $\mathbb{V}(I_Q)
 \git_{\vartheta}T\subseteq \mathcal{M}_\vartheta(Q,\varrho)$ is the
 irreducible component of the $G$-Hilbert scheme that contains free
 $G$-orbit. Some properties of this \emph{coherent component} of
 $\mathcal{M}_\vartheta(Q,\varrho)$ are described in
 Section~\ref{sec:VGIT1}.
 \end{remark} 

 \section{Tilting bundles and exceptional collections}
 \label{sec:tilting}
 This section describes how the bounded derived category of coherent
 sheaves on a smooth projective toric variety can, in certain cases,
 be described via the category of finitely-generated representations
 of a bound quiver $(Q,\varrho)$.  We establish the relation between
 tilting bundles and full strong exceptional collections, and describe
 how one might hope to extend some of the ideas to certain smooth
 toric DM stacks.  We emphasise that, in light of the study of tilting
 bundles on rational surfaces by
 Hille--Perling~\cite{HillePerling,HillePerling2}, tilting bundles appear to
 exist (on toric varieties) only in rather special cases.

 \subsection{On derived categories of coherent sheaves}
 The sections on derived categories that follow are based on lectures
 that were presented in parallel with a sequence of lectures by Andrei
 C{\u{a}}ld{\u{a}}raru~\cite{Caldararu}. We omit in these notes the
 necessary introduction to derived categories, referring instead to
 \cite{Caldararu} and to the excellent book by
 Huybrechts~\cite{Huybrechts} for the background material. It is
 nevertheless appropriate to highlight one fundamental result that has
 special relevance for us.

 To this end, let $X$ be a smooth projective
 variety over $\kk$. Write $\coh(X)$ for the abelian category of
 coherent sheaves on $X$, and $D^b\big(\!\coh(X)\big)$ for the bounded
 derived category of coherent sheaves on $X$.

  \begin{proposition}
 \label{prop:diagonal}
 Let $X$ be a smooth projective variety and write $\pi_1, \pi_2 \colon
 X\times X\to X$ for the first and second projections. Let $\iota
 \colon \Delta\hookrightarrow X\times X$ denote the diagonal
 embedding. The functor 
 \[
 \Rderived(\pi_2)_*\big{(} \pi_1^*(-)\ltensor
 \mathscr{O}_{\Delta}\big{)}\colon D^b\big(\!\coh(X)\big) \to
 D^b\big(\!\coh(X)\big)
 \]
 is naturally isomorphic to the identity.
 \end{proposition}
 \begin{proof}
   For any object $F\in D^b\big(\!\coh(X\times X)\big)$, the
   projection formula for $\iota$ gives
 \[
 \Rderived\iota_*\big{(}\Lderived\iota^*(F)\ltensor
 \mathscr{O}_X\big{)} \cong F\ltensor \Rderived\iota_*(\mathscr{O}_X)
 \cong F\ltensor \mathscr{O}_\Delta.
 \]
 Therefore, for any object $E\in D^b\big(\!\coh(X)\big)$ we have 
 \[
 \Rderived(\pi_2)_*\big{(}\pi_1^*(E)\ltensor
 \mathscr{O}_{\Delta}\big{)} \cong \Rderived(\pi_2)_*\Big{(}\Rderived\iota_*\big{(}\Lderived
 \iota^*(\pi_1^*(E))\ltensor \mathscr{O}_X\big{)}\Big{)} \cong \Rderived(\pi_2\circ \iota)_*\Big{(}\Lderived \big(\pi_1\circ
   \iota\big)^*(E)\ltensor \mathscr{O}_X\big{)}\Big{)}
 \]
 This is isomorphic to $E$ since both $\pi_2\circ \iota$ and
$\pi_1\circ \iota$ coincide with the identity.
 \end{proof}
 
 Many other natural isomorphisms will be useful. The following
 exercise highlights a pair of these, the proofs of which can be found in
 Huybrechts~\cite{Huybrechts} and Hartshorne~\cite{Hartshorne}.

 \begin{exercise}
 \label{ex:naturalisoms}
 Establish the following natural isomorphisms of functors:
 \begin{enumerate}
 \item[\one] for any object $E\in D^b\big(\!\coh(X)\big)$, $\Rderived
 \Hom(E,\mathscr{O}_X)\cong \Rderived\Gamma\circ
 \Rderived\mathcal{H}om(E,\mathscr{O}_X)$;
 \item[\two] $\Rderived(\pi_2)_* \circ \pi_1^*(-)\cong \Rderived \Gamma(-)\otimes \mathscr{O}_X$;
 \end{enumerate}
 \end{exercise}


 \subsection{Tilting sheaves}
 Let $F$ be a coherent sheaf on a smooth variety $X$ that is
 projective over an affine variety $\Spec(R)$. A coherent sheaf of the
 form $T:=\mathscr{O}_X\oplus F$ is \emph{tilting} if
 \begin{enumerate}
 \item[(T1)] the algebra $A:=
   \End_{\mathscr{O}_X}(T)$ has finite
   global dimension, i.e., the maximal projective dimension of any
   object in $\modA$ is finite;
 \item[(T2)] we have $\Ext^k_{\mathscr{O}_X}(T,T)=0$ for all $k>0$;
   and 
 \item[(T3)] the sheaf $T$ classically generates
   $D^b\big(\!\coh(X)\big)$, i.e., the smallest triangulated
   subcategory of $D^b\big(\!\coh(X)\big)$ containing $T$ and all of
   its direct summands is $D^b\big(\!\coh(X)\big)$.
 \end{enumerate}
 We call $A=\End_{\mathscr{O}_X}(T)$ the associated \emph{tilting
 algebra}, and $T$ a \emph{tilting bundle} if it is locally-free (we
 assume $\mathscr{O}_X$ is a summand, but one need not). The main
 result for tilting sheaves on smooth projective varieties is due
 independently to Baer~\cite{Baer} and Bondal~\cite{Bondal}:

 \begin{theorem}
 \label{thm:baerbondal}
 Let $X$ be a smooth variety that is projective over an affine
 variety.  For a tilting sheaf $T$ on $X$ with tilting algebra $A=
 \End(T)$, the functor $\Hom_{\mathscr{O}_X}(T,-)\colon
 \coh(X)\to\modAop$ induces an equivalence of triangulated categories
 \[
 \Rderived\Hom_{\mathscr{O}_X}(T,-)\colon
 D^b\big(\!\coh(X)\big)\to D^b\big(\!\modAop\big)
 \]
 with quasi-inverse $(-)\ltensor_A T\colon D^b\big(\!\modAop\big)\to
 D^b\big(\!\coh(X)\big)$.
 \end{theorem}

 \begin{proof}
   To simplify notation, write the functors as
   $F(-):=\Hom_{\mathscr{O}_X}(T,-)$ and $G(-):= -\otimes_A T$.  The
   first step is to construct the functors $\Rderived F$ and
   $\Lderived G$. Let $E$ be a quasicoherent sheaf on $X$.  The vector
   space $\Hom_{\mathscr{O}_X}(T,E)$ becomes a right $A$-module by
   precomposition, that is, for $a\in\Hom_{\mathscr{O}_X}(T,T)$ and
   $f\in \Hom_{\mathscr{O}_X}(T,E)$ set $a\cdot f = f\circ a\in
   \Hom_{\mathscr{O}_X}(T,E)$. Since $\Hom_{\mathscr{O}_X}(T,-)$ is a
   covariant left-exact functor and since the category of
   quasicoherent sheaves has enough injectives, one obtains a
   right-derived functor
 \[
 \Rderived F(-)=\Rderived \Hom_{\mathscr{O}_X}(T,-)\colon
 D^b\big(Q\!\coh(X)\big)\to D\big(\!\ModAop\big)
 \]
 from the bounded derived category of quasicoherent sheaves on $X$ to
 the derived category of the category $\ModAop$ of (not necessarily
 finitely-generated) right $A$-modules. The cohomology modules of the
 image are
 \begin{equation}
 \label{eqn:cohmods}
 H^i\big(\Rderived \Hom_{\mathscr{O}_X}(T,E)\big) =
 \Rderived^i\Hom_{\mathscr{O}_X}(T,E) =
 \Ext^i_{\mathscr{O}_X}(T,E). 
 \end{equation}
 Smoothness of $X$ implies that $\Ext^i_{\mathscr{O}_X}(T,E) = 0$ for
 $i<0$ and $i>\dim(X)$, so the image of the functor $\Rderived F$ lies
 in $D^b\big(\!\ModAop\big)$. Since $D^b\big(\!\coh(X)\big)$ is
 equivalent to the full subcategory of $D^b\big(Q\!\coh(X)\big)$ whose
 objects have coherent cohomology sheaves, we may consider the
 restriction of $\Rderived F$ to $D^b\big(\!\coh(X)\big)$.  We claim
 that the $A$-module $\Ext^i_{\mathscr{O}_X}(T,E)$ is finitely
 generated for any coherent sheaf $E$. Indeed, $X$ is projective over
 an affine variety $\Spec(R)$ and $T=\mathscr{O}_X\oplus F$, so the centre of $A$ contains $\Hom(\mathscr{O}_X,\mathscr{O}_X)\cong R$ as a subalgebra. The
 observation of Van den Bergh~\cite[proof of Corollary 3.2.8]{VDB} now
 shows that the cohomology modules from \eqref{eqn:cohmods} actually
 lie in $\modA$ whenever $E$ is coherent. Thus, we obtain by
 restriction a functor
 \[
 \Rderived F(-)=\Rderived \Hom_{\mathscr{O}_X}(T,-)\colon
 D^b\big(\!\coh(X)\big)\to D^b\big(\!\modAop\big).
 \]
 Similarly, since the module category $\ModAop$ has enough
 projectives and the functor $-\otimes_A T$ is right-exact, we
 obtain a left-derived functor
 \[
 \Lderived G(-)= -\ltensor_A T\colon D^b\big(\!\ModAop\big)\to D\big(Q\!\coh(X)\big)
 \]
 to the a priori unbounded derived category of quasicoherent sheaves.
 For an $A$-module $B$, the cohomology sheaves of the image are
 \[
 \mathcal{H}^j\Big{(}B\ltensor_A T\Big{)}= \mathcal{T}\text{or}^A_{-j}(B,T),
 \]
 and these vanish off a finite range since $A$ has finite global
 dimension. Furthermore, restricting to finitely generated $A$-modules
 ensures that these cohomology sheaves are coherent, giving
  \[
 \Lderived G(-)= -\ltensor_A T\colon D^b\big(\!\modAop\big)\to D^b\big(\!\coh(X)\big).
 \]
 Since $T$ satisfies property (T2), the composite functor satisfies 
 \[
 \Rderived F\circ \Lderived G(A) = \Rderived F\Big{(}A\ltensor_A
 T\Big{)}= \Rderived \Hom_{\mathscr{O}_X}(T,T) =
 \Hom_{\mathscr{O}_X}(T,T) = A,
 \]
 so $\Rderived F\circ \Lderived G$ is the identity on $A$. The
 smallest triangulated subcategory of $D^b\big(\!\modAop\big)$
 containing $A$ and its direct summands contains finitely-generated
 free $A$-modules, and hence finitely-generated projective
 $A$-modules.  Since $A$ has finite global dimension, every finitely
 generated $A$-module admits a finite projective resolution. This
 implies that smallest triangulated subcategory of
 $D^b\big(\!\modAop\big)$ containing $A$ and its direct summands is
 $D^b\big(\!\modAop\big)$. It follows that $\Rderived F\circ \Lderived
 G$ is the identity on the whole of $D^b\big(\!\modAop\big)$, so
 $\Lderived G(-)$ identifies $D^b\big(\!\modAop\big)$ with the
 triangulated subcategory of $D^b\big(\!\coh(X)\big)$ generated by
 \[
 \Lderived G(A) = A\ltensor_{A}T
 = T.
 \]
 Property (T3) now gives the derived equivalences.
 \end{proof}
 
 In the presence of a tilting sheaf, the derived equivalence from
 Theorem~\ref{thm:baerbondal} provides a significant simplication, and
 is the most that one can hope for in a triangulated category.  For
 example, the module category $\modA$ has finite length whereas the
 abelian category of coherent sheaves on a projective variety does
 not.

 \begin{example}
 \label{exa:affinetilting}
 The trivial bundle always provides a tilting
 bundle on an affine variety $X$.  Indeed, $\mathscr{O}_X$ determines
 both the $\kk$-algebra $A=\Hom(\mathscr{O}_X,\mathscr{O}_X)\cong
 H^0(\mathscr{O}_X)$ and the equivalence of abelian categories
 $\Gamma(X,-)\colon \coh(X)\to \modA$, where $\modA$ is the category
 of finite dimensional $H^0(\mathscr{O}_X)$-modules. Clearly
 $\Rderived \Gamma(X,-)$ is a derived equivalence.
 \end{example}

 \begin{remark}
   See Hille-Van den Bergh~\cite[Theorem~7.6]{HilleVdB} for a stronger
   form of Theorem~\ref{thm:baerbondal}.
 \end{remark}

  \subsection{Strongly exceptional sequences}
  The representation-theoretic notion of tilting sheaf on a smooth
  projective variety is closely related to the more algebro-geometric
  notion of a full strong exceptional sequence.  First, recall from
  Huybrechts~\cite[\S1.4]{Huybrechts} that an object $E$ in a
  $\kk$-linear triangulated category $D$ is \emph{exceptional} if
 \[
 \Hom_{D}(E,E) =\kk \text{ and }\Hom_{D}(E,E[\ell])=0 \text{ for }\ell\neq 0.
 \]
 A sequence $(E_0,E_1,\dots, E_m)$ of exceptional objects is
 \emph{exceptional} if $\Rderived \Hom_{D}(E_i,E_j)=0$ for
 $i>j$, and it is \emph{strongly exceptional} if in addition
 $\Hom_{D}(E_i,E_j[\ell])=0$ for $i<j$ and $\ell\neq 0$.  A sequence
 of objects is \emph{full} if $E_0, E_1,\dots, E_m$ classically
 generates $D$. If $D$ is the bounded derived category of coherent
 sheaves on a variety $X$ over $\kk$, we have
 \[
 \Hom_{D^b(\coh(X))}(E,F[\ell])=\Ext^\ell_{\mathscr{O}_X}(E,F).
 \]
 
 \begin{proposition}
 \label{prop:exceptionaltilting}
 Let $(E_0=\mathscr{O}_X,E_1,\dots, E_m)$ be a sequence of locally-free sheaves on a
 smooth projective variety $X$ with
 $\Hom_{\mathscr{O}_X}(E_i,E_i)=\kk$ for $0\leq i\leq m$, e.g., each
 $E_i$ is a line bundle. 
 \begin{enumerate}
 \item[\one] If $\bigoplus_{i=0}^m E_i$ satisfies conditions (T1)
   and (T2) then, by reordering if necessary, the sequence
   $(E_0,E_1,\dots,E_m)$ is strongly exceptional; if in addition
   $\bigoplus_{i=0}^m E_i$ satisfies (T3) then $(E_0,E_1,\dots,E_m)$ is a full strongly
   exceptional sequence.
 \item[\two] Conversely, every such full strongly exceptional sequence
   defines a tilting bundle.
 \end{enumerate}
 \end{proposition}
 \begin{proof}
 For \one, since $T:=\bigoplus_{i=0}^m E_i$ satisfies (T2) we have
 \begin{equation}
 \label{eqn:T2}
 0 = \Ext^\ell_{\mathscr{O}_X}(T,T) =
 \bigoplus_{i,j}\Ext^\ell_{\mathscr{O}_X}(E_i,E_j) 
 \end{equation}
 for $\ell>0$, which gives the required higher Ext-vanishing.
 The assumption $\Hom_{\mathscr{O}_X}(E_i,E_i)=\kk$ implies that each
 $E_i$ is exceptional. For $0\leq i\neq j\leq m$, one of the vector
 spaces $\Hom_{\mathscr{O}_X}(E_i,E_j)$,
 $\Hom_{\mathscr{O}_X}(E_j,E_i)$ must be trivial, otherwise
 $\Hom_{\mathscr{O}_X}(E_i,E_i)\neq \kk$ which is absurd. We may
 relabel if necessary to ensure that $i<j$ whenever
 $\Hom_{\mathscr{O}_X}(E_i,E_j)\neq 0$. The resulting sequence is
 strongly exceptional. The second statement from part \one\ is
 immediate.
 
 For \two, equation \eqref{eqn:T2} guarantees that a strongly
 exceptional sequence $(E_0,E_1,\dots,E_m)$ gives a locally free sheaf
 $T:= \bigoplus_i E_i$ satisfying (T2).  Set $A:=
 \End_{\mathscr{O}_X}(T)$. Since $X$ is projective, the $\kk$-vector
 spaces $\Hom(E_i,E_j) \cong H^0(E_j\otimes E_i^{-1})$ each have
 finite dimension, and hence $A$ is a finite dimensional
 $\kk$-algebra. Then $A$ is isomorphic to the quotient algebra $\kk
 Q/\langle \varrho\rangle$ of an acyclic bound quiver $(Q,\varrho)$
 with vertex set $Q_0=\{0,1,\dots,m\}$. Since $\Hom(E_i,E_j)=0$ for
 $i> j$, this algebra can be realised as an algebra of lower
 triangular matrices, from which it follows that $A$ has finite global
 dimension. Given (T1) and (T2), the fullness of the sequence then
 implies that $T= \bigoplus_i E_i$ satisfies (T3).
 \end{proof}

 \subsection{Resolution of the diagonal}
  Let $T$ be a locally-free sheaf on $X$ and set
 $A=\End_{\mathscr{O}_X}(T)$. If we regard $T$ as a sheaf of left
 $A$-modules then the (derived) dual $T^\vee=
 \mathcal{H}om(T,\mathscr{O}_X)$ is a sheaf of right $A$-modules. By
 pulling back via the first and second projections $\pi_i\colon
 X\times X\to X$, the sheaves $\pi_2^*(T)$ and $\pi_1^*(T^\vee)$ are
 naturally sheaves of left and right $A$-modules
 respectively. Consider the object of $D^b\big(\!\coh(X\times X)\big)$
 given by
 \[
 T^\vee\overset{\mathbf{L}}{\boxtimes}_AT:=
 \pi_1^*(T^\vee)\ltensor_A\pi_2^*(T).
 \]
 To understand this object, consider an object $M\in \modA$ and let
 $P^\bullet_M$ be a projective resolution of $M$ in $\modA$. Then we have
 \[
 \pi_1^*(T^\vee)\ltensor_A M = \pi_1^*(T^\vee)\otimes_A P^\bullet_M.
 \]
 Similarly, for an object $N\in \modAop$ we calculate
 $N\ltensor_A\pi_2^*(T)$ by replacing $N$ by a projective resolution
 of $N$ in $\modAop$. Putting this together, if we let $P^\bullet_A$
 be a projective resolution of $A$ in the category of left
 $A\otimes_\kk A^{\opp}$-modules then the isomorphism $M\otimes_A
 N\cong M\otimes_A A\otimes_A N$ enables us to compute
 \[
 T^\vee\overset{\mathbf{L}}{\boxtimes}_AT:=
 \pi_1^*(T^\vee)\otimes_A P^\bullet_A\otimes_A\pi_2^*(T).
 \]
 Armed with this observation, King~\cite{King2} established the
 relation between tilting bundles on smooth projective varieties and
 the celebrated resolution of the diagonal by
 Be{\u\i}linson~\cite{Beilinson}.

 \begin{proposition}
 \label{prop:resdiagonal}
 Let $T$ be a locally-free sheaf on $X$ satisfying (T1) and (T2).  If
 $T^\vee\overset{\mathbf{L}}{\boxtimes}_AT\rightarrow
 \mathscr{O}_\Delta$ is an isomorphism in $D^b\big(\!\coh(X\times
 X)\big)$ then $T$ is a tilting bundle.
 \end{proposition}
 \begin{proof}
 If (T1) and (T2) hold then the proof of Theorem~\ref{thm:baerbondal}
 shows (T3) holds if and only if
 \[
 \Rderived\Hom_{\mathscr{O}_X}(T,E)\ltensor_A
 T \cong E
 \quad\text{for every object }E\in
 D^b\big(\!\coh(X)\big).
 \]
 Since $T^\vee=\Rderived \Hom(T,\mathscr{O}_X)$,
 Exercise~\ref{ex:naturalisoms} parts \one\ and \two\ give
 \[
 \Rderived \Hom(T,E)\ltensor_A
 T\cong \Rderived \Gamma(E\ltensor_{\mathscr{O}_{X}} T^\vee)\ltensor_{A}T\cong  \Rderived(\pi_2)_*
   \Big{(}\pi_1^*(E\ltensor_{\mathscr{O}_{X}} T^\vee)\Big{)}\ltensor_{A}T.
 \]
 The projection formula (for sheaves of $A$-modules) and a standard
 result for derived tensor products under pullback gives
 \begin{eqnarray*}
 \Rderived(\pi_2)_* \Big{(}\pi_1^*(E\ltensor_{\mathscr{O}_{X}}
 T^\vee)\Big{)}\ltensor_{A}T & \cong & \Rderived(\pi_2)_*
 \Big{(}\pi_1^*(E\ltensor_{\mathscr{O}_{X}} T^\vee)\ltensor_{A}\pi_2^*(T)\Big{)} \\
 & \cong &
 \Rderived(\pi_2)_*\Big{(}\pi_1^*(E)\ltensor_{\mathscr{O}_{X\times X}}
 \pi_1^*(T^\vee)\overset{\mathbf{L}}{\otimes}_A \pi_2^*(T) \Big{)},
 \end{eqnarray*}
 giving $\Rderived \Hom(T,E)\ltensor_A T\cong
 \Rderived(\pi_2)_*\Big{(}\pi_1^*(E)\ltensor_{\mathscr{O}_{X\times X}}
 T^\vee\overset{\mathbf{L}}{\boxtimes}_AT\Big{)}$.  Since
 $T^\vee\overset{\mathbf{L}}{\boxtimes}_AT$ is isomorphic to
 $\mathscr{O}_\Delta$, the result follows from
 Proposition~\ref{prop:diagonal}.
 \end{proof}

 \begin{example}[Be{\u\i}linson's resolution]
 \label{exa:Beilinson}
 For $X=\mathbb{P}_\kk^n$, consider the sequence
 \[
 \mathcal{L}=\big(\mathscr{O}_{\mathbb{P}^n_\kk},\mathscr{O}_{\mathbb{P}^n_\kk}(1),\dots,
   \mathscr{O}_{\mathbb{P}^n_\kk}(n)\big)
 \]
 and the locally free sheaf $T=\bigoplus_{i=0}^n
   \mathscr{O}_{\mathbb{P}^n_\kk}(i)$.  Proposition~\ref{prop:algebra}
   and Exercise~\ref{ex:Beilinsoncqos} establish that the algebra
   $A=\End_{\mathbb{P}^n_\kk}\big(T\big)$ is isomorphic to the
   quotient algebra $\kk Q/\langle \varrho\rangle$ for the bound
   Be{\u\i}linson quiver $(Q,\varrho)$ from
   Example~\ref{ex:boundBeilinson}. Conditions (T1) and (T2) can both
   be verified explicitly. Condition (T3) also holds, but this is
   harder to verify; it follows from the celebrated result of
   Be{\u\i}linson~\cite{Beilinson}, but it also follows from
   Proposition~\ref{prop:resdiagonal}. We now sketch this argument for
   $\mathbb{P}^2_\kk$, leaving aside the (admittedly important)
   details of the maps below.  Using Butler--King~\cite{ButlerKing} we
   calculate the projective resolution of $A$ in the category of left
   $A^e:=A\otimes_\kk A^{\opp}$-modules
 \[
 \begin{array}{ccccc}
 & &                                                      & & A^e(e_0 \otimes e_0^{\opp}) \\
 & & \big{(} A^e(e_0 \otimes e_1^{\opp})\big{)}^{\oplus 3}& & \oplus \\
 \big{(}A^e(e_2\otimes e_0^{\opp} )\big{)}^{\oplus 3} & \stackrel{\delta_2}{\longrightarrow} & \oplus &
 \stackrel{\delta_1}{\longrightarrow} & A^e (e_1 \otimes e_1^{\opp})\\
 & & \big{(}A^e(e_2 \otimes e_1^{\opp})\big{)}^{\oplus 3}& & \oplus \\
 &&&& A^e(e_2 \otimes e_2^{\opp}) 
 \end{array}
 \]
 Having $T=\mathscr{O}_{\mathbb{P}^2}\oplus
 \mathscr{O}_{\mathbb{P}^2}(1)\oplus\mathscr{O}_{\mathbb{P}^2}(2)$
 gives $T^\vee= \mathscr{O}_{\mathbb{P}^2}\oplus
 \mathscr{O}_{\mathbb{P}^2}(-1)\oplus \mathscr{O}_{\mathbb{P}^2}(-2)$,
 and hence the complex $\pi_1^*(T^\vee)\otimes_A P^\bullet_A\otimes_A
 \pi_2^*(T)$ is
 \[
 \begin{array}{ccccc}
 &&&&\mathscr{O}\boxtimes \mathscr{O} \\
& & \big{(}\mathscr{O}(-1)\boxtimes \mathscr{O}\big{)}^{\oplus
   3}& & \oplus \\
 \big{(}\mathscr{O}(-2)\boxtimes \mathscr{O}\big{)}^{\oplus
   3}&\stackrel{\delta_2}{\longrightarrow} & \oplus &
 \stackrel{\delta_1}{\longrightarrow} &\mathscr{O}(-1)\boxtimes
 \mathscr{O}(1) \\
 & & \big{(}\mathscr{O}(-2)\boxtimes \mathscr{O}(1)\big{)}^{\oplus
   3}& & \oplus \\
 &&&&\mathscr{O}(-2)\boxtimes \mathscr{O}(2) 
 \end{array}
 \]
 One can verify that this complex provides a resolution of
 $\mathscr{O}_\Delta$, so (T3) holds by
 Proposition~\ref{prop:resdiagonal}. Thus, $T$ is tilting and the
 derived equivalences
 \[
 D^b\big(\!\coh(X)\big)\cong D\big(\!\modAop\big)\cong D\big(\!\rep_\kk(Q,\varrho)\big)
 \]
 hold by Theorem~\ref{thm:baerbondal} and
 Proposition~\ref{prop:boundisomcats} (with the duality
 equivalence $\modA\cong \modAop$).
 \end{example}

 \begin{remark}
 To compare the resolution
 $T^\vee\overset{\mathbf{L}}{\boxtimes}_AT\longrightarrow
 \mathscr{O}_\Delta$ on $\mathbb{P}^2_\kk$ with that given by
 Be{\u\i}linson, use the Euler sequences
 \begin{equation*}
 \xymatrix{ 0\ar[r] &  \mathscr{O}(-1)\boxtimes \Omega^1(1)\ar[r]   &
   \mathscr{O}(-1)\boxtimes \mathscr{O}^{\oplus 3}\ar[r]   &  \mathscr{O}(-1)\boxtimes \mathscr{O}(1)\ar[r]   & 0 }
 \end{equation*}
 and
 \begin{equation*}
 \xymatrix{ 0\ar[r] &  \mathscr{O}(-2)\boxtimes \Omega^2(2)\ar[r]   &
   \mathscr{O}(-2)\boxtimes \mathscr{O}^{\oplus 3}\ar[r]   &  \mathscr{O}(-2)\boxtimes \Omega^1(2)\ar[r]   & 0 }
 \end{equation*}
 to see that
 $T^\vee\overset{\mathbf{L}}{\boxtimes}_AT\rightarrow
 \mathscr{O}_\Delta$ is quasi-isomorphic to Be{\u\i}linson's resolution
  \begin{equation*}
 \xymatrix{0\ar[r] &  
 \mathscr{O}(-2)\boxtimes \Omega^2(2)\ar[r] &  
 \mathscr{O}(-1)\boxtimes \Omega^1(1)\ar[r] &  
 \mathscr{O}\boxtimes \mathscr{O}\ar[r] &   \mathscr{O}_\Delta\ar[r] &   0.}
 \end{equation*}
 The calculation for the tilting bundle on $\mathbb{P}^n_\kk$ is similar. 
 \end{remark}

 \subsection{Smooth toric Fanos}
 The techniques introduced above were exploited by King~\cite{King2}
 to construct tilting bundles on each of the 5 smooth toric del Pezzo
 surfaces (see Example~\ref{ex:delPezzo}). The bound Be{\u\i}linson quiver
 that encodes $D^b\big(\!\coh(\mathbb{P}^2_\kk)\big)$ is described in
 Example~\ref{exa:Beilinson}, and the bound quivers $(Q,\varrho)$
 encoding the derived categories for the other 4 examples are
 shown below.

 We first note several factors that are common to each example. For
 every smooth toric del Pezzo surface $X$, King's tilting bundle is a
 direct sum of basepoint-free line bundles $\bigoplus_{i=0}^r L_i$ on
 $X$ with $L_0=\mathscr{O}_X$. If we write $(Q,\varrho)$ for the
 complete bound quiver of sections for the sequence
 $\mathcal{L}=(\mathscr{O}_X, L_1,\dots, L_r)$ as in
 Section~\ref{sec:qos}, then Proposition~\ref{prop:algebra} shows that
 the noncommutative algebra $A=\End_{\mathscr{O}_X}(T)$ is isomorphic
 to the quotient algebra $\kk Q/\langle\varrho\rangle$. Derived
 equivalences
 \begin{equation}
 \label{eqn:tilting}
 D^b\big(\!\coh(X)\big)\cong D\big(\!\modAop\big)\cong D\big(\!\rep_\kk(Q,\varrho)\big).
 \end{equation}
 follow from Theorem~\ref{thm:baerbondal} and
 Proposition~\ref{prop:boundisomcats} (with the equivalence
 $\modA\cong \modAop$ coming from duality). We refer to the bound
 quiver $(Q,\varrho)$ as a \emph{tilting quiver}.

 \begin{example}
 For $X= \mathbb{P}^1\times \mathbb{P}^1$, consider the sequence
 $\mathcal{L}=\big(\mathscr{O}_X,\mathscr{O}_X(1,0),\mathscr{O}_X(0,1),\mathscr{O}_X(1,1)\big)$,
 where $\mathscr{O}_X(1,0)$ and $\mathscr{O}_X(0,1)$ are the pullbacks
 of $\mathscr{O}_{\mathbb{P}^1_\kk}(1)$ via the first and second
 projections respectively. The complete quiver of sections $Q$ for
 $\mathcal{L}$ is shown in Figure~\ref{fig:tiltingquivers}(a);
 \begin{figure}[!ht]
    \centering \mbox{ \subfigure[$\mathbb{P}^{1}\times \mathbb{P}^1$]{ \psset{unit=1.2cm}
        \begin{pspicture}(-0.5,-0.3)(2.5,2.3)
          \cnodeput(0,0){A}{0}
          \cnodeput(2,0){B}{1} 
          \cnodeput(0,2){C}{2}
          \cnodeput(2,2){D}{3}
          \psset{nodesep=0pt}
          \ncline[offset=5pt]{->}{A}{B} \lput*{:U}{$x_1$}
          \ncline[offset=-5pt]{->}{A}{B} \lput*{:U}{$x_3$}
          \ncline[offset=7pt]{->}{A}{C} \lput*{:270}{$x_2$}
          \ncline[offset=-7pt]{->}{A}{C} \lput*{:270}{$x_4$}
          \ncline[offset=7pt]{->}{B}{D} \lput*{:270}{$x_2$}
          \ncline[offset=-7pt]{->}{B}{D} \lput*{:270}{$x_4$}
          \ncline[offset=5pt]{->}{C}{D} \lput*{:U}{$x_1$}
          \ncline[offset=-5pt]{->}{C}{D} \lput*{:U}{$x_3$}
        \end{pspicture} 
 }\quad\quad
 \subfigure[$\text{Bl}_{p,q}(\mathbb{P}^2)$]{ 
 \psset{unit=1.2cm}
        \begin{pspicture}(-0.5,-0.3)(6,2.3)
          \cnodeput(0,1){A}{0}
          \cnodeput(2,0){B}{1} 
          \cnodeput(2,2){C}{2}
          \cnodeput(4,1){D}{3}
          \cnodeput(5.5,1){E}{4}
          \psset{nodesep=0pt}
          \ncline[offset=6pt]{->}{A}{B} \lput*{:U}{$x_1$}
          \ncline[offset=-6pt]{->}{A}{B} \lput*{:U}{$x_3x_4$}
          \ncline[offset=6pt]{->}{A}{C} \lput*{:U}{$x_5$}
          \ncline[offset=-6pt]{->}{A}{C} \lput*{:U}{$x_2x_3$}
          \ncline{->}{B}{D} \lput*{:U}{$x_2$}
          \ncline{->}{C}{D} \lput*{:U}{$x_4$}
          \nccurve[angleA=0,angleB=220]{->}{B}{E} \lput*{:U}{$x_5$}
          \nccurve[angleA=0,angleB=140]{->}{C}{E} \lput*{:U}{$x_1$}
          \ncline{->}{D}{E} \lput*{:U}{$x_3$}
        \end{pspicture} }
    }
    \caption{Tilting quivers on a pair of toric del Pezzo surfaces}\label{fig:tiltingquivers}
  \end{figure}
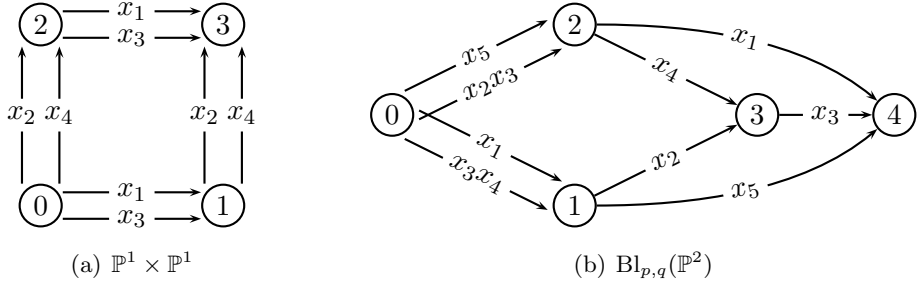
 here we label each arrow by the monomial $x^{\div(a)}$ in the Cox
 ring of $X$ associated to the toric fan from
 Figure~\ref{fig:Fanosurfaces}(a). This labelling makes it
 straightforward to read off the natural set of relations $\varrho$
 defined in \eqref{eqn:varrho}.
 \end{example}

 \begin{example}
   For $X= \mathbb{F}_1$, a tilting quiver $(Q,\varrho)$ is provided
   by the complete bound quiver of sections of the sequence
   $\mathcal{L}$ from Example~\ref{ex:F1four}.
 \end{example}
 
 \begin{example}
For $X= \text{Bl}_{p,q}(\mathbb{P}^2)$, list the $T_X$-invariant
 divisors on $X$ according to the fan of
 Figure~\ref{fig:Fanosurfaces}(c) and consider the sequence
 $\mathcal{L}=\big(\mathscr{O}_X,\mathscr{O}_X(D_1),\mathscr{O}_X(D_5),\mathscr{O}_X(D_1+D_2),\mathscr{O}_X(D_1+D_5)\big)$. The
 complete quiver of sections $Q$ for $\mathcal{L}$ is shown in
 Figure~\ref{fig:tiltingquivers}(b), and the natural relations $\varrho$
 can be read off from the labelling by monomials using
 \eqref{eqn:varrho}.
 \end{example}
 
 \begin{example}
 For $X= \text{Bl}_{p,q,r}(\mathbb{P}^2)$, list the $T_X$-invariant
 divisors on $X$ according to the fan in
 Figure~\ref{fig:Fanosurfaces}(d) and consider the sequence
 $\mathcal{L}=\big(\mathscr{O}_X,\mathscr{O}_X(D_1+D_2),\mathscr{O}_X(D_3+D_4),\mathscr{O}_X(D_5+D_6),$ $\mathscr{O}_X(D_1+D_2+D_3),\mathscr{O}_X(D_4+D_5+D_6)\big)$
 of basepoint-free line bundles on $X$.  The complete quiver of
 sections $Q$ for $\mathcal{L}$ is shown in
 Figure~\ref{fig:BlpqrP2tilting}, and the relations $\varrho$ can be
 read off from the labelling by monomials using \eqref{eqn:varrho}.
 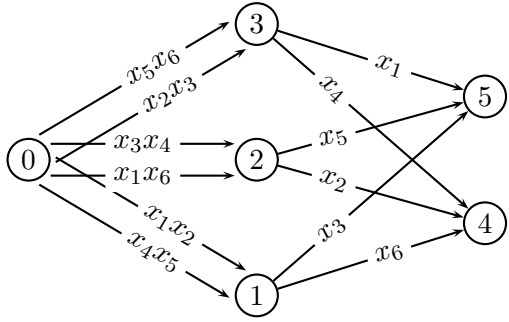
\begin{figure}[!ht]
    \centering \psset{unit=1.2cm}
        \begin{pspicture}(0,0)(6,3.5)
          \cnodeput(0,1.5){A}{0} \cnodeput(2.5,0){B}{1}
          \cnodeput(2.5,1.5){C}{2} \cnodeput(2.5,3){D}{3}
          \cnodeput(5,0.8){E}{4} \cnodeput(5,2.2){F}{5}
          \psset{nodesep=0pt} \ncline[offset=6pt]{->}{A}{B}
          \lput*{:U}(0.6){$x_1x_2$} \ncline[offset=-6pt]{->}{A}{B}
          \lput*{:U}(0.6){$x_4x_5$} \ncline[offset=6pt]{->}{A}{C}
          \lput*{:U}{$x_3x_4$} \ncline[offset=-6pt]{->}{A}{C}
          \lput*{:U}{$x_1x_6$} \ncline[offset=6pt]{->}{A}{D}
          \lput*{:U}(0.6){$x_5x_6$} \ncline[offset=-6pt]{->}{A}{D}
          \lput*{:U}(0.6){$x_2x_3$} \ncline{->}{B}{E}
          \lput*{:U}(0.6){$x_6$} \ncline{->}{B}{F}
          \lput*{:U}(0.3){$x_3$} \ncline{->}{C}{E}
          \lput*{:U}(0.3){$x_2$} \ncline{->}{C}{F}
          \lput*{:U}(0.3){$x_5$} \ncline{->}{D}{E}
          \lput*{:U}(0.3){$x_4$} \ncline{->}{D}{F}
          \lput*{:U}(0.6){$x_1$}
        \end{pspicture} 
   \caption{A tilting quiver on $\text{Bl}_{p,q,r}(\mathbb{P}^2)$}
 \label{fig:BlpqrP2tilting}
  \end{figure}
 \end{example}
 
 \begin{exercise}
 Use the results of Section~\ref{sec:noncommutative} to show that each
 smooth toric del Pezzo surface $X$ is isomorphic to
 $\mathcal{M}_\vartheta(Q,\varrho)$ for the relevant bound quiver
 $(Q,\varrho)$ and for the weight $\vartheta\in \Wt(Q)$ satisfying
 $\vartheta_i>0$ for $i\neq 0$. (The solution for $X=\mathbb{F}_1$ is
 the content of Example~\ref{ex:F1four}).
 \end{exercise}

  Costa--Mir\'{o}-Roig~\cite{CMR} extended King's results to construct
 tilting bundles on 12 of the 18 smooth toric Fano threefolds, and
 this list was extended by Bondal to 16 of the 18. At
 about the same time, Greg~Smith and I exhibited tilting bundles on
 all 18 smooth toric Fano threefolds and many of the 126 smooth toric
 Fano fourfolds.  This suggests the following:

 \begin{conjecture}
 \label{conj:toricFano}
 Let $X$ be a smooth toric Fano variety. There is a sequence of
 basepoint-free line bundles $\mathcal{L} = (\mathscr{O}_X,L_1,\dots,
 L_r)$ on $X$ with bound quiver of sections $(Q,\varrho)$ such that
 \begin{enumerate}
 \item[\one] $X$ is isomorphic to $\mathcal{M}_\vartheta(Q,\varrho)$ for
   the weight $\vartheta\in \Wt(Q)$ satisfying $\vartheta_i>0$ for
   $i\neq 0$;
 \item[\two] the tautological bundle $\bigoplus_{0\leq i\leq r} L_i$ on
   $X$ is tilting, giving the derived
   equivalences \eqref{eqn:tilting}.
 \end{enumerate}
 \end{conjecture}

\begin{remark}
 A couple of important remarks are in order:
 \begin{enumerate}
 \item[\one] The Fano condition is not necessary. For example,
 Costa--Mir\'{o}-Roig~\cite{CMR} constructed tilting bundles on
 several families of smooth toric varieties, none of which consists
 entirely of Fano varieties. Nevertheless, some restriction is
 required, since the example of Hille--Perling~\cite{HillePerling}
 shows that even smooth projective toric surfaces need not admit a
 tilting bundle obtained as a direct sum of line bundles.
\item[\two] If Conjecture~\ref{conj:toricFano}\two\ holds, then one
  expects part \one\ to follow.
  Bergman--Proudfoot~\cite{BergmanProudfoot} show that if $X$ admits a
  tilting bundle for which the corresponding bound quiver
  $(Q,\varrho)$ admits a character `great' $\theta$, then $X$ is
  isomorphic to a connected component of
  $\mathcal{M}_\theta(Q,\varrho)$. On the other hand, getting the
  natural moduli construction correct first can also help enormously
  (see the $G$-Hilbert scheme and Theorem~\ref{thm:BKR}).
 \end{enumerate}
 \end{remark}

 \subsection{Smooth toric DM stacks}
 Since these lectures on derived categories were presented at the Utah
 summer school, Borisov--Hua~\cite{BorisovHua} extended
 Conjecture~\ref{conj:toricFano}\two\ further to include all nef-Fano
 toric DM stacks with trivial generic stabiliser. 
 
 That derived category questions for a projective variety $X$ with at
 worst orbifold singularities should be tackled using a stack
 $\mathcal{X}$ whose coarse moduli space is $X$ is due largely to a
 sequence of papers by
 Kawamata including \cite{Kawamata3,Kawamata4}. He
 generalised many fundamental results on derived categories of smooth
 projective varieties due to Orlov and Bondal--Orlov by replacing
 $\coh(X)$ by the abelian category $\coh(\mathcal{X})$ of coherent
 orbifold sheaves on $\mathcal{X}$. By applying these results with the
 minimal model programme for toric varieties,
 Kawamata~\cite{Kawamata5} established the strongest results known at
 present for the derived category of a general class of toric
 varieties.

 To describe the result, let $X$ be a projective toric variety whose
 fan $\Sigma\subset N\otimes_\ZZ \QQ$ is simplicial, so $X$ has at
 worst orbifold singularities by Remark~\ref{rem:smooth}.  Recall from
 Theorem~\ref{thm:Cox} that $X$ is obtained as the geometric quotient
 of $U:= \mathbb{A}^{\Sigma(1)}_\kk\setminus \mathbb{V}(B_X)$ by an action
 of $\Hom(A_{n-1}(X),\kk^\times)$. Borisov--Chen--Smith~\cite{BCS}
 define a collection of smooth Deligne--Mumford stacks each with
 coarse moduli space $X$ as follows. Let $\beta\colon
 (\ZZ^{\Sigma(1)})^\vee\to N$ be a $\ZZ$-linear map that picks out a
 lattice point $v_\rho=\beta(e_\rho)$ on each ray $\rho\in \Sigma(1)$,
 and let $\Gale(N)$ denote the cokernel of the dual map $\beta^\vee
 \colon N^\vee \to \ZZ^{\Sigma(1)}$. Note that $\Gale(N)=A_{n-1}(X)$
 if each $v_\rho$ is the primitive generator of $\rho\cap N$. The
 diagonalisable algebraic group $G:=\Hom(\Gale(N),\kk^\times)$ acts on
 $U$, and the \emph{smooth toric DM stack} of $(\Sigma,\beta)$ is
 defined to be the stack quotient $\mathcal{X}(\Sigma,\beta):=
 [U/G]$. Our assumption that the abelian group $N$ is free implies
 that $\mathcal{X}(\Sigma,\beta)$ has trivial generic stabiliser. 

 The main result of Kawamata~\cite{Kawamata5} can be stated as
 follows (compare Proposition~\ref{prop:exceptionaltilting}).
 
 \begin{theorem}
 \label{thm:KawamataToric}
 Let $X$ be a projective toric variety with simplicial fan $\Sigma$,
 choose $\beta$ as above and set $\mathcal{X}:=
 \mathcal{X}(\Sigma,\beta)$. Then $D^b\big(\!\coh(\mathcal{X})\big)$
 admits a full exceptional sequence of sheaves.
 \end{theorem}

 The sheaves in the exceptional sequence need not be line bundles, and
 the sequence need not a priori be strong, so we do not in general obtain the
 derived equivalences \eqref{eqn:tilting}. 

 \begin{exercise}
   This exercise exhibits a simple link between the minimal model
   programme and derived categories. We start with
   $\text{Bl}_{p,q,r}(\mathbb{P}^2)$ whose fan is shown in
   Figure~\ref{fig:Fanosurfaces}(d).
 \begin{enumerate}
 \item[\one] Consider the blow-down of the $(-1)$-curve labelled
   $D_6$ in $\text{Bl}_{p,q,r}(\mathbb{P}^2)$. In terms of the tilting
   quiver $Q$ from Figure~\ref{fig:BlpqrP2tilting}, remove $x_6$
   whenever it appears in the label on an arrow; this forces us to
   contract the arrow
   from $1\in Q_0$ to $4\in Q_0$ and hence
   identify vertices 1 and 4.  After removing the arrows in the
   resulting quiver that are not irreducible (see
   Section~\ref{sec:qos}), show that we obtain the tilting quiver on
   $\text{Bl}_{p,q}(\mathbb{P}^2)$ from
   Figure~\ref{fig:tiltingquivers}.
 \item[\two] Continue by contracting successively the $(-1)$-curves to
 produce the tilting quiver from Figures~\ref{fig:F1tilting} and then
 the bound Be{\u\i}linson quiver for $\mathbb{P}^2_\kk$ (up to a
 relabelling of the divisors).
 \end{enumerate}
 \end{exercise}

 \begin{remark}
 Kawamata~\cite{Kawamata5} defines his stack $\mathcal{X}$ in terms of
 pairs $(X,B)$, where $B$ is a $T_X$-invariant $\QQ$-divisor of the
 form $\sum_{\rho\in \Sigma(1)} (1-\frac{1}{r_\rho})D_\rho$ for some
 positive integers $\{r_\rho : \rho\in \Sigma(1)\}$. This choice of
 $B$ is equivalent to choosing the map $\beta\colon
 (\ZZ^{\Sigma(1)})^\vee\to N$ for which $v_\rho$ is precisely $r_\rho$
 times the primitive generator of $\rho\cap N$. Indeed, if we let
 $\mathcal{D}_\rho$ denote the prime divisor in $\mathcal{X}$
 corresponding to $\rho\in \Sigma(1)$, then the morphism to the coarse
 moduli space $\pi\colon \mathcal{X}\to X$ satisfies $\pi^*(D_\rho) =
 r_\rho \mathcal{D}_\rho$ in each case. In particular, $B=0$ if and
 only if $\beta$ chooses the primitive lattice generator in each ray,
 in which case $\mathcal{X}(\Sigma,\beta)$ is the canonical covering stack
 of $X$.
 \end{remark}
 
 Borisov--Hua~\cite{BorisovHua} propose that an analogue of
 Conjecture~\ref{conj:toricFano}\two\ should hold for toric DM stacks
 $\mathcal{X}(\Sigma,\beta)$ that are \emph{nef-Fano}, i.e., a
 positive power of the anticanonical bundle on $\mathcal{X}$ is the
 pullback of a nef and big line bundle on $X$.  As evidence, they
 construct full strong exceptional collections of line bundles on all
 smooth Fano toric DM stacks of dimension two and Picard rank at most
 three, as well as on all smooth Fano toric DM stacks of Picard number
 at most two (generalising the result for toric varieties of
 Costa--Mir\'{o}-Roig~\cite{CMR}).

 \section{The derived McKay correspondence for $\ghilb$}
 \label{sec:McKay}
 In this section we apply derived category methods to provide an
 elegant explanation for the McKay correspondence in dimension $n\leq
 3$.  Rather than pursue further the construction of tilting bundles
 explicitly (see Van den Bergh~\cite{VDB}), we describe in some detail
 the original method of Bridgeland--King--Reid~\cite{BKR} that uses
 the universal sheaf on the $G$-Hilbert scheme to construct an
 equivalence of derived categories as a Fourier--Mukai transform. The
 varieties that we consider in this section are all normal
 semiprojective toric varieties if and only the finite subgroup
 $G\subset \SL(n,\kk)$ is abelian. Imposing this restriction does not
 simply the proof greatly, so we do not restrict to the toric case
 here.

 \subsection{A few words on the McKay correspondence}  
 Let $G\subset \SL(n,\kk)$ be a finite subgroup. We may assume that
 $G$ acts without quasireflections, in which case the quotient
 $X:=\mathbb{A}^n_\kk/G$ is Gorenstein, i.e., the canonical sheaf
 $\omega_X$ is locally free. The globally defined nonvanishing form
 $dx_1\wedge \dots \wedge dx_n$ on $\mathbb{A}^n_\kk$ is $G$-invariant
 and hence it descends to $X$, forcing $\omega_X$ to be trivial. A
 resolution of singularities $\tau\colon Y\to X$ is \emph{crepant} if
 $\tau^*(\omega_X)=\omega_Y$.

 \begin{example}
   The motivating examples of crepant resolutions are the minimal
   resolutions of Kleinian surface singularities, i.e., quotients
   $\mathbb{A}^2_\kk/G$ for a finite subgroup $G\subset \SL(2,\kk)$.
   These crepant resolutions provide an ADE classification for
   Kleinian singularities; the resolution of the $A_2$-singularity is
   described in Example~\ref{exa:res1/3(1,2)}.
 \end{example}
 
 In higher dimensions, crepant resolutions need not exist a priori,
 and when they do they are typically nonunique. Nevertheless, by
 introducing the $G$-Hilbert scheme $Y:= \ghilb(\mathbb{A}^n_\kk)$,
 Ito--Nakamura~\cite{ItoNakamura} and Nakamura~\cite{Nakamura}
 provided a candidate for a crepant resolution, at least in dimension
 two and three. Recall from Remark~\ref{rem:G-Hilbert} that the
 $G$-Hilbert scheme parametrises (coordinate rings of) $G$-homogeneous
 ideals $I\subset \kk[x_1,\dots,x_n]$ for which the quotient module
 $\kk[x_1,\dots,x_n]/I$ is isomorphic as a $\kk[G]$-module to
 $\kk[G]$. The subschemes $\mathbb{V}(I)\subset \mathbb{A}^n_\kk$
 defined by such ideals are called \emph{$G$-clusters}. Examples
 include any free $G$-orbit, all of which define points in a single
 irreducible component of $Y$ (see \cite[\S5]{CMT1}). The map sending
 a $G$-cluster to its supporting $G$-orbit defines a projective and
 surjective Hilbert--Chow morphism $\tau\colon Y \rightarrow
 \mathbb{A}_\kk^n/G$ and hence, irrespective of whether $Y$ is a
 crepant resolution of $X=\mathbb{A}^n_\kk/G$, the map $\tau$ fits into
 a commutative diagram of the form
 \begin{equation}
 \label{eqn:cdFM}
 \xymatrix@C=.6em{&&& Y\times
   \mathbb{A}_\kk^n\ar[dlll]_{\pi_Y}\ar[drrr]^{\pi_V} \\
   Y\ar[drrr]^{\tau} &&&&&&\mathbb{A}_\kk^n\ar[dlll]_{\pi}\\&&& X }
 \end{equation}
 \noindent 
 where $\pi\colon V:=\mathbb{A}^n_\kk \to X$ is the quotient morphism
 and \(\pi_Y\) and \(\pi_V\) are the projections to the first and
 second factors.  Let \(G\) act trivially on both \(Y\) and \(X\), so
 that all morphisms in the above diagram are \(G\)-equivariant. The
 universal sheaf on \(Y\times \mathbb{A}_\kk^n\) is the structure
 sheaf \(\mathscr{O}_{\mathcal{Z}}\) of the universal closed subscheme
 \(\mathcal{Z}\subset Y\times \mathbb{A}_\kk^n\) whose restriction to
 the fibre over a point \([I]\in Y\) is the subscheme
 \(\mathbb{V}(I)\subset \mathbb{A}_\kk^n\). The fibre of $\pi_Y$ over
 each closed point is isomorphic to $\kk[G]$ as a $\kk[G]$-module. If
 $\Irr(G)$ denotes the set of isomorphism classes of the irreducible
 representations of $G$, then $(\pi_Y)_*(\mathcal{O}_\mathcal{Z})$ is
 a $G$-equivariant locally free sheaf on $Y$ that decomposes into a
 sum of locally free sheaves
 \[
 \mathscr{W}=\bigoplus_{\rho\in \Irr(G)} \mathscr{W}_\rho
   ^{\oplus \dim(\rho)}
 \]
 according to the irreducible decomposition of $\kk[G]$, where
   $\rank(\mathscr{W}_\rho)=\dim(\rho)$. 

 \begin{remark}
   Recall from Remark~\ref{rem:G-Hilbert} (compare
   Remark~\ref{rem:BSW}) that the tautological bundle $\mathscr{W}$ on
   the $G$-Hilbert scheme $Y$ is isomorphic a a $G$-equivariant
   locally free sheaf to the tautological sheaf on the fine moduli space
   $\mathcal{M}_\vartheta(Q,\varrho)$ of $\vartheta$-stable
   representations of the bound McKay quiver. An analogous statement
   holds if $G$ is not abelian, see
   Ito--Nakajima~\cite[Section~2]{ItoNakajima}.
 \end{remark}

 \begin{exercise}
 Read Reid~\cite[Theorem~4.11]{Reid1} and deduce that the cyclic
 quotient singularity of type $\frac{1}{2}(1,1,1,1)$ does not admit a
 crepant resolution. Thus, crepant resolutions of Gorenstein quotient
 singularities need not exist in dimension 4 and higher.
 \end{exercise}

 Motivated by the geometric explanation of the McKay correspondence in
 dimension two by Gonzalez--Sprinberg and Verdier~\cite{GSV}, Reid~\cite{Reid2, Reid3}
 formulated a derived category version of the McKay correspondence
 (see Aspinwall et.al.~\cite{ClayMonograph} for background on the
 McKay correspondence):

 \begin{conjecture}
 \label{conj:McKay}
 Let $G\subset \SL(n,\kk)$ be a finite subgroup, and let
 $\Gcoh(\mathbb{A}^n_\kk)$ denote the abelian category of
 $G$-equivariant coherent sheaves on $\mathbb{A}^n_\kk$.  If a crepant
 resolution \(\tau\colon Y\rightarrow \mathbb{A}_\kk^{n}/G\) exists
 then there is an equivalence of triangulated categories 
 \begin{equation}
 \label{eqn:derivedcats}
 \Phi\colon D^b(\coh(Y))\rightarrow D^b\big(\Gcoh(\mathbb{A}^n_\kk)\big).
 \end{equation}
 \end{conjecture}
 
 \subsection{The statement for the $G$-Hilbert scheme}  
 A coherent sheaf $\mathscr{F}$ on $\mathbb{A}_\kk^n$ is
 $G$-equivariant if and only if the space of sections
 $\Gamma(\mathscr{F})$ is a finitely-generated
 $\kk[x_1,\dots,x_n]$-module that carries an action by $G$ such that
 $g \cdot (sm)=(g \cdot s) (g \cdot m)$ for all $g \in G, m \in
 \Gamma(\mathscr{F})$ and $s \in \kk[x_1,\dots, x_n]$. Combining this
 with Proposition~\ref{prop:algebra} and Example~\ref{ex:boundMcKay}
 (or Remark~\ref{rem:BSW}) establishes the next result:
 
 \begin{lemma}
 \label{lem:abelianequivs}
 The following abelian categories are equivalent: the category
 \begin{enumerate}
 \item[\one] $\Gcoh(\mathbb{A}^n_\kk)$ of $G$-equivariant coherent
 sheaves on $\mathbb{A}^n_\kk$;
 \item[\two] $\modA$ of left-modules over the skew group algebra
 $A:=\kk[x_1,\dots,x_n]*G$;
 \item[\three] $\rep_\kk(Q,\varrho)$ of
 finite-dimensional representations of the bound McKay quiver.
 \end{enumerate}
 \end{lemma}
 
 Given this result, one approach to Conjecture~\ref{conj:McKay} is to
 construct a tilting bundle $T$ on a crepant resolution $Y$ for which
 $\End_{\mathscr{O}_Y}(T)$ is isomorphic to the algebra $A$.
 Reid~\cite{Reid2} proposed the tautological sheaf $\mathscr{W}$ on the
 $G$-Hilbert scheme as a candidate, at least in low dimensions, and
 King began to investigate whether the resolution of the diagonal from
 Proposition~\ref{prop:resdiagonal} could be established explicitly in
 such cases. However, with the timely completion of Bridgeland's
 thesis in the summer of 1998, Bridgeland, King and Reid together
 realised that an alternative approach is to establish the equivalence
 by using the universal sheaf for the $G$-Hilbert scheme to construct a
 $G$-equivariant version of a Fourier--Mukai transform (the same
 observation was made independently and carried out in dimension two
 by Kapranov--Vasserot~\cite{KapranovVasserot}).
 
 Define a functor $\Phi^{\mathscr{O}_{\mathcal{Z}}}\colon
   D^b\big(\!\coh(Y)\big)\rightarrow D^b\big(\modA\big)$ by setting
 \begin{equation}
 \label{eqn:FMformula}
 \Phi^{\mathscr{O}_{\mathcal{Z}}}(-) = {\Rderived}(\pi_V)_*\Big{(}\mathscr{O}_{\mathcal{Z}}\ltensor (\pi_Y)^*(-\otimes\rho_0)\Big{)}.
 \end{equation}
 The tensor product with the trivial representation acknowledges that
 \(G\) acts trivially on $Y$, enabling us to take the
 \(G\)-equivariant pullback via \(\pi_Y\). Therefore,
 $\Phi^{\mathscr{O}_{\mathcal{Z}}}$ is an equivariant version of an
 integral functor (see Huybrechts~\cite{Huybrechts}).  The main result
 of Bridgeland--King--Reid~\cite{BKR}, also known as
 ``$\text{Mukai}\implies\text{McKay}$'', can be stated as follows.

 \begin{theorem}
 \label{thm:BKR}
 Let $G\subset \SL(n,\kk)$ be a finite subgroup and let
 $A=\kk[x_1,\dots,x_n]*G$ denote the skew group algebra. If the irreducible
 component $Y\subseteq \ghilb(\mathbb{A}^n_\kk)$ containing the free $G$-orbits
 satisfies $\dim\big{(}Y\times_X Y\big{)}\leq n+1$, then:
 \begin{enumerate}
 \item[\one] the morphism \(\tau\colon
 Y\rightarrow X\) is a crepant resolution; and 
\item[\two] the functor $\Phi^{\mathscr{O}_{\mathcal{Z}}}\colon
  D^b\big(\!\coh(Y)\big)\rightarrow D^b(\modA)$ is an equivalence of
  derived categories.
 \end{enumerate}
 \end{theorem}

 \begin{corollary}
   Conjecture~\ref{conj:McKay} holds for surfaces, and for 3-folds
   when $Y=\ghilb(\mathbb{A}^3_\kk)$.
 \end{corollary}
 \begin{proof}
   The dimension of \(Y\times_X Y\) is at most twice the dimension of
   the exceptional locus; this is two for \(n = 2\) and four for
   \(n=3\).  Since crepant resolutions of of canonical surfaces are
   unique, Theorem~\ref{thm:BKR} settles completely the case $n=2$
   (this is due to Kapranov--Vasserot~\cite{KapranovVasserot}).
 \end{proof}

 \begin{exercise}
 Show that the condition on the dimension of the fibre product fails
 to hold for the unique crepant resolution of the singularity of type
 $\frac{1}{4}(1,1,1,1)$.
 \end{exercise}
 
 \begin{remark}
 Bridgeland--King--Reid~\cite{BKR} also show that $\ghilb(\mathbb{A}^3_\kk)$
 is connected for a finite subgroup $G\subset \SL(3,\kk)$, giving $Y\cong \ghilb(\mathbb{A}^3_\kk)$.
 \end{remark}

\subsection{First steps of the proof}  
 In this section and the next we describe much of the proof of
 Theorem~\ref{thm:BKR}. The original proof from~\cite{BKR} is very
 well written. My aim is simply to present a roadmap for the
 construction as I learned it from the authors while the paper was
 being written (and to shed more light on the rather tightly written
 \cite[\S6, \textrm{Step 5}]{BKR}). As in Section~\ref{sec:tilting},
 we refer to \cite{Caldararu, Huybrechts} for the necessary
 triangulated category definitions.
 
 To simplify the discussion, we make a couple of rather outrageous
 assumptions: assume
 \begin{enumerate}
 \item[(A1)] that $Y$ is projective and that Serre duality holds on
   $\mathbb{A}^n_\kk$. This is patently false (!);
 \item[(A2)] that $\tau\colon Y\rightarrow X$ is a crepant resolution.
   While this forms half of the result (!), it was known prior to
   \cite{BKR} for finite abelian subgroups $G\subset \SL(3,\kk)$ by
   Nakamura~\cite{Nakamura}.
 \end{enumerate}

 \begin{roadmap}
   The strategy of the proof is to invoke a key result that gives a
   criteria for an integral functor to be an equivalence. Since $Y$ is
   smooth by assumption (A2), the skyscraper sheaves $\{\mathscr{O}_y
   : y\in Y\}$ form a spanning class in $D^b\big(\!\coh(Y)\big)$. If
   we accept that the triangulated category $D^b\big(\!\modA\big)$ is
   indecomposable \cite[Lemma~4.2]{BKR}, then assumption (A1) enables
   us to state the result \cite[Theorem~2.4]{BKR} in the following
   form:

 \begin{theorem}
 \label{thm:Bridgeland}
 Let $Y$ be a smooth projective variety and $\Phi\colon
 D^b\big(\!\coh(Y)\big)\to D^b\big(\!\modA\big)$ an exact functor that
 admits a left adjoint. Then $\Phi$ is fully faithful if and only if
  \[
  \Hom(\Phi(\mathscr{O}_y),\Phi(\mathscr{O}_{y^\prime})[i]) \cong
  \Hom\big(\mathscr{O}_y,\mathscr{O}_{y^\prime}[i]\big)\quad\quad \text{for
  all }i\in \ZZ\text{ and }y, y^\prime \in Y.
 \]  
 If, in addition, $\Phi$ commutes with the Serre functors on the
 triangulated categories $D^b\big(\!\coh(Y)\big)$ and
 $D^b\big(\!\modA\big)$, then $\Phi$ is a derived equivalence.
 \end{theorem}
  
 \begin{remark}
 This result is motivated by earlier results of Mukai and
 Bondal--Orlov, though note that the form of the result used here
 requires that the $\Hom$ groups be tested for all $i\in \ZZ$ and all
 $y,y^\prime \in Y$ (compare
 Huybrechts~\cite[Remark~7.3]{Huybrechts}).
 \end{remark}

The first step is to ensure that the exact functor
 $\Phi^{\mathscr{O}_{\mathcal{Z}}}$ admits a left adjoint. This
 follows from the projectivity assumption (A1), since Neeman's
 $G$-equivariant version of Grothendieck duality enables us to repeat
 the original construction of adjoints due to Mukai (see
 \cite[Proposition~5.9]{Huybrechts}) to deduce that the functor
 $\Psi\colon D^b\big(\!\modA\big)\rightarrow D^b\big(\!\coh(Y)\big)$
 defined by
 \[
 \Psi(-):=
 \Big{[}\Rderived(\pi_Y)_*\Big{(}\mathscr{O}_{\mathcal{Z}}^\vee\ltensor
 \pi^*_V(-)[n]\Big{)}\Big{]}^G
 \]
 is left-adjoint to $\Phi^{\mathscr{O}_{\mathcal{Z}}}$ (we use the
 fact that $G\subset \SL(n,\kk)$ to deduce that $\pi_V^*(\omega_V)$ is
 trivial as a $G$-equivariant sheaf). Note that taking the
 $G$-invariant part is adjoint to the functor obtained by taking the
 tensor product with the trivial representation of $G$.

 \begin{exercise}
   For $\rho\in \Irr(G)$, let $\rho^*$ denote the contragradient
   representation and $\mathscr{O}_{\mathbb{A}_\kk^n}\otimes \rho^*$
   the $G$-equivariant sheaf on $\mathbb{A}^n_\kk$ for which
   $\Gamma(\mathscr{O}_{\mathbb{A}_\kk^n}\otimes \rho^*)$ is the
   $\kk[x_1,\dots,x_n]^G$-module consisting of the
   $\rho^*$-semiinvariant functions. Use Grothendieck duality to show
   that $\Psi(\mathscr{O}_{\mathbb{A}_\kk^n}\otimes \rho^*) =
   \mathscr{W}_\rho^\vee$.
 \end{exercise}


 As a second step, applying Theorem~\ref{thm:Bridgeland} shows that
 $\Phi^{\mathscr{O}_{\mathcal{Z}}}$ is fully faithful if and only if
 \begin{equation}
 \label{eqn:BO}
 \Hom\big{(}\Phi(\mathscr{O}_y),\Phi(\mathscr{O}_{y^\prime})[i]\big{)}  =
 \left\{\begin{array}{cl} 0 & \text{if }y\neq y^\prime\text{ or
   }i\not\in [0,n] \\ \Ext^i_{\mathscr{O}_Y}\!\big(\mathscr{O}_y,\mathscr{O}_y\big) & \text{otherwise }, \end{array}\right.
 \end{equation}
 for all $i\in \ZZ$ and $y, y^\prime \in Y$.  Using formula
 \eqref{eqn:FMformula} for the Fourier--Mukai transform, we obtain
 $\Phi(\mathscr{O}_{y})=\mathscr{O}_{Z_y}$ where $Z_y\subset
 \mathbb{A}_\kk^n$ is the $G$-cluster corresponding to the point $y\in
 Y$. Moreover, under the equivalence $\modA\cong \Gcoh(\mathbb{A}^n_\kk)$, the
 $\Hom$-groups in $D^b\big(\!\modA\big)$ correspond to the
 $G$-invariant part of $\Ext$-groups, so we work with the groups
 \[
 \Hom\big{(}\Phi(\mathscr{O}_y),\Phi(\mathscr{O}_{y^\prime})[i]\big{)}
 =\GExt^i_{\mathscr{O}_{\mathbb{A}^n_\kk}}(\mathscr{O}_{Z_y},\mathscr{O}_{Z_{y^\prime}}).
 \]

 The third step is to show that these groups vanish for $y\neq
 y^\prime$; this requires the assumption on the dimension on the fibre
 product. The $G$-clusters $Z_y$ and $Z_{y^\prime}$ define different
 points of the classical Hilbert scheme of points, giving
 $\Hom(\mathscr{O}_{Z_y},\mathscr{O}_{Z_{y'}}) =0$ and hence
 $\GHom(\mathscr{O}_{Z_y},\mathscr{O}_{Z_{y'}}) =0$. The $G$-clusters
 $Z_y$ and $Z_{y^\prime}$ are disjoint when their supporting
 $G$-orbits $\tau(y), \tau(y^\prime)\in \mathbb{A}_\kk^n$ satisfy
 $\tau(y)\neq \tau(y^\prime)$, so
 \begin{equation}
 \label{eqn:extvanish}
 (y,y^\prime)\in Y\times Y \text{ with }
 \tau(y)\neq\tau(y^\prime)\implies
 \GExt^i(\mathscr{O}_{Z_y},\mathscr{O}_{Z_{y^\prime}}) =0 \quad
 \text{for all }i\in \ZZ.
 \end{equation}
 Otherwise $y\neq y^\prime$ and $\tau(y)=\tau(y^\prime)$, so the pair
 $y, y^\prime \in Y$ satisfies $(y,y^\prime)\in Y\times_X
 Y\smallsetminus \Delta$, where $\Delta$ is the diagonal.  Our
 spurious assumption (A1) that Serre duality holds on $V=\mathbb{A}^n_\kk$,
 combined with the fact that $\omega_V$ is trivial as a
 $G$-equivariant sheaf implies that
 \[
 \GExt^n(\mathscr{O}_{Z_y},\mathscr{O}_{Z_{y^\prime}}) \cong
 \GHom(\mathscr{O}_{Z_{y^\prime}},\mathscr{O}_{Z_{y}})=0
 \]
 (omitting the dual). Since the global dimension of the skew group
 algebra is $n$, we obtain
 \begin{equation}
 \label{eqn:extvanish2}
 (y,y^\prime)\in Y\times_X Y \text{ with }y\neq y^\prime \implies
 \GExt^i(\mathscr{O}_{Z_y},\mathscr{O}_{Z_{y^\prime}}) =0 \quad
 \text{unless }1\leq i\leq n-1.
 \end{equation}
 We're now in a position to make the following important claim:
  \end{roadmap}

 \begin{claim}
 \label{cla:assumption}
   The assumption $\dim(Y\times_X Y)\leq n+1$ in Theorem~\ref{thm:BKR}
   is imposed to ensure that the groups
   $\GExt^i(\mathscr{O}_{Z_y},\mathscr{O}_{Z_{y^\prime}})$ vanish for
   the remaining values $1\leq i\leq n-1$.
 \end{claim}
 
 \subsection{On the Intersection theorem}
 To justify Claim~\ref{cla:assumption} and hence complete step three
 of our discussion, we first collect a few results from
 Bridgeland--Maciocia~\cite{BridgelandMaciocia}. The \emph{support} of
 an object $E\in D^b\big(\!\coh(Y)\big)$, denoted
 $\supp(E)$, is the closed subset of $Y$ obtained as the
 union of the supports of the cohomology sheaves
 $\mathcal{H}^i(E)$.  The proof of the
 following lemma is a simple spectral sequence
 argument~\cite[Lemma~5.3]{BridgelandMaciocia}.

 \begin{lemma}
 \label{lem:BM1}
 Let $E\in D^b\big(\!\coh(Y)\big)$ be an object and $y\in Y$
 a point. Then
 \[
 y\in \supp(E)\iff \exists \; i\in \ZZ \text{ such that }
 \Hom(E,\mathscr{O}_y[i])\neq 0.
 \]
 \end{lemma}
 
 The \emph{homological dimension} of a nonzero object $E\in
 D^b\big(\!\coh(Y)\big)$ is the smallest integer $\homdim(E):=d$ such
 that $E$ is quasi-isomorphic to a complex of locally-free sheaves of
 length $d$. An upper bound on the homological dimension was provided
 by Bridgeland--Maciocia~\cite{BridgelandMaciocia}:
 $\homdim(E)\leq d$ if and only if there exists $j\in \ZZ$
 such that for all $y\in Y$,
 \begin{equation}
 \label{eqn:homdim}
 \Hom(E, \mathscr{O}_{y}[i]) = 0
 \quad\text{unless }j\leq i\leq j+d;
 \end{equation}
 
 The next result, however, is the geometric interpretation of a deep
 result from commutative algebra called the Intersection Theorem.

 \begin{theorem}
 \label{thm:BM2}
 Let $E\in D^b\big(\!\coh(Y)\big)$ be a nonzero object. 
 \begin{enumerate}
 \item[\one] We have $\codim(\supp(E))\leq \homdim(E)$.
 \item[\two] Suppose that $H^0(E)\cong \mathscr{O}_y$ for some closed
   point $y\in Y$. Suppose further that for any point $y'\in Y$ and
   integer $i\in \ZZ$ we have
 \[
 \Hom(E, \mathscr{O}_{y'}[i]) = 0
 \quad\text{unless }y=y'\text{ and }0\leq i\leq n.
 \]
 Then $E\cong\mathscr{O}_y$, and $Y$ is smooth at $y$.
 \end{enumerate}
 \end{theorem}
 \begin{proof}
 See Bridgeland--Maciocia~\cite[Section 5]{BridgelandMaciocia}
 (compare also Bridgeland--Iyengar~\cite{BridgelandIyengar}).
 \end{proof}

 \subsection{Conclusion of the proof}
 We now return to the discussion of the third step. 

 \begin{roadmap2}
   To describe the logic in the rest of the proof, suppose that we can
   find an object $\mathcal{Q}\in D^b\big(\!\coh(Y\times Y)\big)$ such
   that
 \begin{equation}
 \label{eqn:Qobject}
 \Hom_{D^b(\coh(Y\times Y))}(\mathcal{Q}, \mathscr{O}_{(y,y^\prime)}[i]) =
 \GExt^i(\mathscr{O}_{Z_y},\mathscr{O}_{Z_{y^\prime}}).
 \end{equation}
 In light of the vanishing from \eqref{eqn:extvanish}, applying
 Lemma~\ref{lem:BM1} to the object $\mathcal{Q}\vert_{Y\times
   Y\smallsetminus \Delta}$ shows that the support of
 $\mathcal{Q}\vert_{Y\times Y\smallsetminus \Delta}$ is contained in
 $Y\times_X Y\smallsetminus \Delta$. The assumption on the dimension
 of the fibre product from Theorem~\ref{thm:BKR} now implies that
 $\dim\big{(}\supp(\mathcal{Q}\vert_{Y\times Y\smallsetminus
   \Delta})\big{)}\leq n+1$, i.e.,
 \[
 \codim\big{(}\supp(\mathcal{Q}\vert_{Y\times Y\smallsetminus
   \Delta})\big{)}\geq n-1.
 \]
 On the other hand, if the object $\mathcal{Q}\vert_{Y\times
   Y\smallsetminus \Delta}$ is nonzero, then substituting the
 vanishing from \eqref{eqn:extvanish2} into formula
 \eqref{eqn:homdim} gives
 \[
 \homdim(\mathcal{Q}\vert_{Y\times Y\smallsetminus \Delta})\leq n-2.
 \]
 This contradicts the inequality from Theorem~\ref{thm:BM2}\one\ 
 unless $\mathcal{Q}\vert_{Y\times Y\smallsetminus \Delta}\cong 0$.
 Now chase the logic backwards: since $\mathcal{Q}$ is supported on
 $\Delta$, Lemma~\ref{lem:BM1} together with \eqref{eqn:Qobject} imply
 that
 \[
  \GExt^i(\mathscr{O}_{Z_y},\mathscr{O}_{Z_{y'}}) =0 \quad
  \text{ for all } (y,y^\prime)\in Y\times_X Y\smallsetminus \Delta \text{ and } i\in \ZZ.
 \]
 This completes the required $\GExt$-vanishing in the
 case $y\neq y^\prime$ and hence completes step three.
 
 The fourth step, and by far the hardest, is to establish the equality
 from \eqref{eqn:BO} for the case $y=y'$. Our goal is to show that for
 all $y\in Y$ and $i\in \ZZ$, we have
 \[
 \GExt^i_{\mathscr{O}_{\mathbb{A}^n_\kk}}(\mathscr{O}_{Z_y},\mathscr{O}_{Z_{y}})
 = \Ext^i_{\mathscr{O}_Y}\!\big(\mathscr{O}_y,\mathscr{O}_y\big).
 \]
 Note that both vanish for $i\not\in [0,n]$. As for $0\leq i\leq n$,
 we add a little to \cite[\S6,
 \textrm{Step 5}]{BKR}.  For a point $y\in Y$, the adjunction
 $\Psi\dashv \Phi$ gives a canonical map
 $\Psi(\Phi(\mathscr{O}_y))\to\mathscr{O}_y$ and hence a triangle
 \begin{equation}
 \label{eqn:triangle}
 C\to \Psi(\Phi(\mathscr{O}_y))\to \mathscr{O}_y \to C[1]
 \end{equation}
 for some object $C\in D^b\big(\!\coh(Y)\big)$.  Using the adjunction
 $\Psi\dashv \Phi$ again, we claim that the long exact sequence
 obtained by applying $\Hom_{D^b(\coh(Y))}(-,\mathscr{O}_y)$ to the
 triangle \eqref{eqn:triangle} is
\[
 0 \longrightarrow
 \Hom(C,\mathscr{O}_y[-1]) \longrightarrow \Hom(\mathscr{O}_y,\mathscr{O}_y) \stackrel{\alpha}{\longrightarrow}\Hom\big(\Phi(\mathscr{O}_y),\Phi(\mathscr{O}_y)\big) \longrightarrow\]
 \[
 \Hom(C,\mathscr{O}_y)\longrightarrow
 \Hom\!\big(\mathscr{O}_y,\mathscr{O}_y[1]\big)\stackrel{\varepsilon}{\longrightarrow}
 \Hom\big(\Phi(\mathscr{O}_y),\Phi(\mathscr{O}_y)[1]\big)
 \longrightarrow \dots
 \]
 Indeed, both
 $\Hom(\mathscr{O}_y,\mathscr{O}_y[i])=\Ext^i(\mathscr{O}_y,\mathscr{O}_y)$
 and $\Hom\big(\Phi(\mathscr{O}_y),\Phi(\mathscr{O}_y)[i]\big)=
 \GExt^i(\mathscr{O}_{Z_y},\mathscr{O}_{Z_{y}})$ vanish for $i<0$.
 Moreover, $\GHom(\mathscr{O}_{Z_y},\mathscr{O}_{Z_y})\cong \kk$
 because $G$-invariance forces the generator $1\in
 H^0(\mathscr{O}_{Z_y})\cong \kk[G]$ to be sent to a scalar multiple
 of itself. This forces $\alpha$ to be injective, hence $
 \Hom(C,\mathscr{O}_y[-1])=0$.  Since $Y$ is the fine moduli space of
 $G$-clusters, Bridgeland~\cite[Lemma~5.3]{Bridgeland1} shows that the
 Kodaira--Spencer map $\varepsilon$ for the universal sheaf
 $\mathcal{O}_{\mathcal{Z}}$ over the product $Y\times
 \mathbb{A}^n_\kk$ is injective (see
 Huybrechts~\cite[Example~5.4(\textrm{vii})]{Huybrechts}). This gives
 $\Hom(C,\mathscr{O}_y)=0$.  Now that we have
 $\Hom\big(C,\mathscr{O}_y[i]\big) = 0$ for all $i\leq 0$, the
 spectral sequence argument from \cite[Example~2.2]{Bridgeland1} gives
 $H^0(C)=0$. Taking cohomology sheaves in the triangle
 \eqref{eqn:triangle} implies
 \[
 \mathcal{H}^0\big(\Psi(\Phi(\mathscr{O}_y))\big)=\mathscr{O}_y.
 \]
 Therefore $\Psi(\Phi(\mathscr{O}_y))$ satisfies the
 first condition of Theorem~\ref{thm:BM2}\two. Moreover, applying
 adjunction again, our third step shows that the groups
 \[
 \Hom\big(\Psi(\Phi(\mathscr{O}_y)),\mathscr{O}_{y'}[i]\big) =
 \Hom\big{(}\Phi(\mathscr{O}_y),\Phi(\mathscr{O}_{y^\prime})[i]\big{)}
 = \GExt^i(\mathscr{O}_{Z_y},\mathscr{O}_{Z_{y'}})
 \]
 vanish unless $y=y'$ and $0\leq i\leq n$. This vanishing is the
 second condition of Theorem~\ref{thm:BM2}\two, so we conclude
 $\Psi(\Phi(\mathscr{O}_y))\cong \mathscr{O}_y$ and hence that the
 object $C$ in the triangle \eqref{eqn:triangle} is zero.
 Substituting $\Hom\big(C,\mathscr{O}_y[i]\big) = 0$ in to the long
 exact sequence above gives the required isomorphisms
 \[
 \Hom\big(\Phi(\mathscr{O}_y),\Phi(\mathscr{O}_y)[i]\big)\cong
 \Hom\!\big(\mathscr{O}_y,\mathscr{O}_y[i]\big)
 \]
 for $0\leq i\leq n$. This clears up the remaining cases in
 \eqref{eqn:BO} and concludes step four. 
 
 Our fifth and final step is to construct an object $\mathcal{Q}$
 satisfying \eqref{eqn:Qobject}, after which we conclude that
 $\Phi^{\mathscr{O}_{\mathcal{Z}}}$ is fully faithful. In fact, since
 $\omega_{V}$ is trivial as a $G$-equivariant sheaf,
 $\Phi^{\mathscr{O}_{\mathcal{Z}}}$ commutes with the Serre functors
 and hence Theorem~\ref{thm:Bridgeland} implies that
 $\Phi^{\mathscr{O}_{\mathcal{Z}}}$ is a derived equivalence. We
 conclude, then, by constructing the crucial object $\mathcal{Q}$.  As
 with the functor arising in the proof of
 Theorem~\ref{thm:baerbondal}, we compose $\Phi$ with its left-adjoint
 $\Psi$ to obtain the composite
 \[
 \big{(}\Psi\circ \Phi\big{)}(-) = \Rderived (\pi_2)_*
 \big{(}\mathcal{Q}\ltensor \pi_1^*(-)\big{)},
 \]
 where $\pi_1,\pi_2\colon Y\times Y\to Y$ are the first and second
 projections and where $\mathcal{Q}\in D^b(\coh(Y\times Y))$
 is obtained by composition of correspondences. For the closed
 immersion $\iota_y\colon \{y\}\times Y\hookrightarrow Y\times Y$ we
 have $\Lderived \iota^*_y(\mathcal{Q})= \Psi(\Phi(\mathscr{O}_{y}))= \mathscr{O}_{y}$.
 Thus, for $i\in \ZZ$, we have
 \begin{align*}
\Hom_{D^b(\coh(Y\times Y))}\big(\mathcal{Q},
 \mathscr{O}_{(y,y^\prime)}[i]\big) & \cong
  \Hom_{D^b(\coh(Y\times Y))}\big{(}\mathcal{Q},
 (\iota_{y})_*(\mathscr{O}_{y^\prime})[i]\big{)} \\
 & \cong \Hom_{D^b(\coh(Y))}\big(\Lderived \iota^*_y(\mathcal{Q}),
 \mathscr{O}_{y^\prime}[i]\big)
 & \text{by adjunction} \\.
& =  \Hom_{D^b(\coh(Y))}\big(\Psi(\Phi(\mathscr{O}_{y})),\mathscr{O}_{y^\prime}[i]\big) &  \\
& \cong \Hom_{D^b(\modA)}(\Phi(\mathscr{O}_{y}),\Phi(\mathscr{O}_{y^\prime})[i]) & \text{by adjunction} \\
& \cong  \GExt^i(\mathscr{O}_{Z_y},\mathscr{O}_{Z_{y^\prime}}) & 
 \end{align*}
 as required. This completes the proof given our outrageous
 assumptions (A1) and (A2).
 \end{roadmap2}

 \begin{remark}
   We describe briefly how \cite{BKR} proves the
   result without (A1) and (A2).  
 \begin{enumerate}
 \item[(A1)] The trick is to consider a projective closure of
   $\mathbb{A}^n_\kk$ and hence a projective closure of the
   $G$-Hilbert scheme $Y$.  One then restricts the resulting derived
   equivalence to a functor on compactly supported objects. An
   alternative approach to the lack of projectivity was given more
   recently by Logvinenko~\cite{Logvinenko}.
 \item[(A2)] The statement of Theorem~\ref{thm:Bridgeland} uses the
   fact that the skyscraper sheaves $\{\mathscr{O}_y : y\in Y\}$ form
   a spanning class in $D^b\big(\!\coh(Y)\big)$, but this follows only
   from smoothness of $Y$. In fact, smoothness follows when we apply
   the Intersection Theorem in Step 4 above, after which one can apply
   Theorem~\ref{thm:Bridgeland} in the given form.  Finally, that
   $\tau$ is crepant is proven only after establishing the derived
   equivalence. A local argument shows that triviality of the Serre functor on
   $D^b\big(\!\modA\big)$ carries across the equivalence to gives $\omega_Y\cong \mathscr{O}_Y$.
 \end{enumerate}
 \end{remark}

 \begin{remark}
 Van den Bergh~\cite{VDB} observes that the
 object $\mathcal{Q}\in D^b\big(\!\coh(Y\times Y)\big)$ which plays the
 crucial role in the proof of Theorem~\ref{thm:BKR} satisfies
 \[
 \mathcal{Q} = \mathscr{W}^\vee\overset{\mathbf{L}}{\boxtimes}_A\mathscr{W}
 \]
 (compare Proposition~\ref{prop:resdiagonal}). This links the
 Fourier--Mukai construction of Bridgeland, King and Reid with the
 tilting approach to derived equivalences described in Section~\ref{sec:tilting}.
 \end{remark}

 \section{Derived McKay by variation of GIT quotient}
 \label{sec:VGIT}
 The derived McKay Correspondence was proposed by Reid~\cite{Reid3}
 (see Conjecture~\ref{conj:McKay}) for an arbitrary crepant resolution
 $Y$ of $\mathbb{A}^n_\kk/G$ for a finite subgroup $G\subset
 \SL(n,\kk)$, assuming one exists.  For $n=3$, Theorem~\ref{thm:BKR}
 establishes that $\ghilb(\mathbb{A}^3_\kk)$ always provides a crepant
 resolution, but what if $\mathbb{A}^3_\kk/G$ admits more than one
 crepant resolution? And what if $\ghilb(\mathbb{A}^n_\kk)$ is
 singular or the Hilbert--Chow morphism fails to be crepant for $n\geq
 4$? Can a crepant resolution exist and if so, does the McKay
 correspondence still hold? This section answers these questions for
 abelian group actions using toric techniques.

 \subsection{Variation of GIT quotient}
 \label{sec:VGIT1}
 Let $G\subset \SL(n,\kk)$ be a finite abelian subgroup with group
 algebra $\kk[G]$ and skew goup algebra $A=\kk[x_1,\dots,x_n]*G$. The
 vertices of the bound McKay quiver $(Q,\varrho)$ (see Example~\ref{ex:boundMcKay})
 correspond to elements of the character group
 $G^*=\Hom(G,\kk^\times)$. Recall from Example~\ref{exa:notnormal}
 that a \emph{$G$-constellation} is a left $A$-module that is
 isomorphic as a $\kk[G]$-module to $\kk[G]$. For any generic weight
 $\theta\in \Wt(Q)\otimes_\ZZ \QQ$, the GIT quotient
 \[
 \mathcal{M}_\theta:=\mathcal{M}_\theta(Q,\varrho) =
 \mathbb{V}(I_\varrho)\git_\theta T
 \]
 constructed in Section~\ref{sec:moduliboundquiver} is the fine moduli
space of $\theta$-stable $G$-constellations (it represents the functor
from Exercise~\ref{ex:boundfunctor}). Thus, $\mathcal{M}_\theta$
carries a tautological bundle $\mathscr{W}_\theta:= \bigoplus_{\rho\in
G^*} (\mathscr{W}_\theta)_\rho$ obtained as a direct sum of line
bundles with $(\mathscr{W}_\theta)_{\rho_0}$ trivial and, moreover, there is a
natural homomorphism of $\kk$-algebras $A\to
\End(\mathscr{W}_\theta)$. 

 As with the $G$-Hilbert scheme, we will restrict to the irreducible
 component of $\mathcal{M}_\theta$ containing the $G$-constellations
 arising as the structure sheaves of the free $G$-orbits in
 $\mathbb{A}^n_\kk$. This \emph{coherent component} was constructed
 explicitly by Craw--Maclagan--Thomas~\cite{CMT1, CMT2} in terms of a
 toric ideal $I_Q$ in the coordinate ring of $\mathbb{A}^{Q_1}_\kk$
 (compare Remark~\ref{rem:CMT}):

 \begin{theorem}
 \label{thm:CMT1}
   The binomial scheme $\mathbb{V}(I_\varrho)$ has a unique
   irreducible component $\mathbb{V}(I_Q)$ that lies in no coordinate
   hyperplane of $\mathbb{A}^{Q_1}_\kk$. Moreover, the affine toric
   variety $\mathbb{V}(I_Q)$ is such that
 \begin{enumerate}
 \item[\one] the semiprojective toric variety $Y_\theta:=
 \mathbb{V}(I_Q)\git_\theta T$ is the unique irreducible component of
 $\mathcal{M}_\theta$ containing the free $G$-orbits;
 \item[\two] if we write $\overline{\mathcal{M}_0}$ for the
 categorical quotient of the $0$-semistable points of
 $\mathbb{V}(I_\varrho)$ by the action of $T$, then there is a
 commutative diagram
 \[\xymatrix@C=1.2em{
\mathcal{M}_\theta \ar[rr]^-{\tau} && \overline{\mathcal{M}_0} \\
 Y_\theta \ar[u]\ar[rr]^-{\tau\vert_{Y_\theta}} &&
 X=\mathbb{A}^n_\kk/G\ar[u] \\ } \] where the vertical maps are closed
 immersions and the horizontal maps are projective
 morphisms obtained by variation of GIT quotient sending
 $\theta\rightsquigarrow 0$.
 \end{enumerate}
 \end{theorem}
 \begin{corollary}
 For any weight $\vartheta\in \Wt(Q)\otimes_\ZZ \QQ$ such that
 $\mathcal{M}_\vartheta\cong \ghilb(\mathbb{A}^n_\kk)$, the coherent
 component $Y_\vartheta$ is isomorphic to the irreducible scheme
 $\hilbg(\mathbb{A}^n_\kk)$ introduced by Nakamura.
 \end{corollary}

 \begin{remark}
 \label{rem:notnormal2}
 Recall from Example~\ref{exa:notnormal} that the $G$-Hilbert scheme
 can be nonnormal. In fact, it is the coherent component
 $Y_\vartheta=\hilbg(\mathbb{A}^n_\kk)$ that is not normal in that
 example. Thus, even if one hopes to work only with normal toric
 varities (defined by fans), this example forces one to drop the
 normality assumption.
 \end{remark}
 
 As we vary the stability parameter $\theta$, the geometry of
 $Y_\theta$ may change as $\theta$ varies between GIT chambers. Even
 if the variety does not change, the tautological bundle
 $\mathscr{W}_\theta$ on $\mathcal{M}_\theta$ and its restriction to
 $Y_\theta$ does vary between chambers.

 \begin{example}
 \label{exa:VGIT1/3(1,2)}
 For the action of type $\frac{1}{3}(1,2)$, the McKay quiver $Q$ is
 shown in Figure~\ref{exa:qos1/3(1,2)}. For the weight $\vartheta =
 (-2,1,1)\in \Wt(Q)$, Example~\ref{exa:GHilb1/3(1,2)} shows directly
 that $\mathcal{M}_\vartheta$ is isomorphic to the minimal resolution
 $Y$ of the $A_2$-singularity $\mathbb{A}^2_\kk/G$. The GIT chamber
 decomposition of $\Wt(Q)\otimes_\ZZ \QQ$ coincides with the Weyl
 chamber decomposition of type $A_2$ by Kronheimer~\cite{Kronheimer}
 (see also Cassens--Slodowy~\cite{CassensSlodowy}). Here we compute
 this directly.

 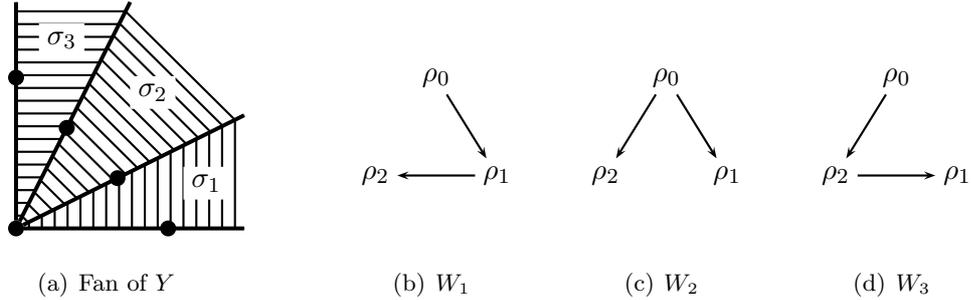
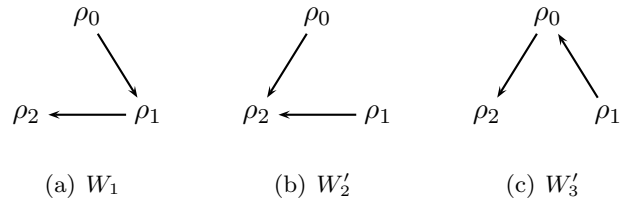
\begin{figure}[!ht]
    \centering \mbox{ \subfigure[Fan of $Y$]{ \psset{unit=1cm}
        \begin{pspicture}(-1.0,-0.3)(3.5,3.2)
        \pspolygon[linecolor=white,fillstyle=hlines*,
          hatchangle=0](0,0)(0,3)(1.5,3)
          \pspolygon[linecolor=white,fillstyle=hlines*,
          hatchangle=135](0,0)(1.5,3)(3,1.5)
          \pspolygon[linecolor=white,fillstyle=hlines*,
          hatchangle=90](0,0)(3,1.5)(3,0)
          \cnode*(0,0){3pt}{P1}
          \cnode*(0.6667,1.3333){3pt}{P2}
          \cnode*(2,0){3pt}{P3}
          \cnode*(1.3333,0.6667){3pt}{P4}
          \cnode*(0,2){3pt}{P5}
          \cnode*(1.5,3){0pt}{P7}
          \cnode*(3,1.5){0pt}{P8}
          \cnode*(0,3){0pt}{P9}
          \cnode*(3,0){0pt}{P10}
          \ncline[linewidth=1.5pt]{-}{P1}{P7}
          \ncline[linewidth=1.5pt]{-}{P1}{P8}
          \ncline[linewidth=1.5pt]{-}{P1}{P9}
          \ncline[linewidth=1.5pt]{-}{P1}{P10}
          \rput(0.6,2.5){\psframebox*{$\sigma_3$}}
          \rput(1.8,1.8){\psframebox*{$\sigma_2$}}
          \rput(2.5,0.6){\psframebox*{$\sigma_1$}}
         \end{pspicture}}
 \qquad
      \subfigure[$W_1$]{
        \psset{unit=1cm}
        \begin{pspicture}(0,0)(2,3.2)
        \rput(1,2.3){\rnode{A}{$\rho_0$}}
        \rput(0.2,1){\rnode{B}{$\rho_2$}}
        \rput(1.8,1){\rnode{C}{$\rho_1$}}
        \ncline[nodesep=3pt]{->}{A}{C}
        \ncline[nodesep=3pt]{->}{C}{B}
        \end{pspicture}}
 \qquad
      \subfigure[$W_2$]{
        \psset{unit=1cm}
        \begin{pspicture}(0,0)(2,3.2)
        \rput(1,2.3){\rnode{A}{$\rho_0$}}
        \rput(0.2,1){\rnode{B}{$\rho_2$}}
        \rput(1.8,1){\rnode{C}{$\rho_1$}}
        \ncline[nodesep=3pt]{->}{A}{B}
        \ncline[nodesep=3pt]{->}{A}{C}
        \end{pspicture}}
 \qquad
      \subfigure[$W_3$]{
        \psset{unit=1cm}
        \begin{pspicture}(0,0)(2,3.2)
        \rput(1,2.3){\rnode{A}{$\rho_0$}}
        \rput(0.2,1){\rnode{B}{$\rho_2$}}
        \rput(1.8,1){\rnode{C}{$\rho_1$}}
        \ncline[nodesep=3pt]{->}{A}{B}
        \ncline[nodesep=3pt]{->}{B}{C}
        \end{pspicture}}
    }
    \caption{$T_Y$-invariant $\vartheta$-stable quiver representations}
    \label{fig:quiverreps1/3(1,2)}
  \end{figure}

 The three $\vartheta$-stable representations of $(Q,\varrho)$ that
 correspond to the $T_Y$-invariant points of $Y$ are shown in
 Figure~\ref{fig:quiverreps1/3(1,2)}; here, the representation $W_i$
 corresponds to the origin in the toric chart $\Spec
 \big(\kk[\sigma_i^\vee\cap M]\big)$ in $Y$ for $i=1,2,3$. One can
 perform this calculation can either by computing the $T_Y$-invariant
 $G$-clusters in $\mathbb{A}^2_\kk$, or directly using
 Craw--Maclagan--Thomas~\cite[Theorem~1.3]{CMT1}. Notice that the
 representation $W_2$ is $\vartheta$-stable for all parameters in the
 cone
 \[
 C=\big\{(\theta_0,\theta_1,\theta_2)\in \QQ^3 : \theta_0
 +\theta_1 + \theta_2 = 0, \theta_1>0, \theta_2>0\big\};
 \]
 this is the cone from Example~\ref{ex:generic}. Now perform variation
 of GIT quotient, pushing $\theta$ through the wall $\theta_1=0$ to
 produce a weight $\vartheta'$ in an adjacent chamber
 \[
 C'=\big\{(\theta_0,\theta_1,\theta_2)\in \QQ^3 :
 \theta_0 +\theta_1 + \theta_2 = 0, \theta_1+\theta_2>0,
 \theta_1<0\big\}.
 \]
 The representation $W_1$ remains $\vartheta'$-stable since
 $\vartheta'_1+\vartheta'_2>0$. However, both $W_2$ and $W_3$ admit
 submodules supported on vertex $\rho_1$, so they become
 $\vartheta'$-unstable for $\vartheta'_1\leq 0$. In their place we
 obtain the new pair of $\vartheta'$-stable representations $W_2'$ and
 $W_3'$ shown in Figure~\ref{fig:quiverrepsprime1/3(1,2)}.
 \begin{figure}[!ht]
    \centering \mbox{\subfigure[$W_1$]{ \psset{unit=1cm}
        \begin{pspicture}(0,0.5)(2,2.5)
        \rput(1,2.3){\rnode{A}{$\rho_0$}}
        \rput(0.2,1){\rnode{B}{$\rho_2$}}
        \rput(1.8,1){\rnode{C}{$\rho_1$}}
        \ncline[nodesep=3pt]{->}{A}{C}
        \ncline[nodesep=3pt]{->}{C}{B}
        \end{pspicture}}
 \qquad
      \subfigure[$W'_2$]{
        \psset{unit=1cm}
        \begin{pspicture}(0,0.5)(2,2.5)
        \rput(1,2.3){\rnode{A}{$\rho_0$}}
        \rput(0.2,1){\rnode{B}{$\rho_2$}}
        \rput(1.8,1){\rnode{C}{$\rho_1$}}
        \ncline[nodesep=3pt]{->}{A}{B}
        \ncline[nodesep=3pt]{->}{C}{B}
        \end{pspicture}}
 \qquad
      \subfigure[$W'_3$]{
        \psset{unit=1cm}
        \begin{pspicture}(0,0.5)(2,2.5)
        \rput(1,2.3){\rnode{A}{$\rho_0$}}
        \rput(0.2,1){\rnode{B}{$\rho_2$}}
        \rput(1.8,1){\rnode{C}{$\rho_1$}}
        \ncline[nodesep=3pt]{->}{A}{B} 
        \ncline[nodesep=3pt]{->}{C}{A}
        \end{pspicture}}
    }
    \caption{$T_Y$-invariant $\vartheta'$-stable quiver representations}
 \label{fig:quiverrepsprime1/3(1,2)}
  \end{figure}
 If we now pass through the wall $\theta_1+\theta_2=0$, it's clear
that $W_3'$ remains $\theta$-stable as long as $\theta_2>0$, while
both $W_1$ and $W_2'$ becomes $\theta$-unstable for
$\theta_1+\theta_2\leq 0$.
 \end{example} 

 \begin{exercise}
   Continue the calculation from Example~\ref{exa:VGIT1/3(1,2)} to
   produce all six GIT chambers. The result coincides
   precisely\footnote{I cannot resist pointing out the the toric del
   Pezzo surface $\text{Bl}_{p,q,r}(\mathbb{P}^2)$ and the crepant
   resolution $\mathcal{M}_\vartheta$ of the $A_2$-singularity are
   linked by the Duality Theorem for affine toric quotients of
   Craw--Maclagan~\cite[Theorem~1.2]{CrawMaclagan}.} with the fan from
   Figure~\ref{fig:Fanosurfaces}(d). The symmetry in the
   $A_2$-decomposition is lost only because we project
   $\Wt(Q)\otimes_\ZZ\QQ\subset \QQ^3$ onto the plane $(\theta_0=0)$.
 \end{exercise}

 \subsection{Derived equivalences from universal bundles}
 The projective morphism constructed in Theorem~\ref{thm:CMT1}\two\ 
 can be used to construct a commutative diagram of the form
 \eqref{eqn:cdFM}, and hence by using the universal $\theta$-stable
 $G$-constellation $\mathscr{U}_\theta$ on the product
 \(Y_\theta\times \mathbb{A}^n_\kk\) we define a functor
 $\Phi^{\mathscr{U}_\theta}\colon
 D^b\big(\!\coh(Y_\theta)\big)\rightarrow D^b\big(\!\modA\big)$ via
 \begin{equation}
 \label{eqn:thetaFMformula}
 \Phi^{\mathscr{U}}(-) = {\Rderived}(\pi_V)_*\Big{(}\mathscr{U}_\theta\ltensor (\pi_{Y_\theta})^*(-\otimes\rho_0)\Big{)}.
 \end{equation}
 The results of Bridgeland, King and Reid~\cite{BKR} described in
 Section~\ref{sec:McKay} require only that the scheme $Y$ is the
 coherent component of a fine moduli space of $\theta$-stable
 $G$-constellations for some generic weight $\theta\in
 \Wt(Q)\otimes_\ZZ \QQ$, so the results extend from
 $\ghilb(\mathbb{A}^n_\kk)$ to $\mathcal{M}_\theta$. Stability
 preserves the vanishing of
 $\GHom(\mathscr{F}_y,\mathscr{F}_{y^\prime})$ for distinct points
 $y,y^\prime \in Y_\theta$, since $\theta$-stable $G$-constellations
 $\mathscr{F}_y$ are simple objects in the full category of
 $\theta$-semistable $A$-modules. Thus, we obtain the following (the
 abelian assumption is not necessary here, see
 \cite[Theorem~2.5]{CrawIshii}):
 
 \begin{theorem}
 \label{thm:CrawIshii1}
   Let $G\subset \SL(n,\kk)$ be a finite abelian subgroup and let
   $\theta\in\Wt(Q)\otimes_\ZZ \QQ$ be generic. If the coherent
   component $Y_\theta \subseteq \mathcal{M}_\theta$ satisfies
   $\dim\big{(}Y_\theta\times_X Y_\theta\big{)}\leq n+1$, then:
 \begin{enumerate}
 \item the morphism \(\tau\colon Y_\theta\rightarrow X\) is a crepant
 resolution; and
\item the functor $\Phi^{\mathscr{U}_{\theta}}\colon
  D^b\big(\!\coh(Y_\theta)\big)\rightarrow D^b(\modA)$ is an equivalence of
  derived categories.
 \end{enumerate}
 \end{theorem}

 \begin{corollary}
 The $\kk$-algebra $\End(\mathscr{W}_\theta)$ is isomorphic to the skew group algebra $A$. 
 \end{corollary}
 \begin{proof}
 To simplify notation, write $\mathscr{W} = \bigoplus_{\rho\in G^*}
 \mathscr{W}_\rho$ thereby omitting the dependence on $\theta$.  Then
 \begin{align*}
 A & \cong \Hom_A(A,A) & \\
 & \cong \GHom_{\mathscr{O}_{\mathbb{A}^n_\kk}}\Big(\bigoplus_{\rho\in G^*}
 \mathscr{O}_{\mathbb{A}^n_\kk}\otimes \rho,\bigoplus_{\rho\in G^*}
 \mathscr{O}_{\mathbb{A}^n_\kk}\otimes \rho\Big) & \text{by the equivalence from Lemma~\ref{lem:abelianequivs}}\\
 & \cong \Hom_{Y_\theta}\Big(\bigoplus_{\rho\in G^*} \mathscr{W}_{\rho^*}^{-1},\bigoplus_{\rho\in G^*} \mathscr{W}_{\rho^*}^{-1}\Big) & \text{by applying the equivalence }\Phi^{\mathscr{U}_{\theta}}
 \end{align*}
 which is isomorphic to $\End_{Y_\theta}(\mathscr{W})$ as claimed.
 \end{proof} 

  \subsection{The 3-fold case}
 In dimension 3, the $G$-Hilbert scheme is a
 projective crepant resolution of $\mathbb{A}^3_\kk/G$, but in general
 there may be more than one. We now describe how to establish the
 derived equivalence from Conjecture~\ref{conj:McKay} for the other
 projective crepant resolutions when the finite subgroup $G\subset \SL(3,\kk)$ is abelian.
 
 First, we describe how to construct crepant resolutions of
 $\mathbb{A}^3_\kk/G$ using toric geometry. The abelian quotient
 singularity $\mathbb{A}^3_\kk/G$ is the normal affine toric variety
 $\Spec\!\big(\kk[\sigma^\vee\cap M]\big)$ from
 Example~\ref{exa:toricabelian}. The \emph{junior simplex} is the
 lattice triangle in $N\otimes_\ZZ \QQ$ obtained as the intersection
 of the positive octant with the hyperplane normal to the vector
 $(1,1,1)$. Every basic triangulation of the junior simplex determines
 uniquely a fan $\Sigma$ supported on the positive orthant such that
 every three-dimensional cone $\sigma\in \Sigma$ is spanned by a
 lattice basis of $N$.  If the triangulation is \emph{coherent} (see
 Sturmfels~\cite{Sturmfels}) then the smooth toric variety $Y$ defined
 by the fan $\Sigma$ is projective over $\mathbb{A}^3_\kk/G$.
 Moreover, the resolution $Y\to \mathbb{A}^3_\kk/G$ is crepant and,
 conversely, every projective crepant resolution arises in this way.

 \begin{example}
 \label{exa:Z2xZ2toric}
   Consider the subgroup $G\cong \ZZ/2\times\ZZ/2$ of $\SL(3,\kk)$,
   where the action is of type
   $\frac{1}{2}(1,1,0)\oplus\frac{1}{2}(1,0,1)$. Following the toric
   recipe described in Example~\ref{exa:toricabelian}, the junior
   simplex is shown in Figure~\ref{fig:Z2xZ2}(a). In this case there
   are four basic triangulations, all of which are coherent. The
   junior simplex for two such fans are shown in
   Figure~\ref{fig:Z2xZ2}(b) and (c).
 \begin{figure}[!ht]
    \centering \mbox{ \subfigure[The junior simplex]{ \psset{unit=1cm}
        \begin{pspicture}(-1.0,-0.5)(3,2.2)
          \cnode*[fillcolor=black](0,0){3pt}{P1}
          \cnode*[fillcolor=black](1,0){3pt}{P2}
          \cnode*[fillcolor=black](2,0){3pt}{P3}
          \cnode*[fillcolor=black](0,1){3pt}{P4}
          \cnode*[fillcolor=black](1,1){3pt}{P5}
          \cnode*[fillcolor=black](0,2){3pt}{P7}
          \ncline[linewidth=2pt]{-}{P1}{P3}
          \ncline[linewidth=2pt]{-}{P3}{P7}
          \ncline[linewidth=2pt]{-}{P1}{P7}
        \end{pspicture}}
      \qquad 
       \subfigure[Fan defining $\ghilb(\mathbb{A}^3_\kk)$]{
        \psset{unit=1cm}
        \begin{pspicture}(-1.0,-0.5)(3,2.2)
          \cnode*[fillcolor=black](0,0){3pt}{P1}
          \cnode*[fillcolor=black](1,0){3pt}{P2}
          \cnode*[fillcolor=black](2,0){3pt}{P3}
          \cnode*[fillcolor=black](0,1){3pt}{P4}
          \cnode*[fillcolor=black](1,1){3pt}{P5}
          \cnode*[fillcolor=black](0,2){3pt}{P7}
          \ncline[linewidth=2pt]{-}{P1}{P3}
          \ncline[linewidth=2pt]{-}{P3}{P7}
          \ncline[linewidth=2pt]{-}{P1}{P7}
         \ncline[linewidth=2pt]{-}{P2}{P4}
         \ncline[linewidth=2pt]{-}{P2}{P5}
         \ncline[linewidth=2pt]{-}{P4}{P5}
        \end{pspicture}}
      \qquad 
       \subfigure[An alternative resolution]{
        \psset{unit=1cm}
        \begin{pspicture}(-1.0,-0.5)(3,2.2)
          \cnode*[fillcolor=black](0,0){3pt}{P1}
          \cnode*[fillcolor=black](1,0){3pt}{P2}
          \cnode*[fillcolor=black](2,0){3pt}{P3}
          \cnode*[fillcolor=black](0,1){3pt}{P4}
          \cnode*[fillcolor=black](1,1){3pt}{P5}
          \cnode*[fillcolor=black](0,2){3pt}{P7}
          \ncline[linewidth=2pt]{-}{P1}{P3}
          \ncline[linewidth=2pt]{-}{P3}{P7}
          \ncline[linewidth=2pt]{-}{P1}{P7}
         \ncline[linewidth=2pt]{-}{P2}{P7}
         \ncline[linewidth=2pt]{-}{P2}{P4}
         \ncline[linewidth=2pt]{-}{P2}{P5}
        \end{pspicture}}}
    \caption{Resolving the quotient by $\ZZ/2\times\ZZ/2$ in $\SL(3,\kk)$} \label{fig:Z2xZ2}
  \end{figure}
 \end{example}
 
 \begin{remark}
   Using the description of the toric fan of
   $\ghilb(\mathbb{A}^3_\kk)$ from Craw--Reid~\cite{CrawReid}, one can
   show that $\mathbb{A}^3_\kk/G$ admits a unique crepant resolution
   if and only if the action of $G$ is either \one\ of type
   $\frac{1}{r}(1,r-1,0)$ for some $r\geq 1$; \two\ of type
   $\frac{1}{r}(1,1,r-2)$ for some $r\geq 1$; or \three\ of type
   $\frac{1}{7}(1,2,4)$. In particular, quotients $\mathbb{A}^3_\kk/G$
   that admit a unique crepant resolution are very rare.
 \end{remark}
 
 The following result of Craw--Ishii~\cite{CrawIshii} addresses the
 question of the McKay correspondence for quotients
 $\mathbb{A}^3_\kk/G$ that admit more than one projective crepant
 resolution.

  \begin{theorem}
 \label{thm:CrawIshii2}
 Let \(G\subset \SL(3,\kk)\) be a finite abelian subgroup.  For every
 projective crepant resolution \(\tau\colon Y\to \mathbb{A}^3_\kk/G\),
 there exists a generic weight $\theta\in \Wt(Q)\otimes_\ZZ \QQ$ such
 that \(Y \cong \mathcal{M}_\theta\).
 \end{theorem}
 
 Together with Theorem~\ref{thm:CrawIshii1}, this implies the derived
 McKay correspondence Conjecture~\ref{conj:McKay} for the class of
 projective crepant resolutions of $\mathbb{A}^3_\kk/G$.  Moreover, by
 considering the compositions
 \[
 (\Phi^{\mathscr{U}_{\theta^\prime}})^{-1}\circ
 \Phi^{\mathscr{U}_\theta}\colon
 D^b\big(\!\coh(Y_\theta)\big)\rightarrow
 D^b\big(\!\coh(Y_{\theta^\prime})\big)
 \]
 arising from any pair of generic weights $\theta, \theta^\prime\in
 \Wt(Q)\otimes_\ZZ \QQ$, we see that any two projective crepant
 resolutions of $\mathbb{A}^3_\kk/G$ have equivalent derived
 categories. 
 
 \begin{remark}
   For any finite subgroup $G\subset \SL(3,\kk)$,
   Bridgeland~\cite{Bridgeland2} constructs derived equivalences
   between the bounded derived categories of coherent sheaves on any
   pair of projective crepant resolutions of $\mathbb{A}^3_\kk/G$.
   This establishes Conjecture~\ref{conj:McKay} in dimension \(n = 3\)
   without having to restrict to abelian subgroups. More recently,
   Logvinenko~\cite{Logvinenko} established
   Conjecture~\ref{conj:McKay} for a specific nonprojective crepant resolution.
  \end{remark}
  
 \subsection{Sketch of the proof}
 The $G$-Hilbert scheme provides one crepant resolution of
 $\mathbb{A}^3_\kk/G$, and every other crepant resolution is obtained
 from $\ghilb(\mathbb{A}^3_\kk)$ by a finite sequence of flops.  Since
 $\ghilb(\mathbb{A}^3_\kk)=\mathcal{M}_\vartheta$ for the weight
 $\vartheta$ from Remark~\ref{rem:G-Hilbert}, it is enough to prove
 that if $\mathcal{M}_\theta$ is a crepant resolution for generic
 $\theta\in \Wt(Q)\otimes_\ZZ \QQ$, and if $Y^\prime$ is another
 crepant resolution that is obtained from $\mathcal{M}_\theta$ by the
 flop of a toric curve, then $Y^\prime=\mathcal{M}_{\theta^\prime}$
 for some other generic $\theta^\prime$. Thus, we choose one such flop
 $\mathcal{M}_\theta \dasharrow Y^\prime$ and must show that we can
 induce the flop by variation of GIT quotient in the weight space
 $\Wt(Q)\otimes_\ZZ \QQ$, thereby obtaining
 $Y^\prime=\mathcal{M}_{\theta^\prime}$.
 
 Before explaining the logic of the proof we first describe the link
 between the space of weights and birational geometry of the moduli
 spaces.  For generic $\theta\in \Wt(Q)\otimes_\ZZ \QQ$, write $C$ for
 the open GIT chamber containing $\theta$ (see
 Section~\ref{sec:projectiveGIT}).  For simplicity, write
 $\mathscr{W}=\bigoplus_{\rho\in G^*} \mathscr{W}_\rho$ for the
 tautological bundle on the fine moduli space $Y=\mathcal{M}_\theta$.  It
 follows tautologically from the GIT construction of
 $\mathcal{M}_\theta$ that the map
 \begin{equation}
 \label{eqn:Ltheta}
 L_\theta\colon \Wt(Q)\otimes_\ZZ \QQ \longrightarrow
 \Pic(Y)\otimes_\ZZ \QQ
 \end{equation}
 sending $\eta = (\eta_\rho)_{\rho\in G^*}\in \Wt(Q)\otimes_\ZZ \QQ$
 to the (fractional) line bundle $\bigotimes_{\rho\in G^*}
 \mathscr{W}_\rho^{\eta_\rho}$ sends every weight in $C$ to an ample
 bundle. Thus, there are two options for a weight $\theta_0$ in
 a codimension-one face $W$ of the boundary of the closure
 $\overline{C}$: the line bundle $L:= L_\theta(\theta_0)$ either lies
 in the interior of the ample cone, in which case it is ample; or it
 lies on the boundary of the closure, in which case it is nef but not
 ample. If we could move $L$ freely around the ample cone of $Y$, then
 we would simply procede towards the codimension-one face of the nef
 cone that defines our flop $\mathcal{M}_\theta =Y\dasharrow Y^\prime$
 and pass $L$ through the wall, thereby inducing the flop.
 
 However, while we may move freely in $\Wt(Q)\otimes_\ZZ \QQ$, this
 does not mean we may move the corresponding line bundle freely in
 $\Pic(Y)$. To see this, suppose that passing $\theta$ through the
 wall $W$ defines an isomorphism rather than the desired flop (so
 $L_\theta(\theta_0)$ was in the interior of the ample cone).  We have
 $\mathcal{M}_\theta=Y\cong \mathcal{M}_{\theta'}$ for $\theta'$ in the
 adjacent chamber $C'$, but the new tautological bundle is
 $\mathscr{W}'=\bigoplus_{\rho\in G^*}
 (\mathscr{W}^\prime)_\rho$. Thus, the line bundle associated to
 $\theta'\in C'$ is determined by
 \[
 L_{\theta'}\colon \Wt(Q)\otimes_\ZZ \QQ
 \longrightarrow \Pic(Y)\otimes_\ZZ
 \QQ,
 \] where $L_{\theta'}(\eta) = \bigotimes_{\rho\in G^*}
 (\mathscr{W}^\prime)_\rho^{\eta_\rho}$.  The maps $L_\theta$ and
 $L_{\theta'}$ are $\QQ$-linear, but together they form only a
 piecewise $\QQ$-linear map on the union $C\cup C'$. Thus, while we
 may move $L_\theta(\theta)$ to the boundary of $L_\theta(C)$ and
 beyond, we do not know whether the line bundles
 $L_{\theta'}(\theta')$ for $\theta'\in C'$ lie any closer to the face
 of the nef cone of $Y$ that gives our chosen flop. As
 Example~\ref{exa:VGIT1/3(1,2)} shows, crossing walls in
 $\Wt(Q)\otimes_\ZZ \QQ$ may even define a reflection in
 $\Pic(Y)$. The majority of Craw--Ishii~\cite{CrawIshii} is devoted to
 a careful analysis of how wall crossings affect $\mathcal{M}_\theta$
 and $\mathscr{W}_\theta$.

 \begin{exercise}
 For the action of type $\frac{1}{2}(1,1,0)\oplus\frac{1}{2}(1,0,1)$
   from Example~\ref{exa:Z2xZ2toric}. Compute explicitly the quiver
   representations that define the torus-invariant points of the
   $G$-Hilbert scheme, and hence show that $\ghilb(\mathbb{A}^3_\kk)
   =\mathcal{M}_\vartheta$ if and only if $\vartheta$ lies in the cone 
 \[
 C=\big\{(\theta_0,\theta_1,\theta_2, \theta_3)\in \QQ^3 : \theta_0
 +\theta_1 + \theta_2 +\theta_3= 0, \theta_1>0, \theta_2>0,
 \theta_3>0\big\};
 \]
  (compare Example~\ref{exa:VGIT1/3(1,2)}). Show that all three flops
  of $\ghilb(\mathbb{A}^3_\kk)$ can be produced by crossing a single
  wall of this chamber.
 \end{exercise}

 We now sketch the proof the Theorem~\ref{thm:CrawIshii2}. The key to
 understanding the wall crossings is to interpret geometrically the
 weight space $\Wt(Q)\otimes_\ZZ \QQ$ using the Fourier--Mukai
 transform. For a generic weight $\theta\in C$ with
 $Y=\mathcal{M}_\theta$, restricting $\Phi^{\mathscr{U}_\theta}$ to
 the full subcategory $D^b_0\big(\!\coh(Y)\big)$ consisting of objects
 supported on the subscheme $\tau^{-1}(\pi_V(0))$ induces a
 $\ZZ$-linear isomorphism
 \[
 \varphi_C \colon K_0(Y)\longrightarrow
 K_0(\modA)\cong\textstyle{ \bigoplus_{\rho\in G^*} \ZZ \rho}
 \]
 between the Grothendieck groups of coherent sheaves supported on
 $\tau^{-1}(\pi_V(0))$ and finitely generated left nilpotent
 $A$-modules. There is a natural isomorphism 
 \[
 \Wt(Q)\otimes_\ZZ \QQ \cong \big\{\theta\in \Hom_\ZZ(K_0(\modA),\QQ)
: \theta(\kk[G])=0\big\}, 
 \]
 and hence for any compactly supported sheaf $\mathscr{F}$ on
 $Y$ we may compute $\theta(\varphi_C(\mathscr{F}))\in
 \QQ$. The geometric interpretation of the walls of any given chamber
 can be described thus:

 \begin{proposition}
 \label{prop:chamberstrong}
 Let \(C \subset\Wt(Q)\otimes_\ZZ \QQ \) be a GIT chamber.  Then \(\theta\in C\) if
 and only if
 \begin{enumerate}
 \item[\one] for every exceptional (e.g., flopping) curve \(\ell\) in
 $Y$ we have \(\theta(\varphi_{C}(\mathscr{O}_{\ell})) > 0\); and
 \item[\two] for every compact reduced divisor \(D\) in $Y$ and
 \(\rho\in G^*\) we have
 \[
 \theta(\varphi_{C}(\mathscr{R}_{\rho}^{-1}\otimes \omega_{D})) <
 0\quad\text{and}\quad\theta(\varphi_{C}(\mathscr{R}_{\rho}^{-1}\vert_D))
 > 0.
 \]
 \end{enumerate}
 \end{proposition}
 
 The inequalities listed in \one, i.e., those of the form
 \(\theta(\varphi_{C}(\mathscr{O}_{\ell})) > 0\) for curves $\ell$ in
 $Y$, are good in the sense that they lift via the map $L_\theta$ from
 \eqref{eqn:Ltheta} inequalities defining the walls of the ample cone
 of $Y$. In particular, when we pass through these walls then we
 induce the corresponding birational change in $Y$. However, those
 listed in \two, dubbed \emph{walls of Type 0} in \cite{CrawIshii},
 are a priori problematic since they are not present in $\Pic(Y)$ and
 hence, as described above, we may lose track completely of the
 polarising line bundle as $\theta$ passes through such a wall. It is
 important therefore to understand what happens to the variety
 $\mathcal{M}_\theta$ and its tautological bundle $\mathscr{W}_\theta$ as
 $\theta$ passes through a wall of type 0.
  
 To describe the changes, we recall that an object \(E \in
 D^b\big(\!\coh(Y)\big)\) that is supported on the subscheme
 $\tau^{-1}(\pi_V(0))$ is \emph{spherical} if
 \[
 \Hom_{D_0^b(\coh(Y))}(E,E[j]) = \left\{\begin{array}{cc} \kk & \text{
 if } j = 0,3; \\ 0 & \text{otherwise} \end{array}\right.
 \]
 (see Huybrechts~\cite[Section~8.1]{Huybrechts}; we require no
 additional condition since $\omega_Y\cong \mathscr{O}_Y$). The
 \emph{twist} along a spherical object \(E\) is defined via the
 distinguished triangle
 \[
 \Rderived\Hom_{\mathscr{O}_{Y}}(E, F) \ltensor_{\CC} E
 \stackrel{ev}{\longrightarrow} F \longrightarrow T_{E}(F)
 \]
 for any object \(F \in D^b\big(\!\coh(Y)\big)\), where ev is the
 evaluation morphism.  The corresponding \emph{spherical twist
 functor} $T_{E}\colon D^b\big(\!\coh(Y)\big)\rightarrow
 D^b\big(\!\coh(Y)\big)$ is an autoequivalence.

 \begin{lemma}
 Let $C, C'$ be adjacent chambers in $\Wt(Q)\otimes_\ZZ \QQ$ separated
 by a wall of type 0, and choose $\theta\in C$ and $\theta'\in
 C'$. Then $\mathcal{M}_{\theta}$ is isomorphic to
 $\mathcal{M}_{\theta^\prime}$, and the composition
 \[
 (\Phi^{\mathscr{U}_{\theta^\prime}})^{-1}\circ
 \Phi^{\mathscr{U}_\theta}\colon
 D^b\big(\!\coh(\mathcal{M}_{\theta})\big)\rightarrow
 D^b\big(\!\coh(\mathcal{M}_{\theta'})\big)
 \]
 is a spherical twist functor (up to tensoring by a line bundle).
 Moreover, the changes in the tautological bundles
 $\mathscr{W}_\theta$ can be tracked across the equivalence.
 \end{lemma}
 
 We now complete the description of the proof of
 Theorem~\ref{thm:CrawIshii2}. For $Y=\mathcal{M}_\theta$, let \(C\)
 denote the chamber containing $\theta$. If $C$ has a wall that
 defines the desired flop according to
 Proposition~\ref{prop:chamberstrong}, then we pass through the wall,
 inducing the flop as required. Otherwise, move $\theta$ into an
 adjacent chamber \(C'\) separated from \(C\) by a wall of type 0.  A
 property of twists reveals that, roughly speaking, the images
 \(L_\theta(C)\) and \(L_{\theta'}(C')\) for $\theta'\in C'$ are
 adjacent cones in the Picard group.  We proceed in this way towards
 the desired wall of the ample cone, passing through at most finitely
 many such type 0 walls.  By crossing the final wall we induce the
 desired flop \(\mathcal{M}_{\theta}=Y \dashrightarrow
 Y'=\mathcal{M}_{\theta'}\).

 \subsection{McKay in higher dimensions}
 \label{sec:higherdims}
 It is natural to ask whether new examples of the McKay correspondence
 can be constructed using this moduli space approach in dimension
 \(n\geq 4\), when the singularity \(X = \mathbb{A}^n_\kk/G\) may
 admit crepant resolutions even though \(\ghilb(\mathbb{A}^n_\kk)\) is
 singular or discrepant. This final section studies one such example.
 
 Consider the quotient of \(\mathbb{A}^{4}_\kk\) by the action of type
 $\frac{1}{2}(1,1,1,0)\oplus\frac{1}{2}(1,1,0,1)\oplus\frac{1}{2}(1,0,1,1)$,
 that is, the action by the maximal diagonal subgroup \(G =
 (\ZZ/2)^{\oplus 3}\subset \SL(4,\kk)\) of exponent two. The
 resolution \(Y = \ghilb(\mathbb{A}^4_\kk)\rightarrow X\) has one
 exceptional divisor \(E\cong
 \mathbb{P}^1\times\mathbb{P}^1\times\mathbb{P}^1\) of discrepancy one
 that can be blown down to \(\mathbb{P}^1\times \mathbb{P}^1\) in
 three different ways, giving rise to crepant resolutions
 \(Y_i\rightarrow X\) with exceptional loci \(E_i \cong
 \mathbb{P}^1\times \mathbb{P}^1\) for \(i= 1,2,3\).  All four
 resolutions of \(X\) are toric morphisms, and the 3-dimensional
 cross-sections of the 4-dimensional fans defining these resolutions
 are shown in Figure~\ref{fig:toricZ2^3}. (In fact, the junior simplex
 is a tetrahedron, and the octahedra below are obtained only after
 chopping off all four corners; see Chiang--Roan~\cite{ChiangRoan}):
\begin{figure}[ht]
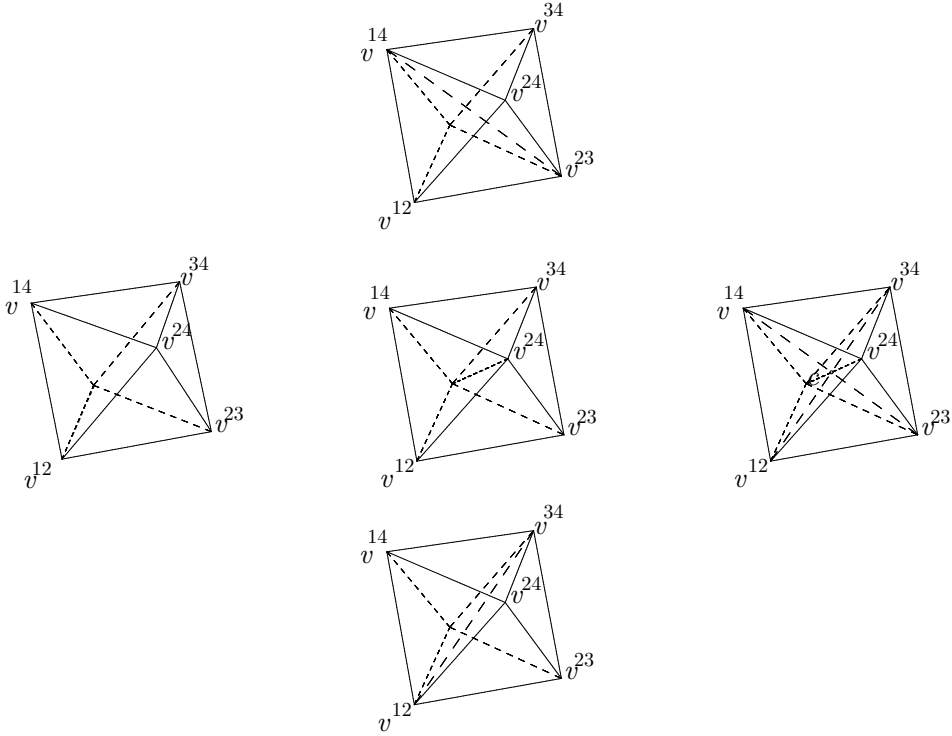

\label{fig:toricZ2^3}
$$\mbox{\psfig{figure=crep2.ps}}$$
$$\mbox{\psfig{figure=octa.ps}}\hskip 45pt \mbox{\psfig{figure=crep3.ps}}\hskip 45pt \mbox{\psfig{figure=hilbs.ps}}$$
$$\mbox{\psfig{figure=crep1.ps}}$$
 \caption{Toric picture as drawn by
   Chiang--Roan~\cite[Figure~3]{ChiangRoan}}
\end{figure}
We now show that each crepant resolution \(Y_i\) is a fine moduli
space \(\mathcal{M}_{\theta}\) of \(\theta\)-stable representations of
the bound McKay quiver $(Q,\varrho)$ for some generic \(\theta\in
\Wt(Q)\otimes_\ZZ \QQ\). Note that these are not the only ones, there
are another 189 distinct crepant resolutions (!).
 
 The chamber containing the weights defining \(\mathcal{M}_{\vartheta} =
 \ghilb(\mathbb{A}^4_\kk)\) is 
 \[
 C_0 := \big\{\theta =(\theta_\rho)_{\rho\in G^*}\in \Wt(Q)\otimes_\ZZ \QQ : \theta_\rho > 0 \text{ for
   all } \rho \neq \rho_{0} \big\}.
 \]
 Write \(\kk[x,y,z,w]\) for the coordinate ring of $\mathbb{A}^4_\kk$.
 The variables $x,y,z,w$ lie in distinct eigenspaces of the
 \(G\)-action that we denote \(\rho_{100}, \rho_{010}, \rho_{001},
 \rho_{111}\in G^*\cong(\ZZ/2)^{\oplus 3}\).  To simply the notation,
 we write \(\rho_0=\rho_{000}, \rho_1:= \rho_{011}, \rho_2:=
 \rho_{101}, \rho_3:= \rho_{110}\) for the complementary set of
 nontrivial irreducible representations. For \(i = 1,2,3\), let $W_i
 := \big\{\theta \in \overline{C_0} : \theta_{\rho_i} = 0\big\}$
 denote the codimension-one face of the closure of the chamber $C_0$
 that is contained in the indicated coordinate hyperplane.

 \begin{lemma}
   For $1\leq i\leq 3$, let \(C_i\) denote the unique GIT chamber
   lying adjacent to \(C_0\) that satisfies \(W_i = \overline{C_i}\cap
   \overline{C_0}\). Then \(Y_i = \mathcal{M}_\theta\) for all
   $\theta\in C_i$.
 \end{lemma}
 \begin{proof}
   As in Example~\ref{exa:VGIT1/3(1,2)}, we follow the fate of the
   $\vartheta$-stable \(G\)-constellations defined by the twelve
   torus-invariant points of \(Y\) as \(\vartheta\) passes through the
   wall \(W_i\). There are three cases.  
   
   First, consider the four $G$-clusters defining the torus-invariant
   points corresponding to the four simplices that were omitted in the
   right-most drawing in Figure~\ref{fig:toricZ2^3}. One such
   $G$-cluster $W_1$ is shown in Figure~\ref{fig:quiverrepsZ/2^3}(a),
   while the other three are obtained by cyclically permutting
   $x,y,z,w$. These remain \(\theta\)-stable with respect to
   parameters \(\theta\in C_i\).
 \begin{figure}[!ht] \centering
   \mbox{\subfigure[$W_1$]{ \psset{unit=1cm}
        \begin{pspicture}(0,0.5)(2.5,2.5)
        \rput(0,1){\rnode{A}{$\rho_{0}$}}
        \rput(1.6,1){\rnode{B}{$\rho_{100}$}}
        \rput(0,2.6){\rnode{C}{$\rho_{001}$}}
        \rput(1.6,2.6){\rnode{D}{$\rho_{2}$}}
        \rput(0.8,1.8){\rnode{A'}{$\rho_{010}$}}
        \rput(2.4,1.8){\rnode{B'}{$\rho_{3}$}}
        \rput(0.8,3.4){\rnode{C'}{$\rho_{1}$}}
        \rput(2.4,3.4){\rnode{D'}{$\rho_{111}$}}
         \ncline[nodesep=3pt]{->}{A}{B}
        \ncline[nodesep=3pt]{->}{A}{C}
        \ncline[nodesep=3pt]{->}{B}{D}
        \ncline[nodesep=3pt]{->}{C}{D}
        \ncline[nodesep=3pt]{->}{A'}{B'}
        \ncline[nodesep=3pt]{->}{A'}{C'}
        \ncline[nodesep=3pt]{->}{B'}{D'}
        \ncline[nodesep=3pt]{->}{C'}{D'} 
        \ncline[nodesep=3pt]{->}{A}{A'}
        \ncline[nodesep=3pt]{->}{B}{B'}
        \ncline[nodesep=3pt]{->}{C}{C'}
        \ncline[nodesep=3pt]{->}{D}{D'} 
       \end{pspicture}}
 \qquad
      \subfigure[$W_2$]{
        \psset{unit=1cm}
        \begin{pspicture}(-0.5,0.2)(2.5,2.5)
        \rput(0,1){\rnode{A}{$\rho_{0}$}}
        \rput(1.6,1){\rnode{B}{$\rho_{100}$}}
        \rput(0,2.6){\rnode{C}{$\rho_{001}$}}
        \rput(1.6,2.6){\rnode{D}{$\rho_{2}$}}
        \rput(0.8,1.8){\rnode{A'}{$\rho_{010}$}}
        \rput(2.4,1.8){\rnode{B'}{$\rho_{3}$}}
        \rput(0.8,3.4){\rnode{C'}{$\rho_{1}$}}
        \rput(-1,0.4){\rnode{D'}{$\rho_{111}$}}
         \ncline[nodesep=3pt]{->}{A}{B}
        \ncline[nodesep=3pt]{->}{A}{C}
        \ncline[nodesep=3pt]{->}{B}{D}
        \ncline[nodesep=3pt]{->}{C}{D}
        \ncline[nodesep=3pt]{->}{A'}{B'}
        \ncline[nodesep=3pt]{->}{A'}{C'}
        \ncline[nodesep=3pt]{->}{A}{A'}
        \ncline[nodesep=3pt]{->}{B}{B'}
        \ncline[nodesep=3pt]{->}{C}{C'}
        \ncline[nodesep=3pt,linecolor=gray]{->}{A}{D'} 
       \end{pspicture}}
 \qquad 
      \subfigure[$W_3$]{
        \psset{unit=1cm}
        \begin{pspicture}(-0.7,0)(2.5,2.5)
        \rput(0,1){\rnode{A}{$\rho_{0}$}}
        \rput(1.6,1){\rnode{B}{$\rho_{100}$}}
        \rput(0,2.6){\rnode{C}{$\rho_{001}$}}
        \rput(1.6,2.6){\rnode{D}{$\rho_{2}$}}
        \rput(0.8,1.8){\rnode{A'}{$\rho_{010}$}}
        \rput(2.4,1.8){\rnode{B'}{$\rho_{3}$}}
        \rput(-1,2){\rnode{C'}{$\rho_{1}$}}
        \rput(-1,0.4){\rnode{D'}{$\rho_{111}$}}
         \ncline[nodesep=3pt]{->}{A}{B}
        \ncline[nodesep=3pt]{->}{A}{C}
        \ncline[nodesep=3pt]{->}{B}{D}
        \ncline[nodesep=3pt]{->}{C}{D}
        \ncline[nodesep=3pt]{->}{A'}{B'}
        \ncline[nodesep=3pt]{->}{D'}{C'}
        \ncline[nodesep=3pt]{->}{A}{A'}
        \ncline[nodesep=3pt]{->}{B}{B'}
        \ncline[nodesep=3pt,linecolor=gray]{->}{C}{C'}
        \ncline[nodesep=3pt,linecolor=gray]{->}{A}{D'} 
       \end{pspicture}}
 \qquad 
      \subfigure[$W'_2$]{
        \psset{unit=1cm}
        \begin{pspicture}(-0.7,-0.7)(2.5,2.8)
        \rput(0,1){\rnode{A}{$\rho_{0}$}}
        \rput(1.6,1){\rnode{B}{$\rho_{100}$}}
        \rput(0,2.6){\rnode{C}{$\rho_{001}$}}
        \rput(1.6,2.6){\rnode{D}{$\rho_{2}$}}
        \rput(0.8,1.8){\rnode{A'}{$\rho_{010}$}}
        \rput(2.4,1.8){\rnode{B'}{$\rho_{3}$}}
        \rput(-1,-1.0){\rnode{C'}{$\rho_{1}$}}
        \rput(-1,0.4){\rnode{D'}{$\rho_{111}$}}
         \ncline[nodesep=3pt]{->}{A}{B}
        \ncline[nodesep=3pt]{->}{A}{C}
        \ncline[nodesep=3pt]{->}{B}{D}
        \ncline[nodesep=3pt]{->}{C}{D}
        \ncline[nodesep=3pt]{->}{A'}{B'}
        \ncline[nodesep=3pt]{<-}{D'}{C'}
        \ncline[nodesep=3pt]{->}{A}{A'}
        \ncline[nodesep=3pt]{->}{B}{B'}
         \ncline[nodesep=3pt, linecolor=gray]{->}{A}{D'} 
       \end{pspicture}}
    }
    \caption{$T_Y$-invariant $\vartheta'$-stable quiver representations}
 \label{fig:quiverrepsZ/2^3}
  \end{figure}
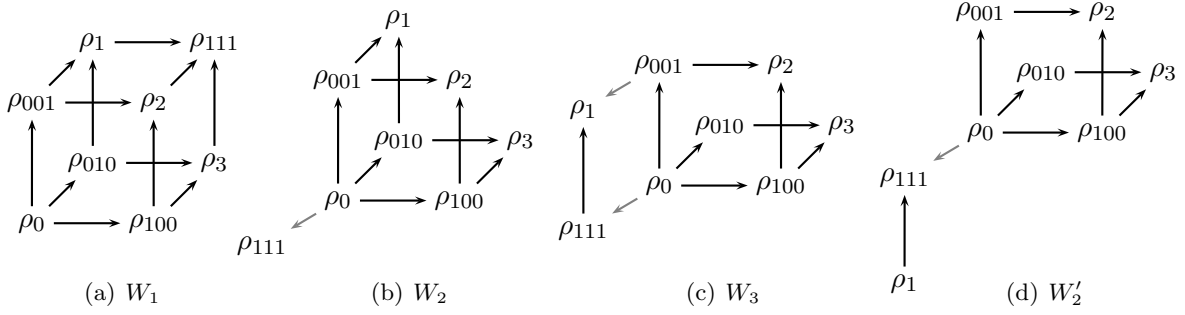
 Of the remaining eight torus-invariant \(G\)-clusters, four are of
   the type shown in Figure~\ref{fig:quiverrepsZ/2^3}(b), each
   obtained by cyclically permutting $x,y,z,w$, while the remaining
   four are of the type shown in
   Figure~\ref{fig:quiverrepsZ/2^3}(c). All 8 $G$-clusters in these
   latter two cases become unstable as $\vartheta\in C_0$ moves to the
   chamber $C_i$.  In addition, for $\theta\in C_i$, one new family
   of four $\theta$-stable $G$-constellations appears, one of which
   is shown in Figure~\ref{fig:quiverrepsZ/2^3}(d). An explicit
   deformation calculation shows that for each \(i = 1,2,3\), the four
   new \(G\)-constellations define the four torus-invariant points of
   \(Y_i\) corresponding to the four innermost simplices in the
   pictures that run down the centre of Figure~\ref{fig:toricZ2^3}.
   This shows that for each \(i = 1,2,3\), the birational map
   \(\ghilb(\mathbb{A}^4_\kk)=\mathcal{M}_\vartheta \rightarrow
   \mathcal{M}_{\theta}\) induced by variation of GIT quotient through
   the wall \(W_i\) coincides with the contraction \(Y\to Y_i\).
 \end{proof}

 Applying Theorem~\ref{thm:CrawIshii1} carefully leads to the following
 result:

 \begin{theorem}
   For \(i = 1,2,3\), there is a chamber \(C_i\subset
   \Wt(Q)\otimes_\ZZ \QQ\) such that \(Y_i =
   \mathcal{M}_{\theta}\) for \(\theta\in C_i\).  Moreover, in each
   case the universal sheaf \(\mathscr{U}_{\theta}\) on $Y_i\times
   \mathbb{A}^4_\kk$ defines an equivalence of derived categories
   $\Phi^{\mathscr{U}_{\theta}}\colon
   D^b\big(\!\coh(Y_\theta)\big)\rightarrow D^b(\modA)$.
 \end{theorem}
 \begin{proof}
   It remains to prove the second statement.  Write \(\tau_i\colon
   \mathcal{M}_{\theta_i} \rightarrow X\) for the crepant resolution
   with \(\theta_i \in C_i\) for \(i = 1,2,3\).  The fibre
   \(\tau_i^{-1}(\pi(0))\) over the point \(\pi(0)\in
   \mathbb{A}_\kk^4/G\) is isomorphic to \(\mathbb{P}^1\times
   \mathbb{P}^1\), so it satisfies the dimension condition required to
   apply Theorem~\ref{thm:CrawIshii2}.  This is not yet enough since
   the images under \(\pi\colon \mathbb{A}_\kk^4\rightarrow
   \mathbb{A}_\kk^4/G\) of all four coordinate axes are also singular
   and hence have been resolved in creating the resolution.
   Nevertheless, these singularities arise along a subvariety of
   dimension one, so the dimension of the fibre product over any such
   point \(\pi(x)\in \mathbb{A}^4_\kk/G\) is $2\cdot 3-1$, which
   equals \(n+1\) in this case. Theorem~\ref{thm:CrawIshii2} now applies.
 \end{proof}
 
 \begin{remark}
   Conjecture~\ref{conj:McKay} has also been established as an
   equivalence of derived categories for finite subgroups \(G\subset
   \Sp(n,\CC)\) by Kaledin--Bezrukavnikov~\cite{BK} and for finite
   abelian subgroups \(G\subset \SL(n,\CC)\) by
   Kawamata~\cite{Kawamata4}.
 \end{remark}

 \bibliography{toricquivers}
 \end{document}